\documentclass{amsart}
\usepackage{amssymb}
\usepackage{amsmath}

\newcommand{\map}[1]{\xrightarrow{#1}}
\newcommand{\mil}{\varprojlim}

\newcommand{\iso}{\cong}
\newcommand{\define}{\stackrel{\mathrm{def}}{=}}

\newcommand{\Hom}{\mathrm{Hom}}
\newcommand{\Aut}{\mathrm{Aut}}
\newcommand{\End}{\mathrm{End}}
\newcommand{\Spec}{\mathrm{Spec}}
\newcommand{\Spf}{\mathrm{Spf}}
\newcommand{\Q}{\mathbb Q}
\newcommand{\Z}{\mathbb Z}
\newcommand{\R}{\mathbb R}
\newcommand{\C}{\mathbb C}
\newcommand{\F}{\mathbb F}
\newcommand{\A}{\mathbb A}
\newcommand{\co}{\mathcal O}
\newcommand{\alg}{\mathrm{alg}}
\newcommand{\ord}{\mathrm{ord}}
\newcommand{\SL}{\mathrm{SL}}
\newcommand{\GL}{\mathrm{GL}}
\newcommand{\Sp}{\mathrm{Sp}}
\newcommand{\Sym}{\mathrm{Sym}}
\newcommand{\length}{\mathrm{length}}
\newcommand{\horizontal}{\mathrm{hor}}
\newcommand{\vertical}{\mathrm{ver}}
\newcommand{\chow}{\mathrm{CH}}

\input xy
\xyoption{all}

\begin{document}

\author{Benjamin Howard}
\title{Intersection theory on Shimura surfaces}

\email{howardbe@bc.edu}
\address{Department of Mathematics, Boston College, Chestnut Hill, MA, 02467}
\thanks{This research was supported in part by NSF grant DMS-0556174, and by a Sloan Foundation Research Fellowship.}

\begin{abstract}
Kudla has proposed a general program to relate arithmetic intersection multiplicities of special cycles on Shimura varieties to Fourier coefficients of Eisenstein series.  The lowest dimensional case, in which one intersects two codimension one cycles on the integral model of a  Shimura curve, has been completed by Kudla-Rapoport-Yang.  In the present paper we prove results in a higher dimensional setting.   On the integral model  of a Shimura surface we consider the intersection of a Shimura curve with a codimension two  cycle  of complex multiplication points, and relate the intersection to certain cycles classes constructed by Kudla-Rapoport-Yang.  As a corollary we deduce that our intersection multiplicities appear as Fourier coefficients of a Hilbert modular form of half-integral weight.
\end{abstract}


\maketitle

\theoremstyle{plain}
\newtheorem{Thm}{Theorem}[subsection]
\newtheorem{Prop}[Thm]{Proposition}
\newtheorem{Lem}[Thm]{Lemma}
\newtheorem{Cor}[Thm]{Corollary}
\newtheorem{Conj}[Thm]{Conjecture}
\newtheorem{BigThm}{Theorem}

\theoremstyle{definition}
\newtheorem{Def}[Thm]{Definition}
\newtheorem{Hyp}[BigThm]{Hypothesis}
\newtheorem{BigCor}[BigThm]{Corollary}

\theoremstyle{remark}
\newtheorem{Rem}[Thm]{Remark}

\renewcommand{\labelenumi}{(\alph{enumi})}
\renewcommand{\theHyp}{\Alph{Hyp}}
\renewcommand{\theBigThm}{\Alph{BigThm}}

\section{Introduction}

Suppose that $B_0$ is an indefinite quaternion division algebra over $\Q$.  Let $G_0$ be the algebraic group over $\Q$ whose functor of points is $G_0(A)=(B_0\otimes_\Q A)^\times$ for any $\Q$-algebra $A$, fix a maximal order $\co_{B_0}\subset B_0$, and  set  $U_0^\mathrm{max}=\widehat{\co}_{B_0}^\times$ and  
$$
\Gamma_0^\mathrm{max} = G_0(\Q)\cap U_0^\mathrm{max}.
$$
 After fixing an isomorphism $G_0(\R)\iso \GL_2(\R)$ the group $G_0(\R)$ acts on the complex manifold $X_0=\C\smallsetminus\R$, and the quotient
$$
\mathcal{M}_0(\C)= \Gamma_0^\mathrm{max} \backslash X_0
$$
is isomorphic to the complex points of a Shimura curve over $\Q$ which parametrizes abelian surfaces over $\Q$-schemes with an action of $\co_{B_0}$.  Extending the moduli problem over $\Spec(\Z)$ one obtains an algebraic (Deligne-Mumford) stack $\mathcal{M}_0$.     Let $\Sym_2(\Z)$ denote the $\Z$-module of symmetric $2\times 2$  matrices with entries in $\Z$ and set 
$$
\Sym_2(\Z)^\vee=\left\{  \left( \begin{matrix} a& \frac{b}{2} \\ \frac{b}{2} & c  \end{matrix} \right) \Big| \ a,b,c\in\Z   \right\}.
$$
 For every $T\in\Sym_2(\Z)^\vee$ Kudla-Rapoport-Yang \cite{KRY} have constructed an arithmetic cycle class
$$
\widehat{\mathcal{Z}}(T,\mathbf{v})\in  \widehat{\chow}^2_\R (\mathcal{M}_0)
$$
in the Gillet-Soul\'{e} arithmetic Chow group (with real coefficients  and modified for stacks, as in \cite[Chapter 2]{KRY}) of the arithmetic surface $\mathcal{M}_0$.  Here $\mathbf{v}\in M_2(\R)$ is  a symmetric positive definite parameter.   Letting $\tau=\mathbf{u}+i\mathbf{v} \in\mathfrak{h}_2$ denote the variable on the Siegel half-space of genus two, these classes have the remarkable property  that the generating series
$$
\widehat{\phi}_2(\tau)= \sum_{T\in\Sym_2(\Z)^\vee } \widehat{\deg}\ \widehat{\mathcal{Z}}(T,\mathbf{v}) \cdot q^T
$$
is a non-holomorphic Siegel modular form of weight $3/2$.  Here the isomorphism
$$
\widehat{\deg} :  \widehat{\chow}^2_\R(\mathcal{M}_0) \map{} \R
$$
is  the \emph{arithmetic degree}  of \cite[(2.4.10)]{KRY}. Pulling $\widehat{\phi}_2$ back to the product of two upper half-planes via the diagonal embedding $\mathfrak{h}_1\times\mathfrak{h}_1\map{}\mathfrak{h}_2$ results in a modular form of parallel weight $3/2$  for a congruence subgroup of $\SL_2(\Z)\times\SL_2(\Z)$, and Kudla-Rapoport-Yang express the Fourier coefficients of the pullback in terms of the arithmetic intersections of classes in $\widehat{\chow}_\R^1(\mathcal{M}_0)$.  More precisely, for each $t\in \Z$ and $v\in \R^+$ they define a class
$$
\widehat{\mathcal{Z}}(t,v)\in \widehat{\chow}_\R^1(\mathcal{M}_0)
$$
and prove that the pullback of $\widehat{\phi}_2$ by the diagonal embedding has Fourier expansion
\begin{equation}\label{KRY expansion}
\widehat{\phi}_2(\tau_1,\tau_2) = \sum_{t_1,t_2\in\Z}  \left\langle   \widehat{\mathcal{Z}}(t_1,v_1),  \widehat{\mathcal{Z}}(t_2,v_2)\right\rangle       q_1^{t_1}q_2^{t_2}
\end{equation}
where $\tau_j=u_j+iv_j$ is a variable in the upper half plane $\mathfrak{h}_1$,  $q_j=e^{2\pi i \tau_j}$, and the pairing $\langle\ ,\ \rangle$ is the Gillet-Soul\'{e} intersection pairing 
$$
\widehat{\chow}^1_\R(\mathcal{M}_0) \times \widehat{\chow}^1_\R(\mathcal{M}_0) \map{} \widehat{\chow}^2_\R(\mathcal{M}_0) \map{\widehat{\deg}} \R.
$$
The proof of (\ref{KRY expansion}) amounts to proving the decomposition \cite[Theorem C]{KRY}
\begin{equation}\label{KRY decomposition}
\left\langle   \widehat{\mathcal{Z}}(t_1,v_1),  \widehat{\mathcal{Z}}(t_2,v_2)\right\rangle   
= \sum_T \widehat{\deg}\ \widehat{\mathcal{Z}}(T,\mathbf{v}) 
\end{equation}
in which the sum is over all $T\in\Sym_2(\Z)^\vee$ of the form $T=\left(\begin{matrix} t_1 & * \\ * & t_2 \end{matrix}\right)$, and $\mathbf{v}$ is the diagonal matrix with diagonal entries $v_1,v_2$.

The question which motivates this article  is the following.  Given a real quadratic field $F\subset\R$ and a $\Z$-basis $\{\varpi_1,\varpi_2\}$ of $\co_F$  (which we now fix once and for all) one can define (\ref{siegel embedding}) a twisted embedding $\mathfrak{h}_1\times\mathfrak{h}_1\map{}\mathfrak{h}_2$ in such a way that the pullback of $\widehat{\phi}_2$ to $\mathfrak{h}_1\times\mathfrak{h}_1$ is a half-integral weight Hilbert modular form.  Can one interpret the Fourier coefficients of the twisted pullback  as arithmetic  intersection multiplicities of cycles on a Shimura variety?  To answer this, let
 $$
 B=B_0\otimes_\Q F \hspace{1cm} \co_B=\co_{B_0}\otimes_\Z\co_F.
 $$  
  Throughout the paper we assume the following equivalent conditions on $F$ and $B_0$:
\begin{Hyp}\label{Hyp:discriminant}\
\begin{enumerate}
\item $\mathrm{disc}(B_0)\cdot \co_F=\mathrm{disc}(B)$,
\item
$\co_B$ is a maximal order of $B$,
\item
every prime divisor of $\mathrm{disc}(B_0)$ splits in $F$.
\end{enumerate}
\end{Hyp}

 Let $G$ be the algebraic group over $\Q$  whose functor of points  is 
$$
G(A) = \{ g\in (B\otimes_\Q A)^\times    \mid \mathrm{Nm}(g)\in A^\times\}
$$
for any $\Q$-algebra $A$, where $\mathrm{Nm}$ denotes the reduced norm on $B$.  Define a maximal compact open subgroup $U^\mathrm{max}\subset G(\A_f)$ by
$$
U^{\mathrm{max}} =\{g\in \widehat{\co}_B^\times \mid \mathrm{Nm}(g)\in \widehat{\Z}^\times\}
$$
 and set
$$
\Gamma^\mathrm{max} = G(\Q)\cap U^\mathrm{max}.
$$
Let $\sigma$ be the nontrivial Galois automorphism of $F/\Q$ and identify 
\begin{equation}\label{F split}
F\otimes_\Q\R\iso \R\times\R
\end{equation}
using the map $x\otimes 1\mapsto (x,x^\sigma)$.   The earlier choice of isomorphism $G_0(\R)\iso \GL_2(\R)$ together with (\ref{F split}) determines an isomorphism
 $$
 G(\R)\iso \{ (x,y)\in \GL_2(\R)\times\GL_2(\R) \mid \det(x)=\det(y) \}
 $$
  in such a way that the subgroup $G_0(\R)\subset G(\R)$ is identified with the diagonal.  Let  $X\subset X_0\times X_0$ be subset of pairs whose imaginary parts are either both positive or both negative.  The compact complex manifold
  $$
  \mathcal{M}(\C)= \Gamma^\mathrm{max} \backslash X
  $$
is then identified with  the complex points of a Shimura surface  which parametrizes abelian fourfolds over $\Q$-schemes with an action of $\co_B$ (and some additional polarization data which we ignore in this introduction).  The obvious embedding $\mathcal{M}_0(\C)\map{}\mathcal{M}(\C)$ induced by the inclusion of $X_0$ into $X$ has a moduli theoretic meaning: an abelian surface $A_0$ with $\co_{B_0}$ action is taken to the abelian fourfold $A=A_0\otimes\co_F$ (Serre's tensor construction \cite[\S 7]{conrad04}) with its induced action of $\co_B=\co_{B_0}\otimes_\Z\co_F$.  Extending the moduli problems over $\Spec(\Z)$ we obtain a closed immersion $\mathcal{M}_0\map{}\mathcal{M}$ of an arithmetic surface into an arithmetic threefold.

On the other hand, the complex surface $\mathcal{M}(\C)$ comes equipped with a natural family of  cycles of dimension zero.  For each totally positive $\alpha\in\co_F$ we consider the finite set  $\mathcal{Y}(\alpha)(\C)$ of isomorphism classes of abelian fourfolds over $\C$ equipped with commuting actions of  $\co_B$ and $\co_F[\sqrt{-\alpha}]$.    There is an evident function $$\mathcal{Y}(\alpha)(\C)\map{} \mathcal{M}(\C)$$ which forgets the $\co_F[\sqrt{-\alpha}]$-action.  Extending the moduli problem across $\Spec(\Z)$ one obtains a finite morphism $\mathcal{Y}(\alpha)\map{}\mathcal{M}$ in which $\mathcal{Y}(\alpha)$ is an algebraic stack over $\Z$ of dimension at most two, and which has dimension one after restricting to  $\Z[\mathrm{disc}(B_0)^{-1}]$.  At a prime $p$ dividing $\mathrm{disc}(B_0)$ the stack $\mathcal{Y}(\alpha)_{/\Z_p}$ may (depending on $\alpha$) have vertical components of dimension two.   Our central object of study is a certain class
$$
\widehat{\mathcal{Y}}(\alpha,v) \in \widehat{\chow}^2(\mathcal{M})
$$
in the codimension two Gillet-Soul\'{e} arithmetic Chow group of $\mathcal{M}$ (with rational coefficients) whose construction is based on the moduli problem $\mathcal{Y}(\alpha)$.  Here $v\in F\otimes_\Q\R$ is an auxiliary totally positive parameter.  This class is obtained by first modifying the vertical components of $\mathcal{Y}(\alpha)$ at the prime divisors of $\mathrm{disc}(B_0)$ to obtain a cycle class in the codimension two Chow group $\chow^2_{\mathcal{Y}(\alpha)}(\mathcal{M})$ with support on $\mathcal{Y}(\alpha)$, and then augmenting this cycle with a Green current of the type constructed by Kudla \cite{kudla97} and Kudla-Rapoport-Yang \cite{KRY}.     It is the definition of the cycle class, and especially of the construction of the vertical components at primes dividing $\mathrm{disc}(B_0)$ carried out in \S \ref{S:bad components}, which is the primary original contribution of this work.

The closed immersion $\mathcal{M}_0\map{}\mathcal{M}$ induces  $\Q$-linear functional
\begin{equation*}
\widehat{\deg}_{\mathcal{M}_0} :  \widehat{\chow}^2(\mathcal{M})  \map{} \R
\end{equation*}
called the \emph{arithmetic degree along $\mathcal{M}_0$}.   Our main result, which appears in the  text as Theorem \ref{Thm:arithmetic decomposition}, is  a twisted form of the decomposition (\ref{KRY decomposition}).

\begin{BigThm}\label{BT}
Suppose that $\alpha\in \co_F$ and $v\in F\otimes_\Q\R$ are totally positive and that $F(\sqrt{-\alpha})/\Q$ is not biquadratic.  Then
\begin{equation*}
\widehat{\deg}_{\mathcal{M}_0} \widehat{\mathcal{Y}}(\alpha,v) =\sum_{T\in\Sigma(\alpha)} \widehat{\deg}\ \widehat{\mathcal{Z}}(T,\mathbf{v})
\end{equation*}
where $v=(v_1,v_2)\in \R\times \R$ and $\mathbf{v}$ are related by (\ref{bold v}).  The set $\Sigma(\alpha)$ appearing in the sum is defined as 
\begin{equation}\label{Sigma set}
\Sigma(\alpha)=\left\{  \left( \begin{matrix} a & \frac{b}{2} \\ \frac{b}{2} & c  \end{matrix}  \right) \in  \Sym_2(\Z)^\vee  \ \Big|  \alpha= a\varpi_1^2+b\varpi_1\varpi_2 +c\varpi_2^2   \right\}
\end{equation}
where $\{\varpi_1,\varpi_2\}$ is the $\Z$-basis of $\co_F$ fixed above, and used in the definition of the embedding $\mathfrak{h}_1 \times\mathfrak{h}_1 \map{}\mathfrak{h}_2$ of (\ref{siegel embedding}).
\end{BigThm}

The assumption that $F(\sqrt{-\alpha})/\Q$ is not biquadratic in the theorem is made to ensure that 
$$
(\mathcal{Y}(\alpha)\times_\mathcal{M} \mathcal{M}_0)_{/\Q}=\emptyset.
$$
I.e.~that $\mathcal{Y}(\alpha)$ and $\mathcal{M}_0$ do not meet in the generic fiber.  The theorem should remain true without this assumption; see  \cite{howardC} for results in this direction.  As a corollary of the theorem we deduce in \S \ref{SS:automorphic} the following partial generalization of (\ref{KRY expansion}).

\begin{BigCor}
The pullback of $\widehat{\phi}_2$ via the twisted embedding $\mathfrak{h}_1\times\mathfrak{h}_1\map{} \mathfrak{h}_2$  has a Fourier expansion of the form
$$
\widehat{\phi}_2(\tau_1,\tau_2) =\sum_{\alpha\in\co_F} c(\alpha,v) \cdot q^\alpha
$$
where $v=(v_1,v_2)$ is the imaginary part of $(\tau_1,\tau_2)$ and $q^\alpha=e^{2\pi i \tau_1\alpha} e^{2\pi i \tau_2\alpha^\sigma}$.   If $\alpha$ is totally positive and $F(\sqrt{-\alpha})/\Q$ is not biquadratic then  
$$
c(\alpha,v) =\widehat{\deg}_{\mathcal{M}_0}  \widehat{\mathcal{Y}}(\alpha,v).
$$ 
\end{BigCor}

Allowing ourselves the luxury of speculation, we  now explain how these results fit into the program of generalized Gross-Zagier style theorems proposed by Kudla \cite{kudla04b}.  Suppose first that one can extend the definition of the arithmetic class $\widehat{\mathcal{Y}}(\alpha,v)$ to all $\alpha\in\co_F$, as opposed to restricting only to totally positive $\alpha$,   in such a way that (passing  to arithmetic Chow groups with real coefficients)  the generating series
$$
\widehat{\theta}(\tau_1,\tau_2)=\sum_{\alpha\in\co_F} \widehat{\mathcal{Y}}(\alpha,v) \cdot q^\alpha \in \widehat{\chow}^2_\R(\mathcal{M}) [[q]]
$$
 is a nonholomorphic vector-valued Hilbert modular form of parallel weight $3/2$ for the congruence subgroup (\ref{congruence subgroup}).    If Theorem \ref{BT} were  extended to every $\alpha\in\co_F$ then (after extending the arithmetic degree along $\mathcal{M}_0$ to arithmetic Chow groups with real coefficients) we would have the equality of $q$-expansions
$$
\widehat{\deg}_{\mathcal{M}_0} \widehat{\theta}(\tau_1,\tau_2) = \widehat{\phi}_2(\tau_1,\tau_2).
$$
 Given another  Hilbert modular form $f$ of parallel weight $3/2$ for (\ref{congruence subgroup}), now with real coefficients, one could then define the \emph{arithmetic theta lift} of $f$ as the Petersson inner product 
$$
\widehat{\Theta}(f)\define \langle f(\tau_1,\tau_2), \widehat{\theta} (\tau_1,\tau_2)\rangle^{\mathrm{Pet}} \in \widehat{\chow}^2_\R(\mathcal{M}),
$$
and then
$$
\widehat{\deg}_{\mathcal{M}_0} \widehat{\Theta}(f) =  \langle f(\tau_1,\tau_2), \widehat{\deg}_{\mathcal{M}_0} \widehat{\theta} (\tau_1,\tau_2)\rangle^{\mathrm{Pet}} = \langle f(\tau_1,\tau_2), \widehat{\phi}_2(\tau_1,\tau_2)\rangle^\mathrm{Pet}.
$$
On the automorphic side, one may attach to $f$ and $B_0$ a Rankin $L$-function obtained by integrating $f$ against the twisted pullback of the genus two Eisenstein series  $\mathcal{E}_2(\tau,s,B_0)$ studied by Kudla-Rapoport-Yang (see \S \ref{SS:automorphic})
$$
L(f,s,B_0) = \langle f(\tau_1,\tau_2), \mathcal{E}_2(\tau_1,\tau_2,s,B_0)\rangle^\mathrm{Pet}.
$$
The Eisenstein series vanishes at $s=0$, and a fundamental result of Kudla-Rapoport-Yang \cite[Theorem B]{KRY}  is that
$$
\mathcal{E}'_2(\tau,0,B_0) = \widehat{\phi}_2(\tau).
$$
Pulling back this equality to $\mathfrak{h}_1\times\mathfrak{h}_1$, it follows that
$$
\langle f(\tau_1,\tau_2), \widehat{\phi}_2(\tau_1,\tau_2)\rangle^\mathrm{Pet} = L'(f,0,B_0) .
$$
Thus Theorem \ref{BT} may be viewed as a first step toward the Gross-Zagier stye result
$$
\widehat{\deg}_{\mathcal{M}_0} \widehat{\Theta}(f) =  L'(f,0,B_0).
$$
The $L$-function $L(f,s,B_0)$ is somewhat mysterious.  One would like to know, for example, if it admits an Euler product, and if so what form the Euler factors take.  We hope to address  this question in a subsequent article.

The author thanks Steve Kudla for suggesting this problem, and for several helpful conversations.

\subsection{Notation}

Throughout all of the paper the symbols  $F$, $\sigma$,  $G$, and $G_0$ have the same meanings as above,  $\mathfrak{D}$ denotes the different of $F/\Q$, Hypothesis \ref{Hyp:discriminant} is assumed, and the $\Z$-basis $\{\varpi_1,\varpi_2\}$ of $\co_F$ is fixed.    We choose the maximal order $\co_{B_0}\subset B_0$ to be stable under the main involution and, as on \cite[p.~127]{boutot-carayol}, choose an $s\in  \co_B$ such that $s^2=-\mathrm{disc}(B_0)$.  Define a positive involution of $B_0$
$$
b\mapsto b^*= s^{-1} \cdot b^\iota \cdot  s.
$$
 Extend $b\mapsto b^*$ to a positive involution of $B$, trivial on $F$.    Let $\A_f$ denote the ring of finite adeles of $\Q$. 

\medskip


\section{Arithmetic intersection theory}
\label{AIT}


Our basic references for stacks are  \cite{gillet84}, \cite{moret}, and \cite{vistoli}. 
By an \emph{algebraic stack} we always mean a Deligne-Mumford stack in the sense of \cite{moret}.   If $S$ is a scheme and $\mathcal{X}$ is an algebraic stack then we denote by  $\mathcal{X}(S)$ the fiber of $\mathcal{X}$ over $S$.  The goal of \S \ref{AIT} is to develop a rudimentary (and somewhat ad-hoc) extension to stacks of the Gillet-Soul\'e arithmetic intersection theory  \cite{gillet-soule90,soule92}.


\subsection{Chow groups}
\label{SS:chow}


Throughout all of \S \ref{SS:chow} we fix an algebraic stack  $\mathcal{M}$ separated and of finite type over a regular Noetherian scheme $S$.  We describe below the basics of the theory of Chow groups of such a stack.  If we assume that  $\mathcal{M}$ is a scheme then the theory is described in \cite{fulton} and \cite{soule92}.   If instead we assume that $S$ is the spectrum of a field then the theory is described in  \cite{gillet84} and \cite{vistoli}, but  for the very rudimentary results described below the methods extend easily to the case of $S$ regular and Noetherian.

An \emph{irreducible cycle} on $\mathcal{M}$ is a nonempty integral closed substack of $\mathcal{M}$.  
For any nonnegative $k\in\Z$ let $Z^k(\mathcal{M})$ be the free $\Q$-module generated by  the codimension $k$ irreducible cycles on $\mathcal{M}$.  Elements of $Z^k(\mathcal{M})$ are called \emph{cycles} of codimension $k$.   Define the group of \emph{rational equivalences} in codimension $k$
\begin{equation}\label{rational equivalence}
R^k(\mathcal{M})=\bigoplus_{\mathcal{D}} k(\mathcal{D})^\times
\end{equation}
where the direct sum is over all irreducible cycles $\mathcal{D}$ of codimension $k-1$ on $\mathcal{M}$, and $k(\mathcal{D})$ is the field of rational functions on $\mathcal{D}$ (i.e. the quotient field of  \cite[Definition 1.14]{vistoli}).   There is a $\Z$-module map 
$$
\partial: R^k(\mathcal{M})\map{} Z^k(\mathcal{M})
$$
which takes $f\in k(\mathcal{D})$ to  its Weil divisor viewed as a cycle on $\mathcal{M}$  (defined for schemes in \cite{fulton} or  \cite{soule92} and extended to algebraic stacks as in \cite[\S 4.4]{gillet84}).    Define the \emph{codimension $k$ Chow group} (with rational coefficients)
$$
\chow^k(\mathcal{M})=\mathrm{coker}\big(  R^k(\mathcal{M}) \otimes_\Z\Q \map{\partial} Z^k(\mathcal{M}) \big).
$$

There is also a notion of Chow group with support along a closed subscheme, or, slightly more generally, with support along a proper map.  Suppose we are given an algebraic stack $\mathcal{Y}$ over $S$ and a proper map $\phi :\mathcal{Y}\map{}\mathcal{M}$.   Let $Z^k_\mathcal{Y}(\mathcal{M})\subset Z^k(\mathcal{M})$ be the subspace generated by the irreducible cycles supported on the image of  $\mathcal{Y}$. Similarly let $R^k_\mathcal{Y}(\mathcal{M})\subset R^k(\mathcal{M})$ be defined exactly as in (\ref{rational equivalence}), but where the direct sum is over only those irreducible cycles $\mathcal{D}$ supported on the image of $\mathcal{Y}$.  The map $\partial$ defined above restricts to a map $\partial: R^k_\mathcal{Y}(\mathcal{M})\map{} Z^k_\mathcal{Y}(\mathcal{M})$, and the \emph{codimension $k$ Chow group with support along $\mathcal{Y}$} (again, with rational coefficients) is defined as
$$
\chow_\mathcal{Y}^k(\mathcal{M})=\mathrm{coker}\big( R_\mathcal{Y}^k(\mathcal{M})\otimes_\Z\Q \map{\partial} Z_\mathcal{Y}^k(\mathcal{M})  \big).
$$
Suppose that $\mathcal{Y}$ and $\mathcal{M}_0$ are algebraic stacks and that there  is a finite type \emph{flat} morphism $f:\mathcal{M}_0\map{}\mathcal{M}$ and a proper morphism $\mathcal{Y}\map{}\mathcal{M}$.   As in  \cite[Proposition 4.6(i)]{gillet84} or \cite[Proposition 3.7]{vistoli}  there is an induced  \emph{flat pullback} 
on Chow groups
$$
f^*:\chow^k_\mathcal{Y}(\mathcal{M})\map{} \chow^k_{\mathcal{Y}_0}(\mathcal{M}_0)
$$
in which  $\mathcal{Y}_0=\mathcal{Y}\times_\mathcal{M} \mathcal{M}_0$.

\begin{Def}\label{Def:T uniformization}
Let $T$ be a regular Noetherian scheme over $S$.  A \emph{$T$-uniformization} of $\mathcal{M}$ is an isomorphism $\mathcal{M}_{/T}\iso [H\backslash M]$ of stacks in which $M$ is a scheme over $T$ and $H$ is a finite group of automorphisms of $M$.
\end{Def}

\begin{Lem}\label{Lem:Galois descent}
Suppose we have a $T$-uniformization $\mathcal{M}_{/T}\iso [H\backslash M]$   and let $f:M\map{}\mathcal{M}_{/T}$ be the canonical morphism.    If  $\mathcal{Y}\map{}\mathcal{M}$ is any  proper morphism of algebraic stacks and $Y=\mathcal{Y}\times_{\mathcal{M}}  M$ 
then the flat pullback
\begin{equation}\label{galois iso}
f^*:\chow^k_{\mathcal{Y}_{/T}}(\mathcal{M}_{/T}) \map{} \chow^k_{Y}(M)^H
\end{equation}
is an isomorphism.
\end{Lem}

\begin{proof}
As in \cite[Proposition 4.6(iii)]{gillet84} or \cite[Proposition 3.7]{vistoli}  there is  a push-forward  homomorphism  $f_*:\chow^k_Y(M)\map{}\chow^k_{\mathcal{Y}_{/T}}(\mathcal{M}_{/T})$   which satisfies   $$(f_*\circ f^*)(\mathcal{C})=|H|\cdot\mathcal{C}$$ and  $$(f^*\circ f_*)(C)=\sum_{h\in H}h\cdot C.$$   Therefore the flat pullback (\ref{galois iso}) has inverse $|H|^{-1}\cdot f_*$. 
\end{proof}

Suppose $S=\Spec(R)$ with $R$ a discrete valuation ring.  Let $\eta$ and $s$ be the generic point and closed point of $S$, respectively, and assume that $\mathcal{M}$ is proper and flat over $S$.    An irreducible cycle $\mathcal{C}$ on $\mathcal{M}$ is \emph{horizontal} if it is flat over $S$, and is \emph{vertical} if it is supported on the special fiber $\mathcal{M}_{/k(s)}$.   Let 
$$
\chow^k_\vertical (\mathcal{M}) \define \chow^k_{\mathcal{M}_{/k(s)}}(\mathcal{M})
$$
denote the Chow group  with support along the special fiber  $\mathcal{M}_{/k(s)}\map{}\mathcal{M}$, and define $Z^k_\vertical(\mathcal{M})$ in the same way.  If  $\mathcal{Y}\map{}\mathcal{M}$ is any morphism of algebraic stacks there is a canonical homomorphism
$$
Z^k_{\mathcal{Y}} (\mathcal{M} )\map{}Z^k_\vertical (\mathcal{M} )\oplus Z^k_{\mathcal{Y}_{/k(\eta)}}(\mathcal{M}_{/k(\eta)})
$$
which takes an irreducible cycle $\mathcal{C}$ to $(\mathcal{C},0)$ if $\mathcal{C}$ is vertical, and to $(0,\mathcal{C}_{/k(\eta)})$ if $\mathcal{C}$ is horizontal.  If the image of $\mathcal{Y}_{/k(\eta)}\map{}\mathcal{M}_{/k(\eta)}$ has codimension at least $k$  (including possibly $\mathcal{Y}_{/k(\eta)}=\emptyset$) then  this homomorphism descends to a map 
\begin{equation}\label{support map}
\chow^k_{\mathcal{Y} } (\mathcal{M}   )\map{} \chow^k_\vertical(\mathcal{M}   )\oplus Z^k_{\mathcal{Y}_{/k(\eta)}}( \mathcal{M}_{/k(\eta)} )
\end{equation}
 as in \cite[Remark III.2.1]{soule92}.


\subsection{A little K-theory}
\label{SS:K theory}


For any algebraic stack $\mathcal{M}$ we let $\mathbf{K}_0(\mathcal{M})$ denote the Grothendieck group of the category of coherent   $\co_\mathcal{M}$-modules (denoted by $K'_0(\mathcal{M})$ in \cite{gillet-soule87} and \cite{soule92}).  That is, the free abelian group generated by coherent $\co_\mathcal{M}$-modules, modulo the subgroup generated by the relations $\mathcal{F}= \mathcal{F}_1+ \mathcal{F}_2$ whenever there is an exact sequence $0\map{}\mathcal{F}_1\map{}\mathcal{F}\map{}\mathcal{F}_2\map{}0$.  The class in $\mathbf{K}_0(\mathcal{M})$ of a coherent sheaf $\mathcal{F}$ will be denoted $[\mathcal{F}]$.  If  $\phi: \mathcal{Y}\map{}\mathcal{M}$ is a proper morphism of algebraic stacks we define  $\mathbf{K}_0^{\mathcal{Y}}(\mathcal{M})$ to be the Grothendieck group of   the category of coherent $\co_\mathcal{M}$-modules  supported on the image of $\mathcal{Y}$.   There is an obvious homomorphism $\mathbf{K}_0^{\mathcal{Y}}(\mathcal{M})\map{} \mathbf{K}_0(\mathcal{M})$ (which is typically  not  injective).   There is a \emph{higher direct image} map
\begin{equation}
\label{derived push-forward}
R\phi_*:\mathbf{K}_0(\mathcal{Y})\map{}\mathbf{K}_0^\mathcal{Y}(\mathcal{M})
\end{equation}
defined by 
$$
R\phi_*[\mathcal{F}] = \sum_{ k\ge 0 } (-1)^k [R^k \phi_*\mathcal{F}].
$$
If $\phi$ is a closed immersion or a finite map, so that $\phi_*$ is an exact functor from coherent $\co_\mathcal{Y}$-modules to coherent $\co_\mathcal{M}$-modules, then  $R\phi_*[\mathcal{F}] = [ \phi_*\mathcal{F}]$. 
If $\mathcal{Z}\map{}\mathcal{M}$ is a proper map with image contained in the image of $\mathcal{Y}\map{}\mathcal{M}$ then there is an evident change of support map $\mathbf{K}_0^\mathcal{Z}(\mathcal{M}) \map{}\mathbf{K}_0^\mathcal{Y}(\mathcal{M})$.  Define a decreasing filtration on $\mathbf{K}_0^\mathcal{Y}(\mathcal{M})$
$$
F^k \mathbf{K}_0^\mathcal{Y}(\mathcal{M})=\bigcup_\mathcal{Z} \mathrm{image}\big( \mathbf{K}_0^\mathcal{Z}(\mathcal{M})  \map{} \mathbf{K}_0^\mathcal{Y}(\mathcal{M}) \big)
$$
where the union is over the closed substacks $\mathcal{Z} \map{} \mathcal{M}$ of codimension $\ge k$ which are contained in the image of $\mathcal{Y}$.

Let $M$ be a regular scheme which is separated and of finite type over a regular Noetherian base $S$, and let   $Y\map{}M$ be a proper map of schemes.  Define $K_0^Y(M)$ to be the free group generated by quasi-isomorphism classes of finite complexes $\mathcal{P}^k\map{}\cdots\map{}\mathcal{P}^0$ of coherent locally free $\co_M$-modules whose homology sheaves $H^i(\mathcal{P}^\bullet)$ are supported on  the image of $Y$, modulo the subgroup generated by the relations
$\mathcal{P}^\bullet = \mathcal{Q}_1^\bullet + \mathcal{Q}_2^\bullet$ if there is an exact sequence  $0\map{}\mathcal{Q}_1^\bullet\map{}\mathcal{P}^\bullet\map{}\mathcal{Q}_2^\bullet\map{}0$.
The class in $K_0^Y(M)$ of a complex $\mathcal{P}^\bullet$ is denoted $[\mathcal{P}^\bullet]$.   If $\mathcal{F}$ is a coherent  sheaf on $M$ supported on the image of  $Y$ then by \cite[\S 1.9]{gillet-soule87}  there is a finite resolution  
$$
\mathcal{P}^\bullet\map{} \mathcal{F}\map{}0
$$ 
of $\mathcal{F}$ by coherent locally free $\co_M$-modules.  The rule  $[\mathcal{F}]\mapsto [\mathcal{P}^\bullet]$ then defines an isomorphism $\mathbf{K}_0^Y(M)\iso K_0^Y(M)$ with inverse $[\mathcal{P}^\bullet]\mapsto \sum_i (-1)^i H^i(\mathcal{P}^\bullet)$.   In particular the group $K_0^Y(M)$ inherits a filtration $F^k K_0^Y(M)$ from the filtration on $\mathbf{K}_0^Y(M)$ defined above.   Suppose that $i:M_0\map{}M$ is a finite type morphism of schemes with $M_0$ regular and separated over $S$, and set $Y_0=Y\times_M M_0$.  If  $i$ is flat then the functor $i^*$ from coherent $\co_M$-modules to coherent $\co_{M_0}$-modules is exact, and so the rule $[\mathcal{F}]\mapsto [i^*\mathcal{F}]$ is a well-defined homomorphism $i^*:\mathbf{K}_0^Y(M)\map{}\mathbf{K}_0^{Y_0}(M_0)$.  The groups $K_0^Y(M)$ admit pullback maps even when $i$ is not flat:  according to \cite[Theorem I.3(iii)]{soule92} the rule  $[\mathcal{P}^\bullet]\mapsto [i^*\mathcal{P}^\bullet]$ defines a homomorphism 
$$
i^*: F^k K_0^Y(M)\map{}F^k K_0^{Y_0}(M_0).
$$

Keeping $M$ and $Y$ as in the preceeding paragraph, if $Z\map{}M$ is an irreducible cycle of codimension $k$ supported on  $Y$ then we may define $\alpha(Z)$ to be the image of $[\co_Z]$ under 
$$
\mathbf{K}_0(Z)\map{}\mathbf{K}_0^Z(M)\map{} F^k \mathbf{K}_0^Y(M)\iso F^k K_0^Y(M).
$$
By a theorem of Gillet-Soul\'e, see \cite[Theorem 8.2]{gillet-soule87} or  \cite[\S 3.3]{soule92}, the rule $Z\mapsto \alpha(Z)$ defines an isomorphism
\begin{equation}\label{GS iso}
\chow^k_Y(M)\iso Gr^k K_0^Y(M)  \otimes_\Z\Q,
\end{equation}
where $Gr^k K_0^Y(M)=F^k K_0^Y(M)/F^{k+1}K_0^Y(M)$.  Combining this with the pullback on $K$-theory constructed above we obtain a pullback homomorphism on Chow groups
\begin{equation}
\label{basic local pullback}
i^*:\chow^k_Y(M)\map{}\chow^k_{Y_0}(M_0).
\end{equation}
If $i$ is flat then this homomorphism is the flat pullback constructed in \S \ref{SS:chow}.

\begin{Lem}\label{Lem:component decomp}
Let $S$ be a Noetherian scheme and let $\eta\in S$ be the generic point of an irreducible component $D$ of $S$.   We give $D$ its reduced subscheme structure and view $\co_D$ as a quotient sheaf of $\co_S$.  For any  coherent sheaf $\mathcal{F}$ on $S$ there is a closed subscheme $Z\map{}S$ not containing  $D$ such that
$$
[\mathcal{F}] - \length_{\co_{S,\eta}} (\mathcal{F}_\eta) \cdot [\co_D] 
 \in \mathrm{Im}\big(\mathbf{K}_0(Z)\map{}\mathbf{K}_0(S)\big).
 $$
\end{Lem}

\begin{proof}
Let $U$ be an open affine neighborhood of $\eta$.  By  \cite[Theorem 6.5(iii)]{matsumura} any coherent $\co_U$-module $\mathcal{A}$ which satisfies  $\mathcal{A}_\eta\not=0$ contains a subsheaf isomorphic to $\co_D|_U$.   Taking $\mathcal{A}=\mathcal{F}|_U$ and using induction on $\length_{\co_{S,\eta}} (\mathcal{F}_\eta)$ we find a coherent $\co_U$-module $\mathcal{G}$ with trivial stalk at $\eta$ and satisfying
 $$
 [\mathcal{F}|_U] = \length_{\co_{S,\eta}} (\mathcal{F}_\eta) \cdot [\co_D|_U] + [\mathcal{G}]
 $$ 
 in $\mathbf{K}_0(U)$.  But  $\mathcal{G}|_V=0$ for some  open neighborhood $V\subset U$ of $\eta$.  Thus $[\mathcal{F}]-\length_{\co_{S,\eta}} (\mathcal{F}_\eta) \cdot [\co_D] $ lies in the  kernel of restriction 
$
\mathbf{K}_0(S)\map{}\mathbf{K}_0(V),
$
and hence is contained in the image of $\mathbf{K}_0(S\smallsetminus V)\map{}\mathbf{K}_0(S)$ by the exact sequence of  \cite[Lemma I.3.2]{soule92}.
\end{proof}

\begin{Lem}
\label{Lem:sheaf decomp}
Let $S$ be a Noetherian scheme and let $\Omega$ denote any subset of the set of irreducible components of $S$.   There is an exact sequence
$$
\bigoplus_{Z  }  \mathbf{K}_0(Z)\map{}\mathbf{K}_0(S) \map{}\bigoplus_{ D \in \Omega} \mathbf{K}_0(\Spec(\co_{S,\eta}) ) \map{}0.
$$
in which the first sum is over all  closed subschemes  of $S$ having no components contained in $\Omega$,  and in the second sum $\eta$ is the generic point of the component $D$.
\end{Lem}

\begin{proof}
By induction it suffices to treat the case in which $\Omega$ consists of a single irreducible component of $S$ with generic point $\eta$.  As $\co_{S,\eta}$ is Artinian there is a canonical isomorphism 
$$\mathbf{K}_0(\Spec(\co_{S,\eta}))\iso \Z$$ given by $[\mathcal{F}_\eta]\mapsto \length_{\co_{S,\eta}}(\mathcal{F}_\eta)$, and so the claim follows easily from  Lemma \ref{Lem:component decomp}.
\end{proof}


\subsection{Arithmetic Chow groups}
\label{SS:chow2}


\begin{Def}
\label{Def:arithmetic stack}
An \emph{arithmetic stack} is an algebraic stack  $\mathcal{M}$ over $\Z$ satisfying the following properties:
\begin{enumerate}
\item
$\mathcal{M}$ is regular,
\item
the structure map $\mathcal{M}\map{}\Spec(\Z)$ is flat and  projective,
\item
for every prime $\ell$ one can find  a positive integer $N$ which is not divisible by $\ell$ and  a $\Z[1/N]$-uniformization $\mathcal{M}_{/\Z[1/N]}\iso [H\backslash M]$ in the sense of Definition \ref{Def:T uniformization}.
\end{enumerate}
\end{Def}

Throughout all of \S \ref{SS:chow2} we work with a fixed arithmetic stack $\mathcal{M}$, equidimensional  of dimension $d=\mathrm{dim}(\mathcal{M}),$
and choose a $\Q$-uniformization  $\mathcal{M}_{/\Q}\iso [H_\Q\backslash M_\Q]$.  Thus $M_\Q$ is a smooth projective variety over $\Q$ of dimension $d-1$.  There is a notion of a Green current \cite[\S III.1]{soule92} for a cycle in $Z^k(M_\Q)$, and  as a consequence of \cite[\S 4]{gillet84} there are canonical isomorphisms 
$$
Z^k(\mathcal{M}_{/\Q})\iso Z^k(M_\Q)^{H_\Q} \hspace{1cm} R^k(\mathcal{M}_{/\Q})\iso R^k(M_\Q)^{H_\Q}.
$$  
Hence we may define a  Green current for a cycle $\mathcal{C}_{\Q}\in Z^k(\mathcal{M}_{/\Q})$ to be an $H_\Q$-invariant Green current for the corresponding cycle  $C_\Q\in Z^k(M_\Q)^{H_\Q}$.  This definition does not depend of the choice of $\Q$-uniformization: if $\mathcal{M}_{/\Q}\iso [H_\Q'\backslash M_\Q']$ is another $\Q$-uniformization then one may form a third $\Q$-uniformization 
$$
\mathcal{M}_{/\Q}\iso [(H_\Q\times H_\Q') \backslash (M_\Q \times_{\mathcal{M}_{/\Q}} M_\Q') ]
$$
which allows one to identify $H_\Q$-invariant currents on $M_\Q$ with $H'_\Q$-invariant currents on $M_\Q'$, as both are identified with $H_\Q\times H'_\Q$-invariant currents on $M_\Q \times_{\mathcal{M}_{/\Q}} M'_\Q$.   Using the notation of \cite[\S III.1]{soule92}, if $\mathcal{C}\in Z^k(\mathcal{M})$ and $\Xi , \Xi' \in D^{k-1,k-1}(M_\Q)$ are two Green currents for $\mathcal{C}_{/\Q}$, we will say that $\Xi$ and $\Xi'$ are \emph{equivalent}  if there are $H_\Q$-invariant currents $u\in D^{k-2,k-1}(M_\Q)$ and $v\in D^{k-1,k-2}(M_\Q)$ such that $\Xi-\Xi'=\partial(u)  +\overline{\partial} (v)$.

Let $\widehat{Z}^k(\mathcal{M})$ be the $\Q$-vector space of  pairs $(\mathcal{C},\Xi)$ in which $\mathcal{C}\in Z^k(\mathcal{M})$ and $\Xi$ is an equivalence class of  Green currents for $\mathcal{C}_{/\Q}$.   For each irreducible cycle $D\in Z^{k-1}(M_\Q)$ and $f\in k(D)^\times$ there is an associated Green current,  constructed in  \cite[\S III.1]{soule92} and denoted $[-\log |f|^2]$, for the Weil divisor $\mathrm{div}(f)\in Z^k(M_\Q)$.   Thus every element of $R^k(M_\Q)$ has a canonically associated Green current.  Taking $H_\Q$-invariants yields a $\Z$-module map
$$
\widehat{\partial}:R^k(\mathcal{M}) \map{}\widehat{Z}^k(\mathcal{M}),
$$
and we define the \emph{codimension $k$ arithmetic Chow group} (with rational coefficients) by 
$$
\widehat{\chow}^k(\mathcal{M})=\mathrm{coker}\big( R^k(\mathcal{M})\otimes_\Z\Q\map{\widehat{ \partial}} \widehat{Z}^k(\mathcal{M}) \big).
$$
For every prime $p$ there is a homomorphism
\begin{equation}\label{vertical embedding}
\chow_\vertical^k(\mathcal{M}_{/\Z_p})\map{} \widehat{\chow}^k(\mathcal{M})
\end{equation}
defined by endowing  a cycle on $\mathcal{M}_{/\F_p}$  with the trivial Green current.

Fix a prime $p$.  For any irreducible cycle $\mathcal{C}$ on $\mathcal{M}_{/\Z_p}$ of codimension $d$ (necessarily vertical) define
$$
\deg_p(\mathcal{C}) =  \sum_{x\in \mathcal{C}(\F_p^\alg)} \frac{1}{|\Aut_{\mathcal{M}}(x)|} 
$$
where the sum is over all isomorphism classes of objects in the category $\mathcal{C}(\F_p^\alg)$.  By $\Aut_\mathcal{M}(x)$ we mean the automorphism group of $x$ in the category $\mathcal{M}(\F_p^\alg)$.  Extending the  definition of $\deg_p$ linearly to all of  $Z^d_\vertical(\mathcal{M}_{/\Z_p})$ there is an induced homomorphism 
$$
\deg_p: \chow^d_\vertical(\mathcal{M}_{/\Z_p}) \map{}\Q.
$$
Suppose  $(\mathcal{C},\Xi)\in\widehat{Z}^d(\mathcal{M})$.  By definition $\Xi$ is a $(d-1,d-1)$-current (up to equivalence) on the  compact complex manifold $M_\Q(\C)$ of dimension $d-1$.  Thus we may define 
$$
\deg_\infty(\Xi) = \frac{1}{2\cdot |H|} \int_{M_\Q(\C)} \Xi
$$
where the integral means evaluation of the current $\Xi$ at the constant function $1$.  The function on $\widehat{Z}^d(\mathcal{M})$ defined by
$$
\widehat{\deg} (\mathcal{C},\Xi)=  \deg_\infty(\Xi)  +  \sum_p \deg_p(\mathcal{C}_{/\Z_p}) \cdot  \log (p)  $$
then descends to a $\Q$-linear map,  the \emph{arithmetic degree},
\begin{equation*}
\widehat{\deg}:\widehat{\chow}^d(\mathcal{M})\map{}\R.
\end{equation*}
Indeed, it suffices to check that if $\mathcal{W}$ is an irreducible cycle of dimension one on $\mathcal{M}$ and $f$ is a rational function on $\mathcal{W}$ then the arithmetic degree of the pair $(\mathrm{div}(f), [-\log|f|^2])$ is $0$.  If $\mathcal{W}$ is a horizontal cycle then this is precisely the calculation of \cite[(2.1.11)]{KRY}.  If $\mathcal{W}$ is a vertical cycle supported in characteristic $p$ then one chooses a $\Z_p$-uniformization $[H\backslash M] \iso \mathcal{M}_{/\Z_p}$, sets $W=M\times_{\mathcal{M}} \mathcal{W}$, and uses the fact that principal Weil divisors on $W$ have degree $0$.

Now suppose $\mathcal{M}_0$ is another arithmetic stack, equidimensional of dimension $d_0=\mathrm{dim}(\mathcal{M}_0)$, and that  $i:\mathcal{M}_0\map{}\mathcal{M}$ is a closed immersion.  We will define a $\Q$-linear functional, the \emph{arithmetic degree along $\mathcal{M}_0$}
$$
\widehat{\mathrm{\deg}}_{\mathcal{M}_0} : \widehat{\chow}^{d_0}(\mathcal{M}) \map{}\R.
$$
Morally speaking, the arithmetic degree along $\mathcal{M}_0$ should be the composition 
$$
\widehat{\chow}^{d_0}(\mathcal{M})\map{i^*}\widehat{\chow}^{d_0}(\mathcal{M}_0)\map{\widehat{\deg}}\R
$$
of the pullback $i^*$ as constructed in  \cite[\S 4.4]{gillet-soule90} and \cite[\S III.3]{soule92} with the arithmetic degree defined above. This requires extending the theory of pullbacks in \cite{gillet-soule90} from schemes to arithemetic stacks, and is further complicated by a error in the construction of $i^*$ identified and corrected by Gubler \cite{gubler}.  Presumably this extension can be successfully done, but we will avoid the issue by borrowing a trick from Bruinier-Burgos-K\"uhn \cite{bruinier-burgos-kuhn}.  Suppose we have an integer $N$ and a $\Z[1/N]$-uniformization $\mathcal{M}_{/\Z[1/N]} \iso [H\backslash M]$.  Set $M_0=\mathcal{M}_0\times_{\mathcal{M}} M$.  There is an obvious homomorphism
$$
\delta_{M/\mathcal{M}}:\widehat{\chow}^{d_0}(\mathcal{M}) \map{} \widehat{\chow}^{d_0}(M)
$$
induced by pullback of cycles through the composition $M\map{}\mathcal{M}_{/\Z[1/N]}\map{}\mathcal{M}$, and the full force of the theories of  \cite{gillet-soule90} and \cite{BKK} (both of which include arithmetic Chow groups of schemes over $\Z[1/N]$) may be applied to the arithmetic Chow group $\widehat{\chow}^{d_0}(M)$.  In particular there is a homomorphism of $\Q$-vector spaces
$$
\Delta_N:\widehat{\chow}^{d_0}(\mathcal{M}) \map{}\widehat{\chow}^1(\Z[1/N])
$$
obtained as the composition
$$
\widehat{\chow}^{d_0}(\mathcal{M}) \map{\delta_{M/\mathcal{M}}} \widehat{\chow}^{d_0} (M)
\map{|H|^{-1}}\widehat{\chow}^{d_0} (M) \map{}\widehat{\chow}^{d_0}(M_0)\map{} \widehat{\chow}^1(\Z[1/N])
$$
in which the third arrow is the pullback of \cite[\S 4.4]{gillet-soule90} and the final arrow is the push-forward of  \cite[\S 3.6]{gillet-soule90}.  The map $\Delta_N$ depends on $N$ but not on the choice of $\Z[1/N]$-uniformization $M$.   If we define $\R_N$ to be the quotient of $\R$ be the additive subgroup of all rational linear combinations of $\{\log(p)\}$, as $p$ ranges over the prime divisors of $N$, then there is a canonical $\Q$-linear map
$$
\widehat{\deg}_{\Z[1/N]}: \widehat{\chow}^1(\Z[1/N]) \map{} \R_N
$$
obtained by imitating the construction of  the arithmetic degree $\widehat{\chow}^1(\Z) \map{} \R$.   As in \cite[\S 6.3]{bruinier-burgos-kuhn} there is a canonical isomorphism
$$
\R \iso \mil_{N\in\Z^+} \R_N,
$$
and our definition of an arithmetic stack guarantees that one can find a coterminal family $\mathcal{N}\subset\Z^+$ of this inverse system such that for every $N\in\mathcal{N}$ the stack $\mathcal{M}$ admits a $\Z[1/N]$-uniformization.  This allows us to define 
$$
\widehat{\deg}_{\mathcal{M}_0} : \widehat{\chow}^{d_0} (\mathcal{M}) \map{} \R 
$$
by
$$
\widehat{\deg}_{\mathcal{M}_0} = \mil_{N\in\mathcal{N}} \widehat{\deg}_{\Z[1/N]}\circ \Delta_N.
$$

Now suppose $\mathcal{Y}\map{}\mathcal{M}$ is a proper morphism of algebraic stacks such that  the image of  $\mathcal{Y}_{/\Q}\map{}\mathcal{M}_{/\Q}$ has codimension $k$.  If $p$ is a prime, by  a  \emph{local cycle datum at $p$} of codimension $k$ with support on $\mathcal{Y}$ we mean a triple $(\mathcal{C}_\Q, \Xi,\mathcal{C}_p)$ in which 
$$
\mathcal{C}_\Q\in Z^{k}_{\mathcal{Y}_{/\Q}} (\mathcal{M}_{/\Q}),
$$
$\Xi$ is an equivalence class of Green currents for  $\mathcal{C}_\Q$, and 
$$
\mathcal{C}_p\in  \chow^{k}_{\mathcal{Y}_{/\Z_p}}(\mathcal{M}_{/ \Z_p })
$$
is a cycle class which  maps to  $\mathcal{C}_\Q\times_\Q\Q_p$ under the homomorphism 
\begin{equation}
\label{generic fiber}
  \chow^{k}_{\mathcal{Y}_{/\Z_p}}(\mathcal{M}_{/ \Z_p }) \map{}Z^{k}_{\mathcal{Y}_{/\Q_p}} (\mathcal{M}_{/\Q_p})
\end{equation}
deduced from (\ref{support map}).  The image of $\mathcal{C}_p$ under
$$
  \chow^{k}_{\mathcal{Y}_{/\Z_p}}(\mathcal{M}_{/ \Z_p }) \map{}   \chow^{k}_\vertical(\mathcal{M}_{/ \Z_p }),
$$
denoted $\mathcal{C}_p^\vertical$, is the \emph{vertical component} of the local cycle datum.
By a \emph{global cycle datum}  of codimension $k$ with support on $\mathcal{Y}$ we mean a triple $(\mathcal{C}_\Q, \Xi,\mathcal{C}_\bullet)$ in which  $\mathcal{C}_\Q$ and $\Xi$ are as above and 
$\mathcal{C}_\bullet=\{\mathcal{C}_p\}$ 
is a collection indexed by the primes such that each $(\mathcal{C}_\Q, \Xi,\mathcal{C}_p)$ is a local cycle datum at $p$ of codimension $k$ with support along $\mathcal{Y}$.  We further require that $\mathcal{C}_p^\vertical=0$ for all but finitely many $p$.
Any global cycle datum $(\mathcal{C}_\Q, \Xi,\mathcal{C}_\bullet)$  determines an arithmetic cycle class 
\begin{equation}\label{datum to cycle}
 \widehat{\mathcal{C}}^\horizontal+\sum_p \widehat{\mathcal{C}}_p^\vertical
\in \widehat{\chow}^{k}(\mathcal{M})
\end{equation}
in which $\widehat{\mathcal{C}}_p^\vertical$ is the image of $\mathcal{C}_p^\vertical$ under the map (\ref{vertical embedding}) and $ \widehat{\mathcal{C}}^\horizontal$ is the arithmetic cycle consisting of the Zariski closure of $\mathcal{C}_\Q$ in $\mathcal{M}$ with its Green current $\Xi$.

 Set $\mathcal{Y}_0=\mathcal{Y}\times_\mathcal{M}\mathcal{M}_0$ and assume that the image of  $\mathcal{Y}_{0/\Q} \map{}\mathcal{M}_{0/\Q}$ again has codimension $k$. There is a pullback global cycle datum $ (i^*\mathcal{C}_\Q,i^*\Xi,i^*\mathcal{C}_\bullet)$ of codimension $k$ with support along $\mathcal{Y}_0$ defined as follows.  For every prime $p$  choose a $\Z_p$-uniformization $\mathcal{M}_{/\Z_p} \iso [H\backslash M]$ and define $\Z_p$-schemes
\begin{eqnarray}
Y &=& \mathcal{Y} \times_{\mathcal{M}} M  \label{cycle uniformization} \\
M_0 &=& \mathcal{M}_{0}\times_{\mathcal{M}} M  \nonumber \\
Y_0 &=& \mathcal{Y}_{0}\times_{\mathcal{M}_{0}} M_0. \nonumber
\end{eqnarray}
There is a pullback map
\begin{equation*}
i^*: \chow^{k}_{\mathcal{Y}_{/\Z_p}} (\mathcal{M}_{/\Z_p})\map{}\chow^{k}_{\mathcal{Y}_{0/\Z_p}}(\mathcal{M}_{0/\Z_p})
\end{equation*}
defined as the composition
$$
\chow^k_{\mathcal{Y}_{/\Z_p}} (\mathcal{M}_{/\Z_p}) \iso \chow^k_{Y} (M)^{H} \map{} \chow^k_{Y_0}(M_0)^{H}  \iso \chow^k_{\mathcal{Y}_{0/\Z_p}}(\mathcal{M}_{0/\Z_p})
$$
in which the middle arrow is (\ref{basic local pullback}).  In the exact same way one defines a pullback
$$
i^*: Z^k_{\mathcal{Y}_{/\Q}}(\mathcal{M}_{/\Q})  \iso \chow^k_{\mathcal{Y}_{/\Q}} (\mathcal{M}_{/\Q})\map{}\chow^k_{\mathcal{Y}_{0/\Q}}(\mathcal{M}_{0/\Q})\iso Z^k_{\mathcal{Y}_{0/\Q}}(\mathcal{M}_{0/\Q})
$$
using the hypotheses on the codimensions of $\mathcal{Y}_{/\Q}$ and $\mathcal{Y}_{0/\Q}$ for the isomorphisms.  This defines $i^*\mathcal{C}_\Q$ and $i^*\mathcal{C}_\bullet$. The  current $i^*\Xi$ is defined as in \cite[\S II.3.2 and \S II.3.3]{soule92}:  after replacing $\Xi$ by an equivalent Green current  we may assume that $\Xi$ is a Green form of logarithmic type on $M_\Q$, and the pullback current $i^*\Xi$ is then just the usual pullback of $\Xi$ to $M_{0\Q}=\mathcal{M}_{0}\times_{\mathcal{M}} M_\Q $ in the sense of differential forms.

As before let $\mathcal{M}_0$ be an arithmetic stack, equidimensional of dimension $d_0$, and suppose we have a closed immersion $i:\mathcal{M}_0\map{}\mathcal{M}$.  Let $\mathcal{Y}\map{}\mathcal{M}$ be a proper morphism of algebraic stacks such that the image of $\mathcal{Y}_{/\Q}\map{}\mathcal{M}_{/\Q}$ has codimension $d_0$ and assume that $\mathcal{Y}_{0/\Q}=\emptyset$, where we again set $\mathcal{Y}_0=\mathcal{Y}\times_{\mathcal{M}}\mathcal{M}_0$.  In particular the image of $\mathcal{Y}_{0/\Q}\map{}\mathcal{M}_{0/\Q}$ has codimension $d_0$.   If $(\mathcal{C}_\Q,\Xi,\mathcal{C}_\bullet)$ is a global cycle datum on $\mathcal{M}$ of codimension $d_0$ supported on $\mathcal{Y}$ then we have  defined above a pullback global cycle datum $(i^*\mathcal{C}_\Q,i^*\Xi,i^*\mathcal{C}_\bullet)$ on $\mathcal{M}_0$ of codimension $d_0$ supported on $\mathcal{Y}_0$.  Denoting by 
$$
\widehat{\mathcal{C}} \in \widehat{\chow}^{d_0}(\mathcal{M}) \hspace{1cm}
 \widehat{\mathcal{C}}_0 \in \widehat{\chow}^{d_0}(\mathcal{M}_0)
$$
the arithmetic cycle classes corresponding to these cycle data one may easily check that
$$
\widehat{\mathrm{deg}}_{\mathcal{M}_0} \widehat{\mathcal{C}} = \widehat{\mathrm{deg}}\ \widehat{\mathcal{C}}_0.
$$
In other words
\begin{equation}\label{pullback degree}
\widehat{\mathrm{deg}}_{\mathcal{M}_0} (\mathcal{C}_\Q,\Xi,\mathcal{C}_\bullet) = \widehat{\mathrm{deg}}\ (i^*\mathcal{C}_\Q,i^*\Xi,i^*\mathcal{C}_\bullet) .
\end{equation}


\section{A Shimura surface and its special cycles}
\label{S:Hilbert-Blumenthal} 


We take   \cite[Chapter 6]{mumford65} as our basic reference for abelian schemes.  We will be dealing with integral models of the Shimura varieties associated to the algebraic groups $G_0$ and $G$ of the introduction.  These have been dealt with thoroughly in the literature: for the Shimura curve associated to $G_0$ we point out the articles of Buzzard \cite{buzzard} and Boutot-Carayol \cite{boutot-carayol}, and for the Shimura surface associated to $G$ we point  out the volume \cite{breen-labesse}, as well as the articles of Milne \cite{milne79}, Kottwitz \cite{kottwitz92}, Rapoport-Zink \cite{rapoport-zink82,rapoport96}, Boutot-Zink \cite{boutot-zink}, Kudla-Rapoport \cite{kudla99}, and Hida \cite{hida04}.  These Shimura surfaces are closely related to classical Hilbert modular surfaces, whose integral models are dealt with in work of Rapoport \cite{rapoport78}, Deligne-Pappas \cite{deligne-pappas}, Pappas \cite{pappas95}, Stamm \cite{stamm}, and Vollaard \cite{vollaard}.  General references on Hilbert modular varieties include the books of Goren \cite{goren}, Hida \cite{hida06}, and van der Geer \cite{vanderGeer}.   The theory of Shimura varieties, from Deligne's points of view, can be found in works of Milne \cite{milne90, milne05}.


\subsection{Moduli problems}
\label{SS:moduli}


 Let $S$ be a scheme.  By a \emph{QM abelian fourfold} over $S$ we mean a pair $\mathbf{A}=(A,i)$ in which $A$ is an abelian scheme over $S$ of relative dimension four, and $i:\co_{B}\map{}\End(A)$ is a ring homomorphism taking $1\mapsto 1$ and satisfying the    \emph{Kottwitz condition} (see for example \cite[\S 5]{kottwitz92}, \cite[\S 7.1]{hida04}, or \cite{vollaard}), everywhere locally on $S$.  This last condition means that for every  $s\in S$  there is an open affine neighborhood $U$  over which $\mathrm{Lie}(A)$ is free $\co_U$-module satisfying  the equality of polynomials in $\co_U[x]$
  $$
 \mathrm{char}_{\co_U}(i(b);\mathrm{Lie}(A)) = f_b(x) \cdot f_b(x)^\sigma
 $$
for every $b\in\co_B$.  Here $$f_b(x)=(x-b)(x-b^\iota)\in\co_F[x]$$ is the reduced characteristic polynomial of $b$.   Define  $\End(\mathbf{A})$ to be the $\co_F$-algebra of endomorphisms of $A$ which commute with the action of $\co_B$, and set $\End^0(\mathbf{A})=\End(\mathbf{A})\otimes_{\co_F} F$.   For a good theory of moduli of QM abelian fourfolds we must introduce polarization data as well.  By a $\mathfrak{D}^{-1}$-\emph{polarized} QM abelian fourfold we mean a pair $(\mathbf{A},\lambda)$ in which $\mathbf{A}=(A,i)$ is a QM abelian fourfold and $\lambda:A\map{}A^\vee$ is a polarization satisfying
   \begin{enumerate}
 \item $ \lambda \circ  i(b^*) = i(b)^\vee\circ \lambda$ for every $b\in\co_B$
 \item the kernel of $\lambda$ is $A[\mathfrak{D}]$.
 \end{enumerate}
  Define $\mathcal{M}$ to be the category, fibered in groupoids over the category of schemes, whose objects are $\mathfrak{D}^{-1}$-polarized QM abelian fourfolds. A morphism from $(\mathbf{A}',\lambda')$ to $(\mathbf{A},\lambda)$, defined over schemes $S'$ and $S$, respectively, in the category $\mathcal{M}$ is a commutative diagram
\begin{equation*}
\xymatrix{
{ A' } \ar[r]\ar[d] & { A  }\ar[d] \\
S'\ar[r] & S
}
\end{equation*}
such that the induced map $A'\map{}A\times_S S'$ is an isomorphism of abelian schemes over $S$ respecting the action of $\co_{B}$ and identifying the polarizations $\lambda$ and $\lambda'$.  The category $\mathcal{M}$ is an arithmetic stack (in the sense of Definition \ref{Def:arithmetic stack}) of dimension three with geometrically connected fibers.   Furthermore $\mathcal{M}$  is smooth over $\Z[1/D]$, where $D=\mathrm{disc}(F)\cdot \mathrm{disc}(B_0)$.   The existence of a smooth projective models for  $\mathcal{M}$ over $\Z_p$ for primes  $p\nmid D$ is explained in \cite{kottwitz92,milne79}, and in much greater detail in Chapters 6-7 of \cite{hida04}.  When $p\mid \mathrm{disc}(B_0)$ the regularity of the integral model follows from the Cerednik-Drinfeld uniformization  described in \cite{boutot-carayol} and \cite{rapoport96}; see \S \ref{SS:CD}.  The higher dimensional case needed here is dealt with in \cite{boutot-zink} and \cite{rapoport96}.  Finally, for primes dividing $\mathrm{disc}(F)$ one proves the existence of regular projective models using the methods of Deligne-Pappas \cite{deligne-pappas}, who treat the case of classical Hilbert-Blumenthal moduli (e.g., the case $B\iso M_2(F)$, which we exclude).  Deligne and Pappas use a different type of polarization data in the statement of their moduli problem, and do not impose the Kottwitz condition.  The equivalence of moduli problems of the type defined above and those considered by Deligne and Pappas is explained in  work of Vollaard \cite{vollaard}.

By a \emph{QM abelian surface} over a scheme $S$ we mean a pair $\mathbf{A}_0=(A_0,i_0)$ in which $A_0$ is an abelian scheme over $S$ of relative dimension two and $i_0:\co_{B_0}\map{}\End(A_0)$ is a ring homomorphism satisfying $1\mapsto 1$.  Again we require that $\mathbf{A}_0$ satisfies the Kottwitz  condition every locally on $S$: for any $b_0\in\co_{B_0}$ the equality of polynomials in $\co_U[x]$ 
$$
\mathrm{char}_{\co_U}(i(b_0),\mathrm{Lie}(A_{0})) = (x-b_0)(x-b_0^\iota)
$$  
holds over any sufficiently small open affine neighborhood  $U$ of any point $s\in S$.  Define $\End(\mathbf{A}_0)$ and $\End^0(\mathbf{A}_0)$ as above.  By a \emph{principally polarized} QM abelian surface we mean a pair $(\mathbf{A}_0,\lambda_0)$ in which $\mathbf{A}_0$ is a QM abelian surface and $\lambda_0:A_0\map{}A_0^\vee$ is a  polarization satisfying  
\begin{enumerate}
\item  $\lambda \circ  i_0(b^*) = i_0(b)^\vee\circ \lambda_0$ for every $b\in \co_{B_0}$.
\item $\lambda_0$ is an isomorphism.
\end{enumerate}
We remark that in fact every QM abelian surface admits a unique such $\lambda_0$.  This follows by combining the argument  used in \cite{buzzard} to prove the claim when $\mathrm{disc}(B_0)$ is invertible on the base, with \cite[Proposition III.3.3]{boutot-carayol}.  Let $\mathcal{M}_0$ be the category, fibered in groupoids over the category of schemes, whose objects are  principally polarized QM abelian surfaces $(\mathbf{A}_{0},\lambda_0)$ over schemes.  Morphisms are defined exactly as in  $\mathcal{M}$.  The category $\mathcal{M}_0$ is then an arithmetic stack of dimension two with geometrically connected fibers, and is smooth over $\Z[\mathrm{disc}(B_0)^{-1}]$.  See \cite{boutot-carayol,buzzard} and the references therein for details.

Let $\mathbf{A}_0=(A_0,i_0)$ be a QM abelian surface over a scheme $S$ with a polarization $\lambda_0$ as above.  The abelian fourfold $A_0\otimes \co_F\iso \Hom_\Z(\mathfrak{D}^{-1},A_0)$ (see \cite[\S 7]{conrad04} for background on Serre's tensor construction) is then equipped with  commuting actions of $\co_{B_0}$ and $\co_F$, and hence with an action $i:\co_B\map{}\End(A_0\otimes \co_F)$.  To be very concrete, we may define $A_0\otimes\co_F=A_0\times A_0$ with $\co_{B_0}$ acting diagonally and $\co_F$ acting through the  ring homomorphism $\kappa:\co_F\map{}M_2(\Z)$ determined by the $\Z$-basis $\{\varpi_1,\varpi_2\}$ of $\co_F$.   We obtain a QM abelian fourfold
$$
\mathbf{A}_0\otimes\co_F= (A_0\otimes \co_F,i)
$$
and an isomorphism of $\co_F$-algebras
\begin{equation}\label{endo base change}
\End(\mathbf{A}_0\otimes\co_F) \iso \End(\mathbf{A}_0)\otimes_\Z \co_F.
\end{equation}
The identification $A_0\otimes\co_F=A_0\times A_0$ determines an isomorphism   
$$(A_0\otimes\co_F)^\vee \iso A_0^\vee\times A_0^\vee$$ in which the action of  $\co_{B_0}$ on the right hand side  is again diagonal, but the action of $\co_F$ is through the \emph{transpose} of the homomorphism $\kappa$.  This transpose ${}^t\kappa:\co_F\map{}M_2(\Z)$ is the embedding determined by the action of $\co_F$ on $\mathfrak{D}^{-1}$ with respect to the dual basis $\{\varpi_1^\vee,\varpi_2^\vee\}$ relative to the trace form on $F$, and allows us to identify 
$$
(A_0\otimes\co_F)^\vee\iso  A_0^\vee\times A_0^\vee\iso A_0^\vee \otimes\mathfrak{D}^{-1}.
$$
If denote by   $\lambda_0\otimes\co_F$  the isogeny
$$
A_0\otimes \co_F\map{\lambda_0\otimes 1 }  A_0^\vee\otimes \mathfrak{D}^{-1} \iso  (A_0\otimes \co_F)^\vee
$$
then the pair $(\mathbf{A}_0\otimes\co_F,\lambda_0\otimes\co_F)$ is  a $\mathfrak{D}^{-1}$-polarized QM abelian fourfold.  We  check  that $\lambda_0\otimes\co_F$ is a polarization: 
let $\Delta$ be the matrix of the trace form on $\co_F$ relative to the basis $\{\varpi_1,\varpi_2\}$.  Identifying $A_0\otimes\co_F \iso A_0\times A_0$ and $(A_0\otimes\co_F)^\vee \iso A_0^\vee \times A_0^\vee$ as above, the isogeny $\lambda_0\otimes\co_F$  is identified with 
\begin{equation}\label{polarized}
 A_0 \times  A_0  \map{\lambda_0\times\lambda_0}  A_0^\vee\times A_0^\vee \map{\Delta}  A_0^\vee  \times A_0^\vee .
\end{equation}
As $\Delta$ is positive definite there is some positive integer multiple $m\Delta={}^tE \Delta' E$ with $\Delta',E \in M_2(\Z)$ and $\Delta'$ diagonal with diagonal entries $d_1,d_2>0$.  The pullback of the polarization
$$
 A_0\times A_0 \map{d_1\lambda_0\times d_2\lambda_0} A_0^\vee\times A_0^\vee
$$
by the isogeny
$$
A_0  \times A_0  \map{E} A_0 \times A_0
$$
is then $m$ times the isogeny (\ref{polarized}).  As $m$ times the isogeny (\ref{polarized}) is a polarization, so is  (\ref{polarized}).    We now have a functor  $$i:\mathcal{M}_0\map{}\mathcal{M}$$ defined by 
$$
(\mathbf{A}_0,\lambda_0)\mapsto  (\mathbf{A}_0 ,\lambda_0 ) \otimes \co_F = (\mathbf{A}_0\otimes\co_F,\lambda_0\otimes \co_F).
$$ 
This functor induces a proper morphism of algebraic stacks.   Combining the properness with the following lemma (which implies that the morphism $i$ is injective on geometric points) shows that $i:\mathcal{M}_0\map{}\mathcal{M}$ is a closed immersion.

\begin{Lem}
For any principally polarized QM abelian surfaces $(\mathbf{A}_0,\lambda_0)$ and $(\mathbf{A}_0',\lambda_0')$ over a common base scheme $S$ the natural function
$$
\mathrm{Isom}_{ \mathcal{M}_0(S) }\big( (\mathbf{A}_0,\lambda_0), (\mathbf{A}'_0,\lambda'_0) \big) \map{}
\mathrm{Isom}_{\mathcal{M}(S)} \big( (\mathbf{A}_0,\lambda_0) \otimes\co_F , (\mathbf{A}'_0,\lambda'_0) \otimes\co_F \big)
$$
is a bijection. 
\end{Lem}

\begin{proof}
Abbreviate $\mathbf{A}=\mathbf{A}_0\otimes\co_F$ and $\lambda=\lambda_0\otimes\co_F$ and similarly for $\mathbf{A}_0'$ and $\lambda_0'$. The isomorphism of $\co_F$-modules
$$
\Hom (\mathbf{A}_0 , \mathbf{A}'_0) \otimes\co_F \iso 
\Hom (\mathbf{A} , \mathbf{A}' )
$$
proves the injectivity of the function in question.  For surjectivity, suppose we start with an isomorphism $f:  (\mathbf{A},\lambda)  \iso (\mathbf{A}',\lambda') $  in $\mathcal{M}(S)$.  The condition that $f$ identifies $\lambda$ with $\lambda'$ is equivalent to the commutativity of the diagram
$$
\xymatrix{
{ \mathbf{A}_0\times\mathbf{A}_0  } \ar[r] \ar[d]_{  \Delta\circ \lambda_0  } & \mathbf{A} \ar[r]^f \ar[d]^{\lambda} & \mathbf{A}' \ar[r] \ar[d]^{\lambda'} & \mathbf{A}_0'\times \mathbf{A}_0' \ar[d]^{  \Delta\circ\lambda_0' } \\
 { \mathbf{A}_0^\vee \times\mathbf{A}_0^\vee  }  & \mathbf{A}^\vee  \ar[l] & \mathbf{A}'^\vee \ar[l]^{f^\vee}  & \mathbf{A}_0'^\vee \times \mathbf{A}_0'^\vee \ar[l]
}
$$
in which all horizontal arrows are isomorphisms.  The composition of arrows along the top row is given by some matrix 
$$
\Phi\in \Hom(\mathbf{A}_0,\mathbf{A}_0') \otimes M_2(\Z)
$$ 
which lies in the $\Z$-submodule $ \Hom(\mathbf{A}_0,\mathbf{A}_0') \otimes  \kappa(\co_F).$ The composition of arrows in the bottom row, from right to left, is then given by the transpose dual
$$
{}^t\Phi^* \in \Hom(\mathbf{A}_0',\mathbf{A}_0) \otimes M_2(\Z)
$$ 
where  $\Phi^*$  denotes the entry-by-entry dual of the matrix $\Phi$.  Using the relation $\Delta \cdot \kappa(x)={}^t\kappa(x) \cdot \Delta$ for all $x\in\co_F$ the commutativity of the above diagram is equivalent to
\begin{equation}\label{polaris}
 \left(\begin{matrix} \lambda_0 & \\ & \lambda_0 \end{matrix}\right) = \Phi^* \cdot  \left(\begin{matrix} \lambda_0' & \\ & \lambda_0' \end{matrix}\right) \cdot  \Phi.
\end{equation}

We now choose a $\rho\in \GL_2(\Q)$ such that 
$$
\rho\kappa(F)\rho^{-1} = \left\{ \left( \begin{matrix} a & bD \\  b & a \end{matrix} \right) \mid a,b\in\Q
\right\}
$$
with $D>1$ a squarefree integer.  Suppose first that $D\equiv 1\pmod{4}$.  Then there are $f_1,f_2\in \Hom(\mathbf{A}_0,\mathbf{A}_0')$ for which
$$
\rho \Phi \rho^{-1} = \frac{1}{2} \left(\begin{matrix} 2f_1+f_2 & f_2D \\ f_2 & 2f_1+f_2  \end{matrix}\right).
$$
Conjugating both sides of  (\ref{polaris}) by $\rho$ and comparing the upper left entries on each side yields
\begin{eqnarray*}
4\lambda_0 &=& (2f_1+f_2)^\vee\circ \lambda_0' \circ (2f_1+f_2) + f_2^\vee\circ \lambda_0' \circ f_2 D \\
&=& \deg(2f_1+f_2) \lambda_0 + \deg(f_2)D\lambda_0
\end{eqnarray*}
where the second equality uses the fact that the $\Z$-module of symmetric homomorphisms $\mathbf{A}_0\map{}\mathbf{A}_0^\vee$ satisfying $\lambda_0\circ i_0(b^*) =i_0(b)^\vee\circ \lambda_0$ for all $b\in\co_{B_0}$ is free of rank one and generated by $\lambda_0$, with the polarizations corresponding to positive multiples of $\lambda_0$.  From this equality it is clear that $f_2=0$, and hence 
$$
\Phi =  \left(\begin{matrix} f_1  &   \\  & f_1  \end{matrix}\right)
$$
as desired.  The case $D \equiv 2,3 \pmod{4}$ is similar.
\end{proof}

\begin{Lem}
\label{Lem:endo algebra}
 If $\mathbf{A}_0$ is a  QM abelian surface over a connected scheme $S$ then $\End^0(\mathbf{A}_0)$ is either $\Q$,  a quadratic imaginary field, or a definite quaternion algebra over $\Q$. If $\mathbf{A}$ is a  QM abelian fourfold over $S$ then $\End^0(\mathbf{A})$ is either $F$,  a totally imaginary quadratic extension of $F$, or a totally definite quaternion algebra over $F$.
 \end{Lem}

\begin{proof}
 Fix a geometric point $s\map{} S$, and let  $\mathbf{A}=(A,i)$ be either a QM abelian  fourfold over $S$. By \cite[Corollary 6.2]{mumford65} the reduction map
$$
\End^0(\mathbf{A})\map{}\End^0(\mathbf{A}_{/k(s)})
$$
is injective, and so it suffices to treat the case in which $S=\Spec(k)$ with $k$ a field.  One then further reduces to the case of $k$ of finite transcendence degree over its prime subfield, and by considering the N\'eron model of $A$ over a valuation ring inside of $k$ and again invoking the injectivity of the reduction map on endomorphism algebras, one is reduced to the case of $k$ finite over its prime subfield.  If $\mathrm{char}(k)=0$ then $A(\C)$ is isogenous to  $(B\otimes_\Q\R)/\co_{B}$ for some complex structure on $B\otimes_\Q\R$, and from this one can show that $\End(\mathbf{A})$ is either $F$ or a quadratic imaginary extension of $F$.  If $\mathrm{char}(k)\not=0$ then \cite[Proposition 5.2]{milne79} shows that either $A$ is isogenous to the fourth power of a supersingular elliptic curve, or $A$ has no isogeny factor isomorphic to a supersingular elliptic curve.
In the first case  $\End^0(\mathbf{A})$ is isomorphic to the quaternion algebra $\overline{B}$ over $F$ satisfying
$$
[B]+[\overline{B}]= [H\otimes_\Q F]
$$
in the Brauer group of $F$, where $H$ is the rational quaternion algebra of discriminant $p$.  In the second case $\End^0(\mathbf{A})$ is a quadratic imaginary extension of $F$.  Note that in the statement of \cite[Proposition 5.2]{milne79} it is assumed that $\mathrm{char}(k)$ is prime to both $\mathrm{disc}(F)$ and to $\mathrm{disc}(B_0)$.  This hypothesis is not used in the proof until the sentence ``It splits $B$ because..." and so has no bearing on the results cited above.  The case of a QM abelian surface is identical (in this case one may also invoke \cite[Proposition III.2]{boutot-carayol} for primes dividing $\mathrm{disc}(B_0)$).
\end{proof}

Fix a finite set of rational primes $\Sigma$ and let $S$ be a scheme such that 
$$
\ell\not\in\Sigma\implies \ell^{-1}\in\co_S.
$$ 
Let $U_0\subset U_0^\mathrm{max}$ be a compact open subgroup which factors as $U_0=\prod_\ell U_{0,\ell}$ with $U_{0,\ell}=U_0^\mathrm{max}$ for every $\ell\in\Sigma$.  Write
$$
U_0^{\mathrm{max},\Sigma}=\prod_{\ell\not\in\Sigma} U_{0,\ell}^\mathrm{max}
\hspace{1cm}
U_0^\Sigma = \prod_{\ell \not\in\Sigma} U_{0,\ell}.
$$
Let $\Lambda_0=\co_{B_0}$ viewed as a left $\co_{B_0}$-module, and  define a perfect alternating  bilinear form  $\psi_0:   \Lambda_0\times \Lambda_0 \map{}\Z$ as on \cite[p.~130]{boutot-carayol} by 
$$
\psi_0(x,y)=\frac{1}{\mathrm{disc}(B_0)} \mathrm{Tr}(xsy^*)
$$  
where $\mathrm{Tr}$ is the reduced trace on $B_0$ and $s\in \co_{B_0}$ is the element fixed in the introduction.  Extend $\psi_0$ to an alternating form 
\begin{equation}
\label{polarization pairing}
\psi_0:\widehat{\Lambda}_0^\Sigma\times\widehat{\Lambda}_0^\Sigma\map{} \widehat{\Z}^\Sigma
\end{equation}
on the restricted topological product
$$
\widehat{\Lambda}_0^\Sigma= \prod_{\ell\not\in\Sigma} ( \Lambda_0 \otimes_\Z\Z_\ell ).
$$
We now define the notion of a $U_0$ level structure on a  principally polarized QM abelian surface $(\mathbf{A}_0,\lambda_0)$ defined over $S$.  Fix a geometric point $s\map{}S$ and let 
$$
\mathrm{Ta}^\Sigma(A_0) = \prod_{\ell\not\in\Sigma} \mathrm{Ta}_\ell(A_{0/k(s) })
$$
be the prime-to-$\Sigma$ adelic Tate module of $A_0$, equipped with the action of $(\widehat{\co}_{B_0})^\Sigma$ defined by $i_0$ as well as the action of the \'etale fundamental group $\pi_1(S,s)$.    A $U_0$ level structure on $\mathbf{A}_0$ is then a $U_0$ equivalence class of  $(\widehat{\co}_{B_0})^\Sigma$-module isomorphisms
$$
\nu_0:\widehat{\Lambda}_0^\Sigma \map{} \mathrm{Ta}^\Sigma(A_0)
$$
(here $\nu_0$ and $\nu_0'$ are $U_0$ equivalent if there is a $u\in U_0$ such that $\nu_0(x)=\nu_0'(xu)$ for all $x\in\widehat{\Lambda}_0^\Sigma$) which  satisfy
\begin{enumerate}
\item 
$\nu_0$ identifies  the Weil pairing  
$$
 \mathrm{Ta}^\Sigma(A_0) \times  \mathrm{Ta}^\Sigma(A_0)  \map{} \widehat{\Z}^\Sigma(1)
$$
induced by $\lambda_0$ with the pairing (\ref{polarization pairing}) up to a $(\widehat{\Z}^\Sigma)^\times$-multiple
\item for any $\sigma\in\pi_1(S,s)$ there is a $u\in U_0$ such that $\sigma\cdot \nu_0(x)=\nu_0(x u)$; that is to say, the equivalence class of $\nu_0$ is defined over $S$.
\end{enumerate}

Set $\Lambda=\Lambda_0\otimes_\Z\co_F$.  For a compact open subgroup $U\subset U^\mathrm{max}$ (factorizable and maximal at primes in  $\Sigma$) one defines the notion of $U$ level structure on a $\mathfrak{D}^{-1}$-polarized QM abelian fourfold $(\mathbf{A},\lambda)$ in the same way, replacing $\widehat{\Lambda}_0^\Sigma$ by $\widehat{\Lambda}^\Sigma=\widehat{\Lambda}^\Sigma_0\otimes_\Z\co_F$ and replacing $\psi_0$ with the alternating form  $$\psi:\Lambda \times\Lambda\map{}\Z$$
 defined by extending $\psi_0$  to a $\co_F$-bilinear form $\psi_0:\Lambda \times\Lambda\map{}\co_F$ and setting $\psi=\mathrm{Tr}_{F/\Q}\circ \psi_0$.

If $S$ is a scheme and $(\mathbf{A}_0,\lambda_0)$ is an object of $\mathcal{M}_0(S)$ then the polarization $\lambda_0$ induces a Rosati involution $\tau\mapsto \tau^\dagger$ on  $\End^0(\mathbf{A}_0)$.  The \emph{trace} of an endomorphism $\tau\in\End^0(\mathbf{A}_0)$ is defined to be $\mathrm{Tr}(\tau)=\tau+\tau^\dagger$, and we define a bilinear form on the trace zero elements of $\End^0(\mathbf{A}_0)$  (with values in the commutative subalgebra of Rosati-fixed elements)   by
$$
[\tau_1,\tau_2]=-\mathrm{Tr}(\tau_1\tau_2).
$$
The associated quadratic form is   denoted $Q_0(\tau) =  -\tau^2 .$
Similar remarks and definitions hold for any object $(\mathbf{A},\lambda)$  of $\mathcal{M}(S)$, and  on the trace zero elements of $\End^0(\mathbf{A})$ we have the quadratic form $Q(\tau)= -\tau^2.$

For every $\alpha\in \co_F$ define a category $\mathcal{Y}(\alpha)$, fibered in groupoids over the category of schemes,  whose objects are triples $(\mathbf{A},\lambda,t_\alpha)$ in which  $(\mathbf{A},\lambda)$ is a $\mathfrak{D}^{-1}$-polarized QM abelian fourfold over a scheme and $t_\alpha$ is a trace zero element of $\End(\mathbf{A})$  satisfying $Q(t_\alpha)=\alpha$.  There is an obvious functor
\begin{equation*}
\phi: \mathcal{Y}(\alpha)\map{}\mathcal{M}
\end{equation*}
which forgets the endomorphism $t_\alpha$.   Using the methods of \cite[Chapters 6-7]{hida04} to prove the relative representability of $\phi$ one can show that the category $\mathcal{Y}(\alpha)$ is an algebraic stack over $\Z$, finite over $\mathcal{M}$.

\begin{Lem}\label{generic etale}
The stack $\mathcal{Y}(\alpha)_{/\Q}$ is \'etale over $\Spec(\Q)$.  In particular the generic fiber of $\mathcal{Y}(\alpha)$ is reduced of dimension zero.
\end{Lem}

\begin{proof}
 Choose a prime $\ell$ which does not divide $\mathrm{disc}(B_0)$ and with the property that  $$\co_F[\sqrt{-\alpha}]\otimes_\Z \Z_\ell\iso \Z_\ell^4.$$   We abbreviate $W=W(\F_\ell^\alg)$ for the ring of Witt vectors of $\F_\ell^\alg$ and $M=\mathrm{Frac}(W)$ for the fraction field of $W$.  Suppose  $z:\Spec(\F_\ell^\alg) \map{}\mathcal{Y}(\alpha)$ is a geometric point corresponding to a triple $(\mathbf{A},\lambda,t_\alpha)$ and let  $\mathfrak{g}$ denote the $\ell$-Barsotti-Tate group of $\mathbf{A}$.   The completion $R$ of the strictly Henselian local ring of $\mathcal{Y}(\alpha)$ at $z$  classifies deformations of $(\mathbf{A},\lambda,t_\alpha)$ to Artinian local $W$-algebras with residue field $\F_\ell^\alg$,   and by the Serre-Tate theorem (and \cite[p.~51]{breen-labesse} or \cite{vollaard} to see that the polarization $\lambda$ deforms uniquely through Artinian thickenings of $\F_\ell^\alg$) the $W$-algebra $R$ is identified with the deformation space of the $\ell$-Barsotti-Tate group $\mathfrak{g}$ with its action of 
 $$
( \co_B\otimes_{\co_F} \co_F[\sqrt{-\alpha}]) \otimes_{\Z}\Z_\ell \iso M_2(\Z_\ell^4).
 $$
 Using the usual idempotents on the right hand side we may decompose $\mathfrak{g}\iso \mathfrak{g}_0\times\mathfrak{g}_0$ in which $\mathfrak{g}_0$ is an $\ell$-Barsotti-Tate group of height four and dimension two with an action of $\Z_\ell^4$, and then $R$ classifies deformations of $\mathfrak{g}_0$ with its action of $\Z_\ell^4$. The Kottwitz condition on $\mathbf{A}$ implies that the usual idempotents in $\Z_\ell^4$ induce a decomposition $\mathfrak{g}_0\iso (\Q_\ell/\Z_\ell)^2\times\mu_{\ell^\infty}^2$, and the theory of Serre-Tate coordinates as in \cite[Theorem 4.2]{goren} or \cite[Theorem 8.9]{hida04} then implies that $R\iso W$.
 
 If $Y\map{}\mathcal{Y}(\alpha)_{/W}$ is a finite \'etale morphism with $Y$ a $W$-scheme then, by the previous paragraph, the local ring of $Y$ at any closed point is isomorphic to $W$.  It follows that $Y$ itself is simply a disjoint union of copies of $\Spec(W)$, and hence that $Y_{/M}$ is a disjoint union of copies of $\Spec(M)$.  We deduce that $\mathcal{Y}(\alpha)_{/M}$ is \'etale over $\Spec(M)$, and it follows that $\mathcal{Y}(\alpha)_{/\Q}$ is \'etale over $\Spec(\Q)$.
   \end{proof}

\begin{Rem}
By Lemma \ref{Lem:endo algebra}, for any object $(\mathbf{A},\lambda)$ of $\mathcal{M}$ the $F$-valued quadratic form $Q$ is totally positive on the trace zero elements of $\End^0(\mathbf{A})$.  Thus 
for any $\alpha\in\co_F$
$$
\mathcal{Y}(\alpha)\not=\emptyset \implies
\alpha \mathrm{\ totally\ positive.} 
$$
 \end{Rem}

 For any $T\in\Sym_2(\Z)^\vee$ let $\mathcal{Z}(T)$ be the category, fibered in groupoids over schemes, whose objects are quadruples $(\mathbf{A}_0,\lambda_0,s_1,s_2)$ in which $(\mathbf{A}_0,\lambda_0)$ is an object of $\mathcal{M}_0$ and $s_1,s_2\in\End(\mathbf{A}_0)$ are trace zero endomorphisms of $\mathbf{A}_0$ which satisfy
\begin{equation}
\label{fundamental form}
T=\frac{1}{2}\left(\begin{matrix}
[s_1,s_1]& [s_1,s_2] \\
[s_2,s_1]& [s_2,s_2]
\end{matrix}\right).
\end{equation}
  We also define
$$
\mathcal{Y}_0(\alpha)= \mathcal{Y}(\alpha)\times_{\mathcal{M}}\mathcal{M}_0
$$
so that an object of   $\mathcal{Y}_0(\alpha)$ is a triple  $( \mathbf{A}_0,\lambda_0, t_\alpha)$ in which $(\mathbf{A}_0,\lambda_0) $ is an object of  $\mathcal{M}_0$ and  $t_\alpha$ is a trace zero  element of $\End(\mathbf{A}_0\otimes \co_F)$   satisfying $Q(t_\alpha)= \alpha$.  Fixing such a triple over a connected scheme $S$, (\ref{endo base change})  implies that $t_\alpha$ has the form
\begin{equation*}
t_\alpha=  s_1\varpi_1+s_2\varpi_2
\end{equation*}
for some trace zero $s_1,s_2\in \End(\mathbf{A}_0)$.  The condition $Q(t_\alpha)=\alpha$ is equivalent to the relation
\begin{equation*} 
\alpha = -s_1^2\varpi_1^2-(s_1s_2+s_2s_1)\varpi_1\varpi_2 - s_2^2\varpi_2^2.
\end{equation*}
Recalling (\ref{Sigma set}), this last relation is  equivalent to
$$
\frac{1}{2}\left(\begin{matrix}
[s_1,s_1]& [s_1,s_2] \\
[s_2,s_1]& [s_2,s_2]
\end{matrix}\right)
\in \Sigma(\alpha).
$$
It follows easily  that for each $\alpha\in\co_F$ the rule $(\mathbf{A}_0,\lambda_0,t_\alpha)\mapsto (\mathbf{A}_0, \lambda_0,s_1,s_2)$ identifies 
\begin{equation}\label{moduli decomposition}
\mathcal{Y}_0(\alpha)\iso \bigsqcup_{ T\in\Sigma(\alpha) }
\mathcal{Z} (T).
\end{equation}

\begin{Lem}
\label{Lem:biquadratic}
Suppose $\alpha \in\co_F$ is totally positive and let $T$ be an element of $\Sigma(\alpha)$.
\begin{enumerate}
\item
If $\mathcal{Z}(T)\not=\emptyset$ then $T$ is positive semi-definite.
\item
If $\mathcal{Z}(T)_{/\Q}\not=\emptyset$ then $\det(T)=0$.
\item
If    $F(\sqrt{-\alpha})$ is  not a biquadratic extension of $\Q$ then $\det(T)\not=0$.
\end{enumerate}
In particular if  $\mathcal{Y}_0(\alpha)_{/\Q}\not=\emptyset$  then $F(\sqrt{-\alpha}) /\Q$ is biquadratic.
\end{Lem}

\begin{proof}
For (a), if $\mathcal{Z}(T)\not=\emptyset$ then there is some object  $(\mathbf{A}_0,\lambda_0,s_1,s_2)$ of $\mathcal{Z}(T)$ defined over a connected scheme, and Lemma \ref{Lem:endo algebra} implies that  the quadratic form represented by (\ref{fundamental form}) is positive definite.  Part (b) follows from \cite[Theorem 3.6.1]{KRY}. For (c) suppose that $\det(T)=0$.  Then there is a one dimensional $\Q$-vector space $V_0$ equipped with a  $\Q$-valued symmetric bilinear form $[\cdot,\cdot]$ and vectors $s_1,s_2\in V_0$ for which the relation  (\ref{fundamental form}) holds.  Let $V=V_0\otimes_\Q F$ and extend $[\cdot,\cdot]$ to an  $F$-valued  symmetric bilinear form on $V$.  The relation (\ref{fundamental form}) then implies that the vector $t_\alpha=s_1\varpi_1+s_2\varpi_2$ satisfies $[t_\alpha,t_\alpha]= 2 \alpha$.  As $V_0$ is one dimensional we may write $s_1=n_1 s$, $s_2=n_2 s$ as $\Q$-multiples of a common vector $s\in V_0$, and then  
$$
2 \alpha=[t_\alpha,t_\alpha]=[s(n_1\varpi_1+n_2\varpi_2),s(n_1\varpi_1+n_2\varpi_2)]= (n_1\varpi_1+n_2\varpi_2)^2\cdot [s,s]
$$
shows that $\alpha\in  \Q^\times\cdot (F^\times)^2$.  Hence $F(\sqrt{-\alpha})/\Q$ is biquadratic.

The final claim follows by combining  (b) and (c) with the decomposition (\ref{moduli decomposition}).
\end{proof}

For the remainder of \S \ref{S:Hilbert-Blumenthal}  we fix  a totally positive $\alpha\in\co_F$  and abbreviate $\mathcal{Y}=\mathcal{Y}(\alpha)$ and $\mathcal{Y}_0=\mathcal{Y}_0(\alpha)$ for the algebraic stacks over $\Z$ constructed above.  Define $j$ and $\phi_0$ by the  requirement that the diagram
$$
\xymatrix{
{\mathcal{Y}_0}\ar[r]^{\phi_0}  \ar[d]_j   &  {\mathcal{M}_0 } \ar[d]^i  \\
{\mathcal{Y} } \ar[r]_\phi & {\mathcal{M} }
}
$$
is cartesian.


\subsection{Green currents}
\label{SS:Green}


Let $\mathcal{M}_{/\Q}\iso [H\backslash M]$ be a $\Q$-uniformization of $\mathcal{M}$, define $\Q$-schemes
$$
Y=\mathcal{Y} \times_{\mathcal{M}} M
\hspace{1cm}
M_0=\mathcal{M}_{0}\times_{\mathcal{M}} M
$$  
and let $\phi:Y\map{}M$ and $i:M_0\map{}M$ be the projections to the second factors.  By Lemma \ref{generic etale} the scheme $Y$ is a disjoint union of  reduced zero dimensional components, and we define
\begin{equation}
\label{0-cycle}
C_\Q=\sum_{y\in Y} \phi(y) \in Z^2_Y(M).
\end{equation}
As $C_\Q$ is $H$-invariant it may be identified, using \cite[Lemma 4.3]{gillet84},  with a cycle $\mathcal{C}_\Q\in Z^2_{\mathcal{Y}_{/\Q}}(\mathcal{M}_{/\Q})$ which  is independent of the choice of $\Q$-uniformization of $\mathcal{M}$.

The remainder of  \S \ref{SS:Green} is devoted to the construction of a Green current for  $\mathcal{C}_\Q$.  In order to construct such a current we will use the $\Q$-uniformizations of $\mathcal{M}$ which come from the theory of canonical models of Shimura varieties.  Let  $\mathrm{Nm}:B^\times\map{}F^\times$ denote the reduced norm on $B$.  Define a real algebraic group  $\mathbb{S}=\mathrm{Res}_{\C/\R}\mathbb{G}_m$ and  pick a $z\in G_0(\R)$ with $z^2=-1$.  The point $z$ determines a map of $\R$-algebras $\C\map{}B_0\otimes_\Q\R$ by $a+bi\mapsto a+bz$, which in turn determines a map of real algebraic groups $h_z:\mathbb{S}\map{}G_{0/\R}.$  Let $X_0$ denote the $G_0(\R)$-conjugacy class of any  such map (all are conjugate by the Noether-Skolem theorem).  We may also view $h_z$ as map $\mathbb{S}\map{}G_{/\R}$ and let $X$ be the $G(\R)$-conjugacy class of $h_z$.   The set $X_0$ is naturally identified with the subset of $X$ consisting of those maps $\mathbb{S}\map{} G_{/\R}$ which factor through the subgroup  $G_{0/\R}$.  The isomorphism of $\R$-algebras (\ref{F split}) determines an isomorphism 
\begin{equation}\label{real splitting}
B\otimes_\Q\R \iso (B_0\otimes_\Q\R) \times (B_0\otimes_\Q\R)
\end{equation}
which in turn determines an injection $G(\R)\map{} G_0(\R)\times G_0(\R)$ and an injection $X\map{}X_0\times X_0$ taking the subset $X_0\subset X$ bijectively onto the diagonal.  If we fix an isomorphism $B_0\otimes_\Q\R\iso M_2(\R)$ then we may identify the set $X_0$ with the complex manifold $\C\smallsetminus \R$ in the usual way \cite[Example 5.6]{milne05}.  The resulting complex manifold structure on $X_0$ is independent of the choice of isomorphism $B_0\otimes_\Q\R\iso M_2(\R)$.   The topological space $X_0$ has two connected components, say $X_0=X_0^+\cup X_0^-$, and the injection $X\map{}X_0\times X_0$ identifies
\begin{equation}\label{components}
X\iso (X_0^+\times X_0^+) \cup (X_0^-\times X_0^-).
\end{equation}
The groups $G_0(\R)$ and $G(\R)$ act on $X_0$ and $X$, respectively, through conjugation.  In particular the complex manifold $X\times G(\A_f)$ admits a left action by $G(\Q)$ (acting on both factors) and a right action (on the second factor) by the maximal compact open subgroup $U^\mathrm{max}\subset G(\A_f)$.

Fix a normal compact open subgroup $U\subset U^\mathrm{max}$ small enough that the algebraic stack of $\mathfrak{D}^{-1}$-polarized QM abelian fourfolds over $\Q$ with level $U$ structure is a scheme, $M$, so that there is a $\Q$-uniformization of $\mathcal{M}$
$$
\mathcal{M}_{/\Q}\iso [H\backslash M]
$$
where $H=U^\mathrm{max}/U$.  By Shimura's theory \cite{milne79,milne05} there is an isomorphism of complex manifolds
\begin{equation}
\label{shimura uniformization}
M(\C)\iso G(\Q)\backslash X\times G(\A_f)/U.
\end{equation}
To (briefly) make the isomorphism (\ref{shimura uniformization}) explicit  fix a point $(x,g)\in X\times G(\A_f)$ and use strong approximation to factor $g=\gamma u$ with $\gamma\in G(\Q)$ and $u\in U^\mathrm{max}$.  The point $\gamma^{-1} x\in X$ determines (in fact, is) a map of real algebraic groups $\mathbb{S}\map{}G_{/\R}$.  Setting $\Lambda=\co_B$, the group $G(\R)$ acts by right multiplication on $\Lambda \otimes_\Z\R$, and the induced Hodge structure 
$$
\mathbb{S}\map{}\Aut_\R(\Lambda \otimes_\Z\R)
$$ 
determines a complex structure on the real vector space $\Lambda \otimes_\Z\R$.  Recall the alternating pairing $\psi:\Lambda \times\Lambda\map{}\Z$ of \S \ref{SS:moduli}.   It follows from \cite[Lemma 1.1]{milne79} that either $\psi$ or $-\psi$, depending on the connected component of $X$ containing $\gamma^{-1}x$, determines a $\mathfrak{D}^{-1}$-polarization $\lambda$ of the complex torus $A=(\Lambda \otimes_\Z\R)/\Lambda$ with its left $\co_B$-action, and, in the notation of \S \ref{SS:moduli}, the isomorphism of left $\widehat{\co}_B$-modules
$$
\nu: \widehat{\Lambda}^\emptyset \map{\cdot u} \widehat{\Lambda}^\emptyset\iso \mathrm{Ta}^\emptyset(A)
$$
is a $U$ level structure on $\mathbf{A}=(A,i)$.  The isomorphism (\ref{shimura uniformization}) then identifies the double coset of $(x,g)$ with the isomorphism class of the triple $(\mathbf{A},\lambda,\nu)$. Similarly one can show that
$$
M_0(\C)\iso G_0(\Q)\backslash X_0\times G_0(\A_f)U^\mathrm{max}/U.
$$

 Let $V_0$ and $V$ denote the trace zero elements of $B_0$ and $B$, respectively, with $G_0(\Q)$ and $G(\Q)$ acting on $V_0$ and $V$ by conjugation .   Then $V$ is endowed with the $G(\Q)$-invariant $F$-valued quadratic form $Q(\tau)=-\tau^2$, and $V_0$ is endowed  with the $G_0(\Q)$-invariant  $\Q$-valued quadratic form $Q_0$ defined by the same formula.   Each nonzero $\tau \in V_0\otimes_\Q\R$, viewed as an element of $G_0(\R)$, acts by conjugation on $X_0$ with fixed point set
$$
X_0(\tau)=\{ x\in X_0\mid \tau \cdot x=x \}
$$
 consisting of two points if $Q_0(\tau)$ is positive and no points otherwise. Assuming that $Q_0(\tau)$ is positive, let $x_0^+(\tau)\in X_0^+$ and $x_0^-(\tau)\in X_0^-$ be the two points in $X_0(\tau)$.   Identify $V\otimes_\Q\R\iso (V_0\otimes_\Q\R)^2$ using (\ref{real splitting}).  If  $\tau=(\tau_1,\tau_2)\in V\otimes_\Q\R$ with $Q(\tau)$ totally positive then we define points $x^+(\tau), x^-(\tau)\in X$ by
$$
x^\pm(\tau)=(x_0^\pm(\tau_1),x_0^\pm(\tau_2))
$$
using the isomorphism of (\ref{components}), and set $X(\tau)=\{ x^+(\tau),x^-(\tau) \}$.  If $Q(\tau)$ is not totally positive then set $X(\tau)=\emptyset$.

  Given a positive parameter $u\in\R$ Kudla has defined (see Proposition 7.3.1 and (7.3.42) of \cite{KRY}; these functions were first constructed in \cite{kudla97})  a symmetric $G_0(\R)$-invariant Green function $g_u^0(z_1,z_2)$ for the diagonal  $X_0\subset X_0\times X_0$, identically zero  off of the subset $X\subset X_0\times X_0$.   For each $\tau_0 \in V_0\otimes_\Q\R$ with $Q_0(\tau_0)>0$ define a  function $\xi_0(\tau_0)$ on $X_0$ by
$$
\xi_0(\tau_0)(z)= \sum_{x\in X_0(\tau_0)} g^0_{ Q_0(\tau_0) }(x,z).
$$
The function $\xi_0(\tau_0)$ is a Green function for the $0$-cycle $X_0(\tau_0)$ on $X_0$.  Now suppose $\tau\in V\otimes_\Q\R$ with $Q(\tau)$ totally positive and write $(\tau_1,\tau_2)$ for the corresponding element of $(V_0\otimes_\Q\R)^2$.  Define functions $\xi^1(\tau), \xi^2(\tau)$ in the variables $(z_1,z_2)\in X_0\times X_0$ by
$$
\xi^1(\tau)(z_1,z_2) = \xi_0(\tau_1)(z_1)
\hspace{1cm}
\xi^2(\tau)(z_1,z_2) = \xi_0(\tau_2)(z_2).
$$
These functions are Green functions for the divisors $X_0(\tau_1) \times X_0$ and $X_0\times X_0(\tau_2)$ (respectively) on $X_0\times X_0$.  It follows that the star product  \cite[II.3]{soule92} of Green functions on $X_0\times X_0$
$$
\xi(\tau)=\xi^1(\tau)*\xi^2(\tau)
$$
is a Green current for the $0$-cycle $X_0(\tau_1)\times X_0(\tau_2)\subset X_0\times X_0$, and restricting to a current on $X$ yields a Green current for the $0$-cycle $X(\tau)$.

 Now fix a totally positive $v\in F\otimes_\Q\R$ and write $v^{1/2}$ for the totally positive square root of $v$.  Define an  $\co_F$-lattice in $V$ by $L=V\cap \co_B $  and  $(1,1)$-current on $X$
$$
\Xi_v(\alpha) =  \sum_{  \substack{ \tau \in L \\ Q(\tau)=\alpha} } \xi (v^{1/2}\tau).
$$
For simplicity we abbreviate $\Xi_v=\Xi_v(\alpha)$.  It follows from the $G_0(\R)$-invariance of the Green function $g_u^0(z_1,z_2)$ that
$$
\gamma^*\xi(\tau) = \xi(\gamma\cdot \tau)
$$ 
for any $\gamma\in G(\R)$  and $\tau\in V$. From this and the stability of $L$ under $\Gamma^\mathrm{max}$ one sees that $\Xi_v$ is $\Gamma^\mathrm{max}$-invariant.  Viewing $\Xi_v$ as a current on the subset $X\times\{1\}\subset X\times G(\A_f)$ and using the factorization $G(\A_f)=G(\Q)U^\mathrm{max}$ of strong approximation we see that $\Xi_v$ extends uniquely to a left $G(\Q)$-invariant and right $U^\mathrm{max}$-invariant current on $X\times G(\A_f)$, which then descends to an $H$-invariant current on $M(\C)$.    If for each $\tau\in V$ we define 
$$
\Omega(\tau) = \{ g\in G(\A_f) \mid \tau\in g\widehat{\co}_Bg^{-1} \}
$$
then there is a bijection of zero dimensional complex manifolds
\begin{equation}
\label{zero dimensional uniformization}
Y(\C) \iso G(\Q)\backslash \bigsqcup_{  \substack{ \tau\in V \\ Q(\tau)=\alpha  }  }\big(  X(\tau)\times \Omega(\tau) \big)/U
\end{equation}
for which  the map $\phi:Y\map{}M$ is induced by the evident map 
$$
X(\tau)\times\Omega(\tau)\map{}X\times G(\A_f).
$$  
Here the action of $G(\Q)$ permutes the summands in the disjoint union through the conjugation action of $G(\Q)$ on $V$.  Indeed, suppose we are given a triple $(\tau, x,g)$ with $\tau\in V$, $Q(\tau)=\alpha$, and $(x,g)\in X(\tau)\times \Omega(\tau)$.  We again factor $g=\gamma u$ with $\gamma\in G(\Q)$ and $u\in U^\mathrm{max}$, so that   $(\tau,x,g)$ and $(\gamma^{-1} \tau,\gamma^{-1}x,u)$ have the same image in the right hand side of  (\ref{zero dimensional uniformization}).  To $(\tau,x,g)$ we associate the quadruple $(\mathbf{A},\lambda,\nu, t_\alpha)$ in which $(\mathbf{A},\lambda,\nu)$ is the $\mathfrak{D}^{-1}$-polarized QM abelian fourfold over $\C$ with $U$ level structure attached to $(x,g)$ as above, and $t_\alpha$ is the trace zero endomorphism of $\mathbf{A}$
$$
t_\alpha: (\Lambda\otimes_\Z\R)/\Lambda  \map{ } (\Lambda\otimes_\Z\R)/\Lambda
 $$
defined by $t_\alpha(b)=b\cdot (\gamma^{-1} \tau)$.  Note that the $\cdot $ is multiplication in $B$.

 If we define an infinite formal sum on $X$
$$
\mathcal{D} = \sum_{ \substack{\tau\in L \\  Q(\tau)=\alpha}  } \sum_{x\in X(\tau)} x
$$
then $\mathcal{D}$, viewed as a formal sum on $X\times \{1\}$, extends uniquely to a left $G(\Q)$-invariant and right $U^\mathrm{max}$-invariant formal sum on $X\times G(\A_f)$ whose associated $0$-cycle on $M(\C)$ is (\ref{0-cycle}).  Comparing $\mathcal{D}$ with the definition of $\Xi_v$ and using the fact that $\xi(v^{1/2}\tau )$ is a Green current for the $0$-cycle $X(\tau)$  shows that $\Xi_v$ is a Green current for $\mathcal{C}_\Q\in Z^2(\mathcal{M}_{/\Q})$.

The quadratic form $Q_0$ determines a bilinear form $[s_1,s_2]=-\mathrm{Tr}(s_1s_2)$ on $V_0$, where $\mathrm{Tr}$ is the reduced trace on $B_0$.  For any $T\in\Sigma(\alpha)$ let $V_0(T)\subset V_0\times V_0$ be the set of pairs $(s_1,s_2)$ for which (\ref{fundamental form}) holds.   Given any  $(s_1,s_2)\in V_0(T)$ set 
$$
\tau=s_1\varpi_1+s_2\varpi_2  \in V
$$ 
and let $\tau_1,\tau_2\in V_0\otimes_\Q\R$ be defined as above.    Define $L_0=V_0\cap \co_{B_0}$ and $L_0(T)=(L_0\times L_0)\cap V_0(T)$.
Let $v=(v_1,v_2)$  under the isomorphism (\ref{F split}) and let $T\in\Sigma(\alpha)$ be nonsingular.    Then $\{s_1,s_2\}$ is a linearly independent set,  as is  $\{\tau_1,\tau_2\}$ which implies  that the  $0$-cycles $X_0(\tau_1)$ and  $X_0(\tau_2)$  on $X_0$ have no common components.  Hence the star product   $\xi_0( v^{1/2}_1\tau_1) * \xi_0( v_2^{1/2} \tau_2) $ is a Green current for the empty cycle on $X_0$.

\begin{Prop}
\label{Prop:archimedean decomp}
Assume that $F(\sqrt{-\alpha})/\Q$ is not biquadratic, so that, by Lemma \ref{Lem:biquadratic}, $\Sigma(\alpha)$ contains no singular matrices and $\mathcal{Y}_{0/\Q}=\emptyset$.  Then 
$$
\deg_\infty( i^* \Xi_v ) = \frac{ 1}{2 }  \sum_{  \substack{ T\in \Sigma(\alpha) \\  (s_1,s_2) \in \Gamma_0^\mathrm{max} \backslash L_0(T)  }   } \frac{1}{ |\mathrm{stab}(s_1,s_2)| } \int_{X_0} \xi_0(v_1^{1/2} \tau_1 ) * \xi_0(v_2^{1/2}\tau_2 )
$$
where $\mathrm{stab}(s_1,s_2)$ is the stabilizer of $(s_1,s_2)$ in $\Gamma_0^\mathrm{max}$.
\end{Prop}

\begin{proof}
Formally one has
\begin{eqnarray} \label{current slight}
\deg_\infty( i^* \Xi_v )  
&=& 
\frac{1}{ 2\cdot   [\Gamma_0^\mathrm{max}:\Gamma_0] } \int_{\Gamma_0\backslash X_0 } i^* \Xi_v \\
&=&  \nonumber
\frac{ 1}{ 2\cdot  [\Gamma_0^\mathrm{max}:\Gamma_0] }\sum_{  \substack{ \tau\in   \Gamma_0\backslash L \\ Q(\tau)=\alpha }  }  
\int_{X_0} i^*\big( \xi^1(v^{1/2} \tau ) * \xi^2(v^{1/2}\tau ) \big)  \\
& =&  \nonumber
\frac{ 1}{ 2\cdot  [\Gamma_0^\mathrm{max}:\Gamma_0] }\sum_{  \substack{ \tau\in   \Gamma_0\backslash L \\ Q(\tau)=\alpha }  }   \int_{X_0}( i^* \xi^1(v^{1/2} \tau ) ) *  ( i^* \xi^2(v^{1/2} \tau ) )  \\
& = & \nonumber
 \frac{ 1}{ 2\cdot  [\Gamma_0^\mathrm{max}:\Gamma_0] }\sum_{  \substack{ \tau\in   \Gamma_0\backslash L \\ Q(\tau)=\alpha }  }   \int_{  X_0}    \xi_0(v_1^{1/2} \tau_1  )* \xi_0(v_2^{1/2}\tau_2 )
\end{eqnarray}
where $\Gamma_0=\Gamma\cap \Gamma_0^\mathrm{max}$ for any sufficiently small subgroup $\Gamma\subset\Gamma^\mathrm{max}$ of finite index.  These calculations require some caution, as the definition of pullback of currents is slightly subtle.  Recall from \S \ref{SS:chow2} that to define $i^*\Xi_v$ we  first replace the Green current $\Xi_v$ on $\Gamma \backslash X$ by an equivalent  Green form of logarithmic type along (the pullback to $\Gamma\backslash X$ of) the $0$-cycle $\mathcal{Y}(\alpha)(\C)$.  Such a form exists by \cite[II.3.3]{soule92}, but the argument given there makes use of the compactness of $\Gamma_0\backslash X_0$.  As we cannot apply this argument to currents on $X$, we in fact do not even have a definition of the current $i^*\big( \xi^1(v_1^{1/2} \tau_1 ) * \xi^2(v_2^{1/2}\tau_2 ) \big)$ on $X_0$.  We avoid this issue by making use of the methods of  \cite[\S 2.1.5]{gillet-soule90}.  For each $\epsilon>0$ one can construct from Kudla's Green function $g^0_u$ a smooth $\Gamma_0$-invariant $(1,1)$-form $\omega_\epsilon$ on $X$ which is supported on the hyperbolic tube of radius $\epsilon$ around $X_0$, satisfies $\lim_{\epsilon\to 0} [\omega_\epsilon] =\delta_{X_0}$ as currents on $X$, and such that for each $P\in X_0$ the restriction of $\omega_\epsilon$ to $X_0\times\{P\}$ (as well as to $\{P\}\times X_0$) converges to the current $\delta_P$ on $X_0$.  The integral in the second line of (\ref{current slight})  is then interpreted as the limit as $\epsilon\to 0$ of the current $\xi^1(v_1^{1/2} \tau_1 ) * \xi^2(v_2^{1/2}\tau_2 )$ evaluated at the smooth form $\omega_\epsilon$. If we now define a smooth $(1,1)$-form 
$$
\Omega_\epsilon = \sum_{\gamma\in \Gamma_0\backslash \Gamma} \gamma^* \omega_\epsilon
$$
on $\Gamma\backslash X$ then $\lim_{\epsilon\to 0}[\Omega_\epsilon] = \delta_{\Gamma_0\backslash X_0}$ as currents on $\Gamma\backslash X$, and the integral in the first line of (\ref{current slight}) is equal to the limit as $\epsilon\to 0$ of the value of the current $\Xi_v$ on the form $\Omega_\epsilon$.  From this the second equality of (\ref{current slight}) is clear.     It is now an easy exercise in the definition of the star product to verify that if we are given Green functions of logarithmic type $a_0$ and $b_0$ for disjoint divisors on $X_0$ and define $a=\pi_1^*a_0$ and $b=\pi_2^*b_0$ (in which $\pi_1, \pi_2:X\map{}X_0$ are the two projections) then 
$$
\int_X i^*(a * b) = \int_X (i^* a) * (i^* b) = \int_{X_0} a_0*b_0 
$$
(compare with \cite[p.~159]{gillet-soule90}).  Taking 
$$
a_0 = \xi_0(v_1^{1/2}\tau_1) \hspace{1cm} b_0=\xi_0(v_2^{1/2}\tau_2)
$$
yields the final two equalities of (\ref{current slight}). 

We have now established 
$$
\deg_\infty( i^* \Xi_v ) =  \frac{ 1}{ 2\cdot  [\Gamma_0^\mathrm{max}:\Gamma_0] }\sum_{  \substack{ \tau\in   \Gamma_0\backslash L \\ Q(\tau)=\alpha }  }   \int_{  X_0}    \xi_0(v_1^{1/2} \tau_1  )* \xi_0(v_2^{1/2}\tau_2 ).
$$
The $\Gamma_0$-invariant function $$L_0\times L_0\map{}L_0\otimes_\Z\co_F\iso L$$ defined by $(s_1,s_2)\mapsto s_1\varpi_1+s_2\varpi_2$ restricts to a $\Gamma_0$-invariant bijection
$$
\bigcup_{T\in\Sigma(\alpha)} L_0(T) \map{} \{ \tau\in L\mid Q(\tau)=\alpha \},
$$
and the proposition follows easily.

\end{proof}


\subsection{Cycle classes at primes  of good reduction}


 Let $p$ be a prime at which $B_0$ is unramified.  Fix  a ${\Z_p}$-uniformization $\mathcal{M}_{/{\Z_p}}\iso [H\backslash M]$  of $\mathcal{M}$ and define $\Z_p$-schemes $Y$, $M_0$, and $Y_0$ by (\ref{cycle uniformization}).  Let $\phi:Y\map{}M$ be the projection.

 \begin{Prop}
  \label{Prop:good dimension} 
 The scheme $Y$ is at most one-dimensional.  In particular  
 $$
 \mathbf{K}_0^Y(M)= F^2\mathbf{K}_0^Y(M).
 $$
\end{Prop}

\begin{proof}
The horizontal components of $Y$ are all Zariski closures of points of the (zero-dimensional) generic fiber $Y_{/\Q_p}$, and so are of dimension one.   We must therefore show that $Y_{/\F_p^\alg}$ has dimension at most one.   One can easily show that the  Serre-Tate canonical lift gives an injection from the set of ordinary points of $Y(\F_p^\alg)$ to the finite set $Y(W(\F_p^\alg))$, and so it suffices to show that the locus of nonordinary points of  $Y_{/\F_p^\alg}$ is of dimension at most one.  As the map $Y\map{}M$ is finite, we are now reduced to proving that the nonordinary locus of $M_{/\F_p^\alg}$ is of dimension one.   We give only a sketch of the proof;   the idea is to reduce the question to the setting of  Hilbert-Blumenthal surface  where it follows from calculations of Goren and others \cite{goren03, goren04, goren05, goren00}.  Briefly, an abelian variety $A$ over $\F_p^\alg$  is ordinary if and only if the Verschiebung $\mathrm{Ver}:\mathrm{Lie}(A^{(p)})\map{}\mathrm{Lie}(A)$ is an isomorphism of $\F_p^\alg$-modules.  Thus  the nonordinary locus of $M_{/\F_p^\alg}$ is defined locally by a single equation, the determinant of the \emph{Hasse-Witt matrix} encoding the action of Verschiebung on Lie algebras.  

Fix a nonordinary point  $x\in M(\F_p^\alg)$ and denote by $(\mathbf{A},\lambda)$ the  $\mathfrak{D}^{-1}$-polarized QM abelian fourfold determined by $x$.  Let $R$ be the completion of the local ring of $M_{/\F_p^\alg}$ at $x$.  By the Serre-Tate theory the functor of deformations of the QM-abelian surface $(\mathbf{A},\lambda)$  to local Artinian $\F_p^\alg$-algebras with residue field $\F_p^\alg$ is isomorphic to the functor of deformations of the associated polarized $p$-Barsotti-Tate group $(\mathbf{A}_p,\lambda_p)$ of height eight and dimension four with  its action of $\co_B\otimes_{\co_F}\co_{F,p}\iso M_2(\co_{F,p})$, and this functor is pro-represented by $R$.   Here we have set  $\co_{F,p}=\co_F\otimes_\Z\Z_p$.   Using the action of the idempotents in $ M_2(\co_{F,p})$ to decompose  $\mathbf{A}_p$  one can identify this deformation functor with the deformation functor of a quasi-polarized $p$-Barsotti-Tate group $(\mathbf{B}_p,\eta_p)$ of height four and dimension two equipped with an action of $\co_{F,p}$.  This is exactly the  type of deformation problem considered by Goren et.~al.~ in their study of local rings of Hilbert-Blumenthal surfaces.

 When $p$ is inert in $F$ then contemplation of the possible slope sequences shows that  $\mathbf{B}_p$ is supersingular.  One can apply Zink's theory of displays \cite{zink02} of connected $p$-Barsotti-Tate groups exactly as in \cite[Chapter 6]{goren} or \cite{goren00}  to show that $R$ is a power series ring in two variables over $\F_p^\alg$, while the determinant $r\in R$ of the Hasse-Witt matrix of the universal deformation is nonzero.  The quotient ring $R/(r)$ is therefore of dimension one as desired.  When $p$ is ramified in $F$ then again $\mathbf{B}_p$ is supersingular.  In this case the $\F_p^\alg$-algebra $R$ may not be smooth, but one can  deduce from calculations of \cite{goren03}  that the quotient of $R$ by the determinant of the Hasse-Witt matrix has dimension one.  When $p$ splits in $F$ then the idempotents in $\co_{F,p}\iso \Z_p\times\Z_p$ further split $(\mathbf{B}_p,\eta_p)$ into a product of two polarized $p$-Barsotti-Tate groups $H_1\times H_2$ each of height two and dimension one.   The formal deformation ring $R$ splits as the completed tensor product $R\iso R_1 \widehat{\otimes}_{\F_p^\alg} R_2$  of the formal deformation rings of $H_1$ and $H_2$, and the determinant of the Hasse-Witt matrix of the universal deformation  is a pure tensor $r_1\otimes r_2$. If $H_i\iso (\Q_p/\Z_p)\times\mu_{p^\infty}$ is ordinary then $R_i\iso \F_p^\alg[[x_i]]$ by the theory of Serre-Tate coordinates, and $r_i\in R_i^\times$.  If $H_i$ is the unique connected $p$-Barsotti-Tate group of dimension one and height two over $\F_p^\alg$ then again Zink's theory of displays shows that $R_i\iso \F_p^\alg[[x_i]]$ and $r_i$ is a uniformizing parameter of $R_i$.  From this we deduce that the quotient of $R$ by the determinant of the Hasse-Witt matrix has dimension one.  
 \end{proof}

 Suppose $\mathcal{F}$ is a coherent $\co_Y$-module.  We denote by
 $
 \mathrm{cl}(\mathcal{F})  \in \chow^2_Y(M)
 $
 the class determined by $R\phi_*[\mathcal{F}]\in F^2\mathbf{K}_0^Y(M)$ under the Gillet-Soul\'e isomorphism (\ref{GS iso}).  Note that the finiteness of   $\phi:Y\map{}M$ implies $R^k\phi_*\mathcal{F}=0$ for all $k >0$, and so we have simply $R\phi_*[\mathcal{F}]=[\phi_*\mathcal{F}]$.  As $\mathrm{cl}(\co_Y)$ is $H$-invariant  Lemma \ref{Lem:Galois descent}  shows that  $\mathrm{cl}(\co_Y)$ arises as the flat pullback of a class
  \begin{equation}
 \label{the good cycle}
 \mathcal{C}_p \in \chow^2_{\mathcal{Y}_{/\Z_p}}(\mathcal{M}_{/\Z_p}).
 \end{equation}

\begin{Lem}\label{Lem:good generic fiber}
The homomorphism (\ref{generic fiber})  takes $\mathcal{C}_p$ to the cycle $\mathcal{C}_\Q\times_\Q\Q_p$ of (\ref{0-cycle}).
  \end{Lem}

\begin{proof}
The construction $\mathcal{F}\mapsto \mathrm{cl}(\mathcal{F})$ from coherent sheaves to cycle classes is compatible with flat base change, so we may compute the image of $\mathcal{C}_p$ in 
$$
\chow^2_{Y_{/\Q_p}}(M_{/\Q_p}) \iso Z^2_{Y_{/\Q_p}}(M_{/\Q_p})
$$
 by repeating the construction of $\mathcal{C}_p$  with $\Z_p$ replaced by $\Q_p$.   As $Y_{/\Q_p}$ is zero dimensional and reduced it is a disjoint union of field spectra. One easily checks that the composition
$$
\mathbf{K}_0(Y_{/\Q_p})\map{R\phi_*}F^2\mathbf{K}_0^{Y_{/\Q_p}}(M_{/\Q_p}) \map{} \chow^2_{Y_{/\Q_p}}(M_{/\Q_p}) \iso Z^2_{Y_{/\Q_p}}(M_{/\Q_p})
$$
takes 
$$
[\mathcal{F}]\mapsto \sum_{y\in Y_{/\Q_p}} \length_{k(y)}(\mathcal{F}_y)\cdot \phi(y)
$$ 
and the claim follows.
\end{proof}

\begin{Def}
A coherent $\co_Y$-module $\mathcal{F}$ is \emph{skyscraper-free} at a closed point $y\in Y$ if  $$\Hom_{\co_{Y,y}}(k(y),\mathcal{F}_y)=0.$$  That is to say, $y$ is neither an embedded point of $\mathcal{F}$ nor an irreducible component  of $\mathrm{sppt}(\mathcal{F})$.
\end{Def}

\begin{Lem}
\label{Lem:cohen-macaulay}
A coherent $\co_Y$-module $\mathcal{F}$ is skyscraper-free at $y$ if and only if  the stalk $\mathcal{F}_y$ is either trivial or a one-dimensional Cohen-Macaulay $\co_{Y,y}$-module.
\end{Lem}

\begin{proof}
Fix a closed point $y\in Y$ in the support of $\mathcal{F}$.  By Proposition \ref{Prop:good dimension} the $\co_{Y,y}$-module $\mathcal{F}_y$ has dimension at most one, and hence is  Cohen-Macaulay of dimension one if and only if $\mathrm{depth}(\mathcal{F}_y)\ge 1$.  The condition $\mathrm{depth}(\mathcal{F}_y)\ge 1$ is equivalent to $\Hom_{\co_{Y,y}}(k(y),\mathcal{F}_y)=0.$
\end{proof}

By mild abuse of notation we denote again by $\deg_p$ the composition
\begin{equation}
\label{degree abuse}
\chow^2_{Y_0}(M_0)\map{}\chow^2_\vertical(M_0)\map{\deg_p}\Q.
\end{equation}

\begin{Lem}
\label{Lem:simple pullback}
Suppose that $\mathcal{F}$ is a  coherent $\co_Y$-module which is skyscraper-free at each closed point $y\in Y_0$, and assume that the support of the coherent $\co_{Y_0}$-module $j^*\mathcal{F}$ has dimension zero.  Then, recalling the pullback $i^*$ of (\ref{basic local pullback}),
$$
\deg_p(i^*\mathrm{cl}(\mathcal{F})) = \sum_{y\in Y_0(\F_p^\alg)} \length_{\co_{Y_0,y}}(j^*\mathcal{F}_y).
$$
\end{Lem}

\begin{proof}
Suppose we are given coherent $\co_M$-modules $\mathcal{G}_1$, $\mathcal{G}_2$ such that $\mathcal{G}_1\otimes_{\co_M}\mathcal{G}_2$ is supported in dimension zero.  Define the \emph{Serre intersection multiplicity} \cite[Chapter V]{serre00} at a closed point $x\in M$ by 
$$
 I_{\co_{M,x}}^\mathrm{serre}(\mathcal{G}_{1}, \mathcal{G}_{2}) = 
 \sum_{\ell\ge 0} (-1)^\ell \length_{\co_{M,x}} \mathrm{Tor}_\ell^{\co_{M,x}} (\mathcal{G}_{1,x}, \mathcal{G}_{2,x}) .
$$
Lemma \ref{Lem:cohen-macaulay} implies that $\mathcal{F}_y$ is trivial or a one-dimensional Cohen-Macaulay $\co_{Y,y}$-module for every closed point $y\in Y_0$, and it follows that at every closed point $x\in M_0$ the stalk $(\phi_*\mathcal{F})_x$  is either trivial or  a one-dimensional Cohen-Macaulay $\co_{M,x}$-module.   For a closed point $x\in M_0$ the local ring $\co_{M_0,x}$ is regular, hence Cohen-Macaulay, and so (abbreviating $\co_{M_0}=i_*\co_{M_0}$)  $\co_{M_0,x}$ is  a two-dimensional Cohen-Macaulay  $\co_{M,x}$-module. Applying the corollary of  \cite[p.~111]{serre00} we find
$$
\mathrm{Tor}_\ell^{\co_{M,x}} ((\phi_*\mathcal{F})_x,  \co_{M_0,x})=0
$$
for all $\ell >0$, and so
$$
 I_{\co_{M,x}}^\mathrm{serre}(\phi_*\mathcal{F}, \co_{M_0}) =
 \length_{\co_{M,x}} ((\phi_*\mathcal{F})_x\otimes_{\co_{M,x}} \co_{M_0,x}).
$$
By examination of the construction of (\ref{GS iso}) the composition 
$$
F^2\mathbf{K}_0^{Y_0}(M_0) \map{} \chow^2_{Y_0}(M_0)\map{\deg_p}\Q
$$
is given by 
$$
[\mathcal{G}]\mapsto \sum_{x\in M_0(\F_p^\alg)} \length_{\co_{M_0,x}}(\mathcal{G}_x).
$$
Putting this all together gives
\begin{eqnarray*}
\deg_p(i^*\mathrm{cl}(\mathcal{F})) &=&
\sum_{x\in M_0(\F_p^\alg)} I_{\co_{M,x}}^\mathrm{serre}(\phi_*\mathcal{F}, \co_{M_0}) \\
&=& \sum_{y\in Y_0(\F_p^\alg)} \length_{\co_{Y_0,y}} (\mathcal{F}_y\otimes_{\co_{M,\phi(y)}} \co_{M_0,\phi(y)})
\end{eqnarray*}
and the claim follows.
\end{proof}

We now examine the local rings of $Y$ at closed points of $Y_0$.  Let $W=W(\F_p^\alg)$ denote the ring of Witt vectors of $\F_p^\alg$, and denote by $\mathbf{Art}$ the category of local Artinian $W$-algebras with residue field $\F_p^\alg$. Let $\mathfrak{G}_0^*$ be a Barsotti-Tate group of dimension one and height two over $\F_p^\alg$.   Thus $\mathfrak{G}_0^*$ is isomorphic to the $p$-divisible group of an elliptic curve.  Abbreviate $\co=\co_F\otimes_\Z\Z_p$.  The $\Z_p$-basis $\{\varpi_1, \varpi_2\}$ of $\co$  determines an injection $\co\map{}M_2(\Z_p)$, and we define 
$$
\mathfrak{G}^*=\mathfrak{G}_0^*\times\mathfrak{G}_0^*
$$ 
with the action of $\co$ induced by the map $\co\map{}M_2(\Z_p)$.    Thus for any $\F_p^\alg$-algebra $R$ we have $\mathfrak{G}^*(R)\iso \mathfrak{G}_0^*(R)\otimes_{\Z_p}\co$.  Let $\mathbf{D}$ be the functor on $\mathbf{Art}$ which classifies deformations of $\mathfrak{G}^*$.  More precisely, for an object $R$ of $\mathbf{Art}$ let $\mathbf{D}(R)$ be the set of isomorphism classes of pairs $(\mathfrak{G},\rho)$ in which $\mathfrak{G}$ is a Barsotti-Tate group over $R$ with an action of $\co$ and 
$$
\rho:\mathfrak{G}^* \map{} \mathfrak{G}_{/\F_p^\alg}
$$
is an isomorphism of Barsotti-Tate groups  over $\F_p^\alg$ respecting $\co$-actions.  Let $\tau \in\End(\mathfrak{G}^*)$ be an endomorphism of $\mathfrak{G}^*$ which commutes with the action of $\co$.  For each object $R$ of $\mathbf{Art}$ let   $\mathbf{D}_\tau(R)\subset\mathbf{D}(R)$ to be the subset consisting of those deformations for which the endomorphism $\tau$ lifts: that is to say, those deformations $(\mathfrak{G},\rho)\in\mathbf{D}(R)$ for which  there exists a (necessarily unique) $\tilde{\tau} \in\End(\mathfrak{G})$ which commutes with the action of $\co$ and satisfies
$$
\tilde{\tau} \circ \rho = \rho\circ \tau
$$
as elements of $\Hom(\mathfrak{G}^*,\mathfrak{G}_{/\F_p^\alg})$.

\begin{Prop}\label{Prop:smooth local ring}
The functor $\mathbf{D}$ is pro-represented by a $W$-algebra $R^\mathrm{univ}$ isomorphic to a power series ring in two variables over $W$, and $\mathbf{D}_\tau$ is pro-represented by $R^\mathrm{univ}/J_\tau$ for some ideal $J_\tau\subset R^\mathrm{univ}$.
\end{Prop}

\begin{proof}
The first claim follows from work of Rapoport \cite{rapoport78} and Deligne-Pappas \cite{deligne-pappas}, who have determined the non-smooth locus of Hilbert-Blumenthal schemes. Briefly, by the Serre-Tate theorem the functor $\mathbf{D}$ is represented by the completion of the local ring of a Hilbert-Blumenthal scheme  of relative dimension two over $W$.  The smoothness of the local ring at this point is a consequence  of the fact that  $$\mathrm{Lie}(\mathfrak{G}^*)\iso \mathrm{Lie}(\mathfrak{G}^*_0)\otimes_{\Z_p}\co$$ is free of rank one over $\F_p^\alg \otimes_{\Z_p} \co$, and so satisfies the \emph{Rapoport condition} of \cite{vollaard}.  See especially \cite[Remark 3.8]{vollaard}.
The representability of $\mathbf{D}_\tau$  follows from \cite[Proposition 2.9]{rapoport96}.
\end{proof}

\begin{Prop}\label{Prop:ordinary lci}
Suppose that $\mathfrak{G}_0^*$ is isomorphic to the Barsotti-Tate group of an ordinary elliptic curve.  Then the ideal $J_\tau\subset R^\mathrm{univ}$  can be generated by two elements.
\end{Prop}

\begin{proof}
This follows from the theory of Serre-Tate coordinates as in \cite[Theorem 4.2]{goren} or \cite[Theorem 8.9]{hida04}.  If we define rank one $\co$-modules
$$
P=\mil \mathfrak{G}^*(\F_p^\alg)[p^k] \hspace{1cm} 
Q=\mil \Hom(\mathfrak{G}^*,\mu_{p^k}) 
$$
then there is a canonical isomorphism of functors on $\mathbf{Art}$
$$
\mathbf{D} \iso \Hom_{\Z_p}(P\otimes_\co Q , \mu_{p^\infty}).
$$
The endomorphism $\tau$ of $\mathfrak{G}$ induces an $\co$-linear endomorphism $x\mapsto t_P(x)$ of $P$ and an $\co$-linear endomorphism $x\mapsto t_Q(x)$ of $Q$, and there is an isomorphism
$$
\mathbf{D}_\tau \iso \Hom_{\Z_p}( (P\otimes_\co Q) / I ,\mu_{p^\infty} )
$$
where $I$ is the $\Z_p$-submodule of $P\otimes_\co Q$ 
$$
I= \{ t_P(x)\otimes y-x\otimes t_Q(y) \mid x\in P, y\in Q \}.
$$
After choosing an appropriate $\Z_p$-basis of $P\otimes_\co Q$ we find that $\mathbf{D}_\tau\iso \mu_{p^s}\times \mu_{p^t}$ for some $0\le s,t\le \infty$.  The claim follows.
\end{proof}

\begin{Prop}\label{Prop:supersingular lci}
Suppose that $\mathfrak{G}_0^*$ is isomorphic to the Barsotti-Tate group of a supersingular elliptic curve.  Then the ideal $J_\tau\subset R^\mathrm{univ}$  can be generated by two elements.
\end{Prop}

\begin{proof}
We modify the argument of \cite[Proposition 5.1]{ARGOS-8}, and use Zink's theory of displays \cite{zink02} in order to avoid the messy language of formal group laws (and formal group cohomology) in dimension two.  Denote by $\mathfrak{G}^\mathrm{univ}$ the universal deformation of $\mathfrak{G}^*$  over $R^\mathrm{univ}$, and let  $\mathfrak{m}$ denote the maximal ideal of $R^\mathrm{univ}$.   Thus the endomorphism $\tau$ of $\mathfrak{G}^*$ lifts  to an $\co$-linear endomorphism of $\mathfrak{G}^\mathrm{univ}$ over $R^\mathrm{univ}/J_\tau$ (which we again denote by $\tau$), but not over 
$$
R=R^\mathrm{univ}/\mathfrak{m}J_\tau.
$$ 
Let $\mathfrak{a}=J_\tau/\mathfrak{m}J_\tau$ be the kernel of the surjection
$$
R\map{}R^\mathrm{univ}/J_\tau
$$
so that $\mathfrak{m}\cdot \mathfrak{a}=0$.  To the Barsotti-Tate group $\mathfrak{G}^\mathrm{univ}$ over $R^\mathrm{univ}$  Zink's theory attaches a universal display $(P,Q,F,V^{-1})^\mathrm{univ}$ over $R^\mathrm{univ}$, and we denote by $(P,Q,F,V^{-1})$  the reduction of the universal  display  to $R$.   For our purposes, we need only know that $P$ is a free module of rank four over the Witt vectors $W(R)$ with an action of $\co$,  that $Q\subset P$ is an $\co\otimes_{\Z_p} W(R)$-submodule with the property that $P/Q$ is  annihilated by the kernel $I_R$ of $W(R)\map{}R$, and that 
$$
Q/I_R P\hspace{1cm} P/Q
$$
are each free of rank two over $R$.  We claim that, in addition, $Q/I_RP$ is generated as an $\co\otimes_{\Z_p} R$-module by a single element.  Indeed, let $(P^*_0,Q_0^*, F_0^*, V_0^{-1*})$ and $(P^*,Q^*,F^*,V^{-1*})$ be the displays of $\mathfrak{G}_0^*$ and $\mathfrak{G}^*$, respectively.  Then, letting $I_{\F_p^\alg}$ denote the kernel of $W\map{}\F_p^\alg$, we have
$$
(Q/I_RP)\otimes_R\F_p^\alg \iso  Q^*/I_{\F_p^\alg}P^* \iso (Q_0^*/I_{\F_p^\alg}P_0^*) \otimes_{\Z_p} \co.
$$
As $Q_0^*/I_{\F_p^\alg}P_0^*$ is free of rank one over $\F_p^\alg$, we see that $(Q/I_RP)\otimes_R\F_p^\alg$ is free of rank one over $\co\otimes_{\Z_p} \F_p^\alg$.  By Nakayama's lemma it follows that $Q/I_RP$ is generated by a single element over $\co\otimes_{\Z_p} R$.

By Zink's theory \cite[Definition 72]{zink02} the obstruction to lifting the endomorphism $\tau$   from $R^\mathrm{univ}/J_\tau$ to $R$ is given by a nontrivial homomorphism of $\co\otimes_{\Z_p}R$-modules
$$
\mathrm{Obst}: Q/I_RP \map{}\mathfrak{a}\otimes_R P/Q.
$$
Let $\gamma$ generate $Q/I_RP$ as an $\co\otimes_{\Z_p} R$-module. If we pick  an $R$-module basis $\{e_1,e_2\}\subset P/Q$ then
$$
\mathrm{Obst}(\gamma)=a_1\otimes e_1 + a_2\otimes e_2
$$
for some $a_1,a_2\in\mathfrak{a}$, and  the composition 
$$
Q/I_RP \map{\mathrm{Obst} }\mathfrak{a}\otimes_R P/Q \map{} 
\big( \mathfrak{a}/(a_1,a_2)  \big)\otimes_R P/Q
$$
is  the trivial map.  This implies \cite[(119)]{zink02} that $\tau$ can be lifted from an endomorphism over $R^\mathrm{univ}/J_\tau$ to an endomorphism over $R/(a_1,a_2)$, and therefore $(a_1,a_2)=\mathfrak{a}$.  Thus $J_r/\mathfrak{m}J_r$ is generated as an $R^\mathrm{univ}$-module by two elements, and so Nakayama's lemma implies that the ideal $J_r$ is also generated by two elements.
\end{proof}

\begin{Cor}\label{Cor:local complete intersection}
Let $y$ be a closed point of $Y_0$.  The structure sheaf $\co_Y$ is skyscraper-free at $y$.
\end{Cor}

\begin{proof}
As above let $W$ denote the ring of Witt vectors of $\F_p^\alg$ and choose a point $z\in Y_{/W}$ lying above $y$.  By Lemma \ref{Lem:cohen-macaulay} and standard results in commutative algebra (e.g.~Theorem 17.5 and the corollary to Theorem 23.3 of \cite{matsumura}) it suffices to prove that the completed  local ring $\co_{Y_{/W},z}^\wedge$ is Cohen-Macaulay of dimension one.  The geometric point $\Spec(k(z))\map{}Y_{/W}$ corresponds to a triple $(\mathbf{A},\lambda,t_\alpha)$ with $\mathbf{A}=\mathbf{A}_0\otimes\co_F$ for some QM abelian surface $\mathbf{A}_0$ over $k(z)=\F_p^\alg$.  Using the splitting $\co_{B_0}\otimes_\Z\Z_p\iso M_2(\Z_p)$ the Barsotti-Tate group of $\mathbf{A}_0$ splits as $\mathfrak{G}_0^*\times\mathfrak{G}_0^*$, and the Barsotti-Tate group of $\mathbf{A}$ splits as $\mathfrak{G}^*\times\mathfrak{G}^*$.   As the Barsotti-Tate group of any deformation of $\mathbf{A}$ admits a similar splitting into two isomorphic deformations of $\mathfrak{G}^*$, it follows from the Serre-Tate theorem that the formal deformation space of the pair $(\mathbf{A},\lambda)$ is isomorphic to $\Spf(R^\mathrm{univ})$, while the formal deformation space of the triple $(\mathbf{A},\lambda,t_\alpha)$ is isomorphic to $\Spf(R^\mathrm{univ}_\tau)$ for a suitable $\tau\in\End_\co(\mathfrak{G}^*)$  (the polarization $\lambda$ lifts uniquely to any deformation of $\mathbf{A}$ by the Corollary to \cite[Theorem 3]{vollaard} or  by \cite[p. 51]{breen-labesse}).  In other words $\co_{Y_{/W},z}^\wedge\iso R_\tau^\mathrm{univ}$. Thus Propositions \ref{Prop:smooth local ring}, \ref{Prop:ordinary lci}, and \ref{Prop:supersingular lci} imply that 
\begin{equation}
\label{the local ring}
\co_{Y_{/W},z}^\wedge \iso W[[s_1,s_2]]/( f_1,f_2 )
\end{equation}
for some power series $f_1,f_2\in W[[s_1,s_2]]$.  We know from Proposition \ref{Prop:good dimension} that the ring on the left hand side of (\ref{the local ring})  has dimension at most one, and it follows from (\ref{the local ring}) that $\co_{Y_{/W},z}^\wedge$ is Cohen-Macaulay of dimension exactly one (by, for example, Exercise 6.3.4 and Corollary 8.2.18 of \cite{liu}).
\end{proof}

\begin{Prop}\label{Prop:the good degree}
If $Y_0$ is zero dimensional then the class $C_p$  satisfies
$$
\mathrm{deg}_p(i^*C_p)  = \sum_{y\in Y_0(\F_p^\alg)} \length_{\co_{Y_0,y}}(\co_{Y_0,y}).
$$
\end{Prop}

\begin{proof}
Set $\mathcal{F}=\co_Y$ in  Lemma \ref{Lem:simple pullback} and combine with Corollary \ref{Cor:local complete intersection}.
\end{proof}


\section{Construction of cycles at primes of bad reduction}
\label{S:bad components}


Throughout \S \ref{S:bad components} we fix  a prime $p$ at which $B_0$ is ramified, and a totally positive $\alpha\in\co_F$.   By Hypothesis \ref{Hyp:discriminant} the  prime $p$ splits in $\co_F$, say as $p\co_F=\mathfrak{p}_1\mathfrak{p}_2$.   Our goal is to construct and examine a cycle class
$$
\mathcal{C}_p\in\chow^2_{\mathcal{Y}(\alpha)_{/\Z_p}} (\mathcal{M}_{/\Z_p}).
$$
The definition used in (\ref{the good cycle}) breaks down for $p\mid\mathrm{disc}(B_0)$ as the algebraic stack $\mathcal{Y}(\alpha)_{/\Z_p}$ may have vertical components of dimension two.  Our modified construction will make use of calculations of Kudla-Rapoport \cite{kudla00} (and of Kudla-Rapoport-Yang \cite{kudla04a,KRY} for the case $p=2$).

The following notation will be used throughout \S \ref{S:bad components}. Abbreviate
$\mathcal{Y}=\mathcal{Y}(\alpha)$ and  $\mathcal{Y}_0=\mathcal{Y}_0(\alpha).$
Let $W=W(\F_p^\alg)$ be the ring of Witt vectors of $\F_p^\alg$.   Denote by $\mathbf{Nilp}_{\Z_p}$ (resp. $\mathbf{Nilp}_W$)  the category of  $\Z_p$-schemes (resp. $W$-schemes)  on which $p$ is locally nilpotent.  Let $\A_f^p$ denote the prime-to-$p$ finite adeles of $\Q$.   Let $U\subset U^\mathrm{max}$ be a normal compact open subgroup which factors as $U=U_p U^p$ with  $U_p= U_p^\mathrm{max}$ and $U^p\subset G(\A_f^p)$.   Let $M$ denote the algebraic stack over $\Z_p$  whose objects are  triples $(\mathbf{A},\lambda, \nu )$ in which $(\mathbf{A},\lambda)$ is a $\mathfrak{D}^{-1}$-polarized QM abelian fourfold over a $\Z_p$-scheme $S$ and $\nu$ is a level $U$ structure on $(\mathbf{A},\lambda)$.  Define $Y$, $M_0$, and $Y_0$, by (\ref{cycle uniformization}) and assume that $U^p$ is chosen small enough that the algebraic stack $M$ (and hence also $Y$, $M_0$, and $Y_0$) is a scheme. Set $ H= U^\mathrm{max}/ U $   so that $\mathcal{M}_{/\Z_p}\iso [H\backslash M]$.   Let $\widehat{M}$   and $\widehat{Y}$ denote the formal completions of $M$ and $Y$ along their special fibers.  As in \S \ref{SS:moduli} we have the cartesian diagram 
$$
\xymatrix{
{\mathcal{Y}_0}\ar[r]^{\phi_0}  \ar[d]_j   &  {\mathcal{M}_0 } \ar[d]^i  \\
{\mathcal{Y} } \ar[r]_\phi & {\mathcal{M} }
}
$$
and similarly with $\mathcal{M}$, $\mathcal{M}_0$, $\mathcal{Y}$, and $\mathcal{Y}_0$ replaced by $M$, $M_0$, $Y$, and $Y_0$.


\subsection{The Cerednik-Drinfeld uniformization}
\label{SS:CD}


We first describe some simple generalizations of the Cerednik-Drinfeld uniformization of $M_0$.   As in \S \ref{SS:moduli} let 
$$
\widehat{\Lambda}^p_0=\prod_{\ell\not=p}(\co_{B_0}\otimes_\Z\Z_\ell)
$$ 
be the restricted topological product,  and define $\widehat{\Lambda}^p$ in the same way, replacing $B_0$ by $B$.  Fix a principally polarized QM abelian surface $(\mathbf{A}^*_0,\lambda_0^*)$ over $\F_p^\alg$, let $A^*_0$ be the abelian surface underlying $\mathbf{A}_0^*$), and choose an isomorphism
$$
\nu^*_0:\widehat{\Lambda}_0^p \map{} \mathrm{Ta}^p(A^*_0)
$$ 
of left $\widehat{\co}_{B_0}^p$-modules, where $\mathrm{Ta}^p(A_0^*)$ is the prime-to-$p$ adelic Tate module of $A_0^*$, in such a way that the Weil pairing 
$$
\mathrm{Ta}^p (A_0)\times \mathrm{Ta}^p (A_0) \map{}\widehat{\Z}^p (1)
$$ 
induced by $\lambda_0^*$ agrees with the pairing (\ref{polarization pairing}) up to a $(\widehat{\Z}^p)^\times$-multiple.  Define a $\mathfrak{D}^{-1}$-polarized QM abelian fourfold
$$
(\mathbf{A}^*,\lambda^*) = (\mathbf{A}^*_0\otimes\co_F,\lambda^*_0\otimes\co_F).
$$
Letting $A^*$ denote the abelian fourfold underlying $\mathbf{A}^*$,  $\nu_0^*$  induces an  isomorphism
$$
\nu^* : \widehat{\Lambda}^p \map{} \mathrm{Ta}^p( A^*)
$$
of left $\widehat{\co}_{B}^p$-modules which is a $U$-level structure on $(\mathbf{A}^*,\lambda^*)$. Define totally definite quaternion algebras over $\Q$ and $F$, respectively,
$$
\overline{B}_0=\End^0(\mathbf{A}^*_0 ) \hspace{1cm} \overline{B}=\End^0(\mathbf{A}^*).
$$
Let $\overline{G}_0\subset \overline{G}$ be the algebraic groups over $\Q$ defined in the same way as $G_0\subset G$, but with $B_0$ and $B$ replaced by $\overline{B}_0$ and $\overline{B}$.

 Let  $\mathfrak{G}^*_0$   denote the $p$-divisible group of $A^*_0$ equipped with its action of $\co_{B_0}\otimes_\Z\Z_p$.   For any  object $S$ of $\mathbf{Nilp}_W$ with special fiber $S_{/\F_p^\alg}$, denote by $\mathfrak{h}_m(S)$ the set of isomorphism classes of pairs $(\mathfrak{G}_0,\rho_0)$ in which  $\mathfrak{G}_0$ is a special formal $\co_{B_0}\otimes_\Z\Z_p$-module of dimension two and height four over $S$ in the sense of \cite[\S II.2]{boutot-carayol}, and  
$$
\rho_0 \in \Hom(\mathfrak{G}^*_0\times_{\F_p^\alg}  S_{/\F_p^\alg},  \mathfrak{G}_0\times_S S_{/\F_p^\alg} )\otimes_{\Z_p}\Q_p
$$ 
is a height $2m$ quasi-isogeny of $p$-divisible groups over $S_{/\F_p^\alg}$  respecting the action of $\co_{B_0}\otimes_\Z \Z_p$.   By a theorem of Drinfeld \cite[Th\'eor\`eme II.8.2]{boutot-carayol} there is a formal $W$-scheme $\mathfrak{h}_m$  whose functor of points is $S\mapsto \mathfrak{h}_m(S)$, and we set
$$
\mathfrak{X}_0=\bigsqcup_{m\in\Z} \mathfrak{h}_m.
$$    
The group $\overline{G}_0(\Q_p)$ acts on $\mathfrak{X}_0$ by 
$$
\gamma\cdot (\mathfrak{G}_0,\rho_0)=  (\mathfrak{G}_0, \rho_0\circ \gamma^{-1}),
$$
taking the component $\mathfrak{h}_m$ to the component $\mathfrak{h}_{m-\ord_p\mathrm{Nm}(\gamma)}$.  The isomorphism (\ref{endo base change}) provides an isomorphism  $\overline{B}_0\otimes_\Q F \iso \overline{B}$, and so, as $p$ splits in $F$, an isomorphism
$$
\overline{G}(\Q_p)\iso \{ (x,y)\in \overline{G}_0(\Q_p)\times \overline{G}_0(\Q_p) \mid \mathrm{Nm}(x)=\mathrm{Nm}(y) \}.
$$
This determines an action of $\overline{G}(\Q_p)$ on 
$$
\mathfrak{X}= \bigsqcup_{m\in\Z} ( \mathfrak{h}_m \times_W\mathfrak{h}_m).
$$

There is a unique isomorphism
$$
\iota_0:  \overline{B}_0 \otimes_\Q  \A_f^p  \map{}  B_0\otimes_\Q  \A_f^p
$$
of $\A_f^p $-modules  such that for every prime $\ell\not=p$ and every  $g \in \overline{B}_0 \otimes_\Q \Q_\ell$ the diagram 
$$
\xymatrix{
{ B_0\otimes_\Q \Q_\ell } \ar[r]^{\nu_0^*} \ar[d]_{\cdot \iota_0(g) }    &   {\mathrm{Ta}_\ell(A_0^*)\otimes_{\Z_\ell}\Q_\ell }  \ar[d]^{g\cdot }  \\
{ B_0\otimes_\Q \Q_\ell  } \ar[r]^{\nu_0^*}   &   {\mathrm{Ta}_\ell(A_0^*)  \otimes_{\Z_\ell} \Q_\ell }
}
$$
commutes, where the vertical arrow on the left is $x\mapsto x\cdot\iota_0(g)$.  The function $\iota_0$ satisfies $\iota_0(gh)=\iota_0(h)\iota_0(g)$.  Let 
$$
\iota: \overline{B} \otimes_\Q \A_f^p \map{}  B\otimes_\Q \A_f^p
$$ 
be the $F \otimes_\Q \A_f^p$-module map obtained by tensoring $\iota_0$ with $F$.  There are induced isomorphisms $\iota_0: \overline{G}_0(\A_f^p)\map{} G^\mathrm{op}_0(\A_f^p)$ and  $\iota: \overline{G}(\A_f^p)\map{} G^\mathrm{op}(\A_f^p)$ where $\mathrm{op}$ denotes the opposite group. For any compact open subgroup $C\subset G(\A_f^p)$ define 
$$
\overline{C}=\iota^{-1}(C)  \subset \overline{G}(\A_f^p).
$$

\begin{Prop}
\label{Prop:CD uniformization}
 There is an isomorphism of formal $W$-schemes 
\begin{equation*}
\widehat{M}_{/W} \iso \overline{G}(\Q)\backslash \mathfrak{X} \times \overline{G}(\A_f^p) / \overline{U}^p.
\end{equation*}
\end{Prop}

\begin{proof}
This follows from the general $p$-adic uniformization results of \cite{boutot97, boutot-zink, rapoport96}, although in this simple case the proof is a straightforward imitation of the proof of the Cerednik-Drinfeld uniformization as in \cite{boutot-carayol}.  We  describe the isomorphism  on $\F_p^\alg$-valued points.   First let us show show that there is a single $\co_B$-linear isogeny class of  QM abelian fourfolds over $\F_p^\alg$.  Let $\mathbf{A}=(A,i)$ be a QM abelian fourfold over $\F_p^\alg$ with  $p$-divisible group $\mathfrak{G}=A[p^\infty]$.  As $p$ splits in $F$ there is a decomposition $\mathfrak{G}\iso \mathfrak{G}^1\oplus \mathfrak{G}^2$ in which $\co_F$ acts on $\mathfrak{G}^i$ through the map $\co_F\map{}\co_{F,\mathfrak{p}_i}\iso\Z_p$, and a corresponding decomposition of the Lie algebra 
$$
\mathrm{Lie}(A)\iso \mathrm{Lie}(\mathfrak{G}^1)\oplus\mathrm{Lie}(\mathfrak{G}^2).
$$  
The Kottwitz condition on the action of $\co_B$ on $\mathrm{Lie}(A)$ implies that each $\mathfrak{G}^i$ has dimension two and is a special formal $\co_{B_0}\otimes_\Z\Z_p$ module in the sense of \cite[\S II.2]{boutot-carayol}.  By \cite[Proposition II.5.1]{boutot-carayol} each $\mathfrak{G}^i$ has height four, and so is isogenous to the square of the $p$-divisible group of a supersingular elliptic curve  by \cite[\S III.4]{boutot-carayol}.  Now let $\F$ be a finite field of $q=p^k$ elements chosen large enough that $A$, all simple isogeny factors of $A$, and all endomorphisms of $A$ are defined over  $\F$.  Let $A'$ be a simple isogeny factor of $A$ defined over $\F$ and let $\pi\in\End(A')$ be the $q$-power Frobenius endomorphism.  As the slopes of the $p$-divisible group of $A'$ are all $1/2$, the argument of \cite[Proposition III.2]{boutot-carayol} shows that $\pi^2/q$ is a root of unity in the field $\Q(\pi)$, and so by enlarging the field $\F$ we have $\pi^2=q$.  By the Honda-Tate theory $A'$ is a supersingular elliptic curve.  Thus $A$ is isogenous to the fourth power of a supersingular elliptic curve, and so any two QM abelian fourfolds are isogenous.  The argument of \cite[Remark 5.3]{milne79} shows that this isogeny can be chosen to be $\co_B$-linear.
 
Now suppose  we are given a $\mathfrak{D}^{-1}$-polarized QM abelian fourfold $(\mathbf{A},\lambda)$  over $\F_p^\alg$.  The claim is that an $\co_B$-linear quasi-isogeny 
 $ \rho\in \Hom(\mathbf{A}^*,\mathbf{A})\otimes_{\co_F}F$ can be chosen in such a way that $\rho^\vee\circ \lambda \circ \rho=\lambda^*$.  Arguing as in \cite[\S III.4]{boutot-carayol} or \cite[Proposition 1.3]{vollaard} the \emph{polarization module} $\mathcal{P}$ of $\mathbf{A}^*$, defined as the  $\co_F$-module  of all $\phi \in \Hom(A^*,A^{*\vee})$ such that
 \begin{enumerate}
 \item   $\phi$ is $\co_B$-linear
\item $\phi=\phi^\vee$
\item $ \phi\circ i(b^*) = i(b)^\vee\circ\phi$ for all $b\in \co_B$,
 \end{enumerate}
is projective of rank one, with the subset of polarizations $\mathcal{P}^+\subset\mathcal{P}$ forming a positive cone (this is also stated without proof on \cite[p.54]{breen-labesse}).  Therefore $\lambda^*$ and $\rho^\vee\circ\lambda\circ\rho$ differ by the action of the group of totally positive elements of $F^\times$.  Write $\rho^\vee\circ\lambda\circ\rho= \lambda^*\circ i(\beta)$ with $\beta\in F^\times$ totally positive.  Viewing $\overline{B}$ as an $F$-algebra we may pick $b\in\overline{B}$ of reduced norm $\beta$.  As $\overline{B}$ is the centralizer of $B$ in $\End(A^*)\otimes_{\Z} \Q$ it is stable under the Rosati involution determined by $\lambda^*$, and the Rosati involution induces the main involution on $\overline{B}$ (the only positive involution of a totally definite quaternion algebra).  It follows that
$$
\rho^\vee\circ\lambda\circ\rho= b^\vee\circ \lambda^*\circ b
$$
and so replacing $\rho$ by  $\rho\circ b^{-1}$ gives $\rho^\vee\circ\lambda\circ \rho=\lambda^*$.

We now make the isomorphism  of the proposition  explicit on geometric points.  Fix a triple $(\mathbf{A},\lambda,\nu)\in M(\F_p^\alg)$ and (denoting by $A$ the abelian scheme underlying $\mathbf{A}$) let $\rho:A^*\map{}A$ be an $\co_B$-linear isogeny chosen so that $\rho^\vee\circ\lambda\circ \rho$ is an integer multiple of $\lambda^*$.   Attached to $\rho:A^*\map{}A$ is an $\co_B\otimes_\Z\Z_p$-linear isogeny of $p$-divisible groups $\rho : \mathfrak{G}^* \map{}  \mathfrak{G}$ which, using the splittings
$$
\mathfrak{G}^*\iso \mathfrak{G}_0^*\times\mathfrak{G}_0^*
\hspace{1cm}
\mathfrak{G} \iso \mathfrak{G}^1\times\mathfrak{G}^2
$$
induced by $\co_B\otimes_{\Z}\Z_p\iso (\co_{B_0}\otimes_\Z\Z_p)^2$,   determines an element 
$$
(\mathfrak{G},\rho)=\big( (\mathfrak{G}^1, \rho^1) , (\mathfrak{G}^2,\rho^2) \big) \in 
 \mathfrak{X}_0(\F_p^\alg)\times \mathfrak{X}_0(\F_p^\alg).
 $$
As the polarizations $\lambda$ and $\lambda^*$ each have degree prime to $p$, the equality $\rho^\vee\circ\lambda \circ \rho=m\lambda^*$ with $m\in\Z$ implies that $\mathrm{deg}(\rho^1)=\mathrm{deg}(\rho^2)$, and so the pair $(\mathfrak{G},\rho)$ lies in $\mathfrak{X}(\F_p^\alg)$.  There is a unique (up to right multiplication by $\overline{U}$) $g\in\overline{G}(\A_f^p)$ for which the diagram
$$
\xymatrix{
{\widehat{\Lambda}^p }  \ar[r]^{\nu^*}\ar[d]^{\cdot \iota(g^{-1})}  &   {\mathrm{Ta}^p(A^*) }  \ar[d]^\rho  \\
{\widehat{\Lambda}^p }  \ar[r]^{\nu}  &   {\mathrm{Ta}^p(A) }  
}
$$
commutes.  The isomorphism of the proposition  then takes the triple $(\mathbf{A},\lambda,\nu)$ to the double coset of $\big( (\mathfrak{G},\rho) , g\big))$.   

The inverse function is constructed as follows.  Starting from $\big( (\mathfrak{G},\rho) , g\big)$ we may act on the left  by an element of $\overline{G}(\Q)$ to assume that $\rho:\mathfrak{G}^*\map{}\mathfrak{G}$ is an isogeny (as opposed to merely a quasi-isogeny) and that left multiplication by $\iota(g^{-1})$ stabilizes $\widehat{\Lambda}^p$.  There is then a unique choice of QM abelian fourfold $\mathbf{A}=(A,i)$ over $\F_p^\alg$,  $\widehat{\co}_B^p$-linear isomorphism $\nu:\widehat{\Lambda}^p \map{}\mathrm{Ta}^p(A)$, isomorphism $A[p^\infty]\iso \mathfrak{G}$, and extension of $\rho$ to an isogeny $\rho:\mathbf{A}^*\map{}\mathbf{A}$ for which the above diagram commutes.  Define a quasi-polarization  $\lambda\in\Hom(A,A^\vee)\otimes_\Z\Q$  by $\rho^\vee\circ\lambda \circ \rho=\lambda^*$.   The claim is that there is a  positive integer multiple of $\lambda$ which  is a $\mathfrak{D}^{-1}$-polarization of $\mathbf{A}$.  If such a multiple exists, it is clearly unique.  First pick any positive integer $m$ large enough that  $m\lambda$ is a polarization, and consider the induced polarization $m\lambda:\mathfrak{G}\map{}\mathfrak{G}^*$.  Using the splitting $\mathfrak{G}\iso \mathfrak{G}^1\times\mathfrak{G}^2$ and \cite[Lemme III.4.2]{boutot-carayol} we see that the kernel of $m\lambda$ has the form $\mathfrak{G}^1[p^{a_1}]\times\mathfrak{G}^2[p^{a_2}]$ for some integers $a_1$ and $a_2$, but from the fact that $\rho^1$ and $\rho^2$ have the same degree one deduces that $a_1=a_2$.  Thus after dividing $m$ by a power of $p$ we may assume that $m\lambda$ is a polarization of degree prime to $p$.  Recalling that the pairing 
$$
\psi : \widehat{\Lambda}^p \times \widehat{\Lambda}^p \map{}\widehat{\Z}^p
$$
of \S \ref{SS:moduli} agrees with the pairing induced by $\nu^*$, the isogeny $\lambda^*$, and the Weil pairing on $\mathrm{Ta}^p(A^*)$, one can show that the pairing on $\widehat{\Lambda}^p$ induced by $\nu$, the polarization $m\lambda$, and the Weil pairing on $\mathrm{Ta}^p(A)$ is exactly $m\cdot \mathrm{Nm}(g) \cdot \psi$.  Choosing a positive integer $k$ prime to $p$ and satisfying $\ord_\ell(k) = \ord_\ell (\mathrm{Nm}(g^{-1}))$ for every prime $\ell\not=p$, we deduce that  $(m/k)\mathfrak{D}\subset\co_F$ and that $m\lambda$ has kernel $A[(m/k)\mathfrak{D}]$.  It follows that $k\lambda$ is a $\mathfrak{D}^{-1}$-polarization of $\mathbf{A}$.
 \end{proof}

\begin{Def}
Suppose $\Sigma$ is a finite set of prime ideals of $\co_F$ and let
$$
\co_{F,\Sigma} = \{x\in F\mid \ord_\mathfrak{q}(x) \ge 0\ \forall \mathfrak{q}\not\in \Sigma\}
$$ 
be the ring of $\Sigma$-integers in $\co_F$. If $\mathbf{A}$ is a QM abelian fourfold over a $\Z_p$-scheme $S$ and  $\tau\in \End^0(\mathbf{A})$,  we say that $\tau$ is \emph{integral away from $\Sigma$} if $\tau$ lies in the image of the inclusion
$$
\End(\mathbf{A})\otimes_{\co_F} \co_{F,\Sigma}  \hookrightarrow \End^0(\mathbf{A}).
$$
\end{Def}

Define set-valued functors on  $\mathbf{Nilp}_{W}$ as follows.  For $S$ an object of $\mathbf{Nilp}_{W}$ let $\mathfrak{M}(S)$ be the set of isomorphism classes of quadruples
 $(\mathbf{A},\lambda,\nu,t_\alpha)$ in which  $(\mathbf{A},\lambda,\nu)\in M(S)$ and $t_\alpha \in\End^0(\mathbf{A})$ is  trace-zero, integral away from $\{\mathfrak{p}_1,\mathfrak{p}_2\}$,  and satisfies $Q(t_\alpha)= \alpha$.    For $k\in\{1,2\}$ we let $\mathfrak{M}^k(S)\subset \mathfrak{M}(S)$ be the subset of quadruples $(\mathbf{A},\lambda, \nu,t_\alpha)$ for which $t_\alpha$ is integral away from ${\mathfrak{p}_k}$.      There are evident morphisms
 \begin{equation}\label{formal cartesian}
 \xymatrix{
 &    {\mathfrak{M}^1  }  \ar[dr]   \\
    {\widehat{Y}_{/W} }  \ar[dr]\ar[ur]   &  & {\mathfrak{M} }  \ar[r] & {\widehat{M}_{/W}}    \\
 &  {\mathfrak{M}^2 }  \ar[ur]
 }
 \end{equation}
 and  the square is cartesian.  Fix a connected object $S$ of $\mathbf{Nilp}_W$  and let $\overline{V}_0$ and $\overline{V}$ be  trace zero elements of $\overline{B}_0$ and $\overline{B}$, respectively.  If for each $\tau\in \overline{V}$ we define
$$
\Omega( \tau )=\{ g\in \overline{G}(\A_f^p)\mid  \widehat{\Lambda}^p\cdot \iota(g^{-1}\tau g) \subset \widehat{\Lambda}^p \}
$$
then Proposition \ref{Prop:CD uniformization} generalizes to a uniformization
\begin{equation}
\label{formal covering}
\mathfrak{M}\iso  \overline{G}(\Q)\backslash \bigsqcup_{  \substack{\tau \in \overline{V}  \\ \overline{Q}(\tau)=\alpha } }  \left( \mathfrak{X} \times  \Omega(\tau)/ \overline{U}^p \right)  
\end{equation}
where $\overline{Q}$ is the quadratic form $\overline{Q}(\tau)=-\tau^2$.   On geometric points the isomorphism is defined as follows. Given a quadruple 
$$
(\mathbf{A},\lambda,\nu,t_\alpha)\in \mathfrak{M}(\F_p^\alg)
$$ 
fix, as in the proof of Proposition \ref{Prop:CD uniformization}, an $\co_B$-linear isogeny $\rho:A^*\map{}A$ for which $\rho^\vee\circ\lambda\circ\rho$ is an integer multiple of $\lambda^*$.  Let
$$
((\mathfrak{G}, \rho) , g)\in\mathfrak{X}(\F_p^\alg)\times \overline{G}(\A_f^p)
$$ 
be the pair associated to $\rho$ and $(\mathbf{A},\lambda,\nu)$ as in the proof of Proposition \ref{Prop:CD uniformization} , and define
$$
\tau=\rho^{-1}\circ t_\alpha\circ \rho \in \End^0(\mathbf{A}^*).
$$
Then $\tau$ is trace zero and satisfies $\overline{Q}(\tau)=-\alpha$, and the condition that $t_\alpha$ be integral away from $p$ is equivalent to the condition that $g\in\Omega(\tau)$.  The isomorphism (\ref{formal covering}) now takes the quadruple $(\mathbf{A},\lambda,\nu,t_\alpha)$ to the point on the right hand side of  (\ref{formal covering}) corresponding  the pair $((\mathfrak{G},\rho),g)$ lying in the component of the disjoint union indexed by $\tau$.

 For every $\tau_0\in \overline{V}_0\otimes_\Q\Q_p$, viewed as a quasi-endomorphism 
 $$
 \tau_0\in \End(\mathfrak{G}_0^* \times_S S_{/\F_p^\alg} )\otimes_{\Z_p}\Q_p,
 $$
 let  $\mathfrak{h}_m(\tau_0)(S)\subset\mathfrak{h}_m(S)$ be the subset consisting of those pairs $(\mathfrak{G}_0,\rho_0)$ for which the quasi-endomorphism
$$
\rho_0 \circ \tau_0\circ \rho_0^{-1} \in \End(\mathfrak{G}_0\times_S S_{/\F_p^\alg} ) \otimes_{\Z_p}\Q_p
$$
lies in the image of the (injective) reduction map 
$$
\End(\mathfrak{G}_0)    \map{} \End (\mathfrak{G}_0\times_S S_{/\F_p^\alg}).
$$
The set valued functor $\mathfrak{h}_m(\tau_0)$ on $\mathbf{Nilp}_W$ is represented  by a closed formal subscheme of $\mathfrak{h}_m$ (see \cite[\S 2]{kudla00} and  \cite[Chapter 2]{rapoport96} for details) and we set 
$$
\mathfrak{X}_0(\tau_0)=\bigsqcup_{m\in\Z} \mathfrak{h}_m(\tau_0) .
$$   
For each $\tau \in \overline{V}$ let $(\tau_1,\tau_2)$ be the image of $\tau$ under the isomorphism
\begin{equation}
\label{CD decomp}
\overline{V} \otimes_\Q\Q_p \iso ( \overline{V}_0 \otimes_\Q\Q_p)\times ( \overline{V}_0 \otimes_{\Q}\Q_p).
\end{equation}
Define closed formal subschemes of $\mathfrak{X}$  by
 $$
 \mathfrak{X}^1(\tau)= \bigsqcup_{m\in\Z} (\mathfrak{h}_m(\tau_1)\times_W\mathfrak{h}_m)
 \hspace{1cm}
 \mathfrak{X}^2(\tau)= \bigsqcup_{m\in\Z} (\mathfrak{h}_m\times_W\mathfrak{h}_m(\tau_2) )
 $$
As with the isomorphism  of functors (\ref{formal covering}), the  functor $\mathfrak{M}^k$ is represented by the closed  formal subscheme of $\mathfrak{M}$
\begin{equation}
\label{formal covering II}
\mathfrak{M}^k\iso  \overline{G}(\Q)\backslash  \bigsqcup_{  \substack{\tau \in \overline{V}  \\ \overline{Q}(\tau)=\alpha } }    \left( \mathfrak{X}^k ( \tau ) \times \Omega(\tau)  / \overline{U}^p \right).
\end{equation}
If we now set $\mathfrak{X}(\tau)=\mathfrak{X}(\tau_1)\times_\mathfrak{X}\mathfrak{X}(\tau_2)$ then there is a diagram
$$
 \xymatrix{
 &    {\mathfrak{X}^1(\tau)  }  \ar[dr]   \\
    {\mathfrak{X}(\tau) }  \ar[dr]\ar[ur]   &  & {\mathfrak{X} }  \ar[r] & { \Spf(W)}    \\
 &  {\mathfrak{X}^2(\tau) }  \ar[ur]
 }
$$
in which the square is cartesian, and an isomorphism of functors
$$
\widehat{Y}_{/W}\iso \overline{G}(\Q)\backslash  \bigsqcup_{  \substack{\tau \in \overline{V}  \\ \overline{Q}(\tau)=\alpha } }    \left( \mathfrak{X} ( \tau ) \times \Omega(\tau)  / \overline{U}^p \right).
$$

\begin{Lem}
\label{Lem:drinfeld codimension}
For  any $x\in\mathfrak{M}^k$ the local ring $\co_{\mathfrak{M},x}$ is regular and   
\begin{equation}
\label{some formal inequality}
\dim\co_{\mathfrak{M}^k,x} < \dim \co_{\mathfrak{M},x}.
\end{equation}
\end{Lem}

\begin{proof}
The regularity of the formal scheme $\mathfrak{X}$ is clear from its construction (and from  the construction of $\mathfrak{h}_m$ as in \cite[Chapter I]{boutot-carayol} or \cite[\S 1]{kudla00}), and so the regularity of $\mathfrak{M}$ follows from the uniformization (\ref{formal covering}).  Using the uniformizations (\ref{formal covering}) and (\ref{formal covering II}), to prove the inequality (\ref{some formal inequality}) it suffices to prove 
$$
\dim\co_{\mathfrak{X}^k(\tau),x} < \dim\co_{\mathfrak{X},x }
$$
for every closed point $x\in\mathfrak{X}^k(\tau)$.  For such an $x$, $\dim \co_{\mathfrak{X},x}=3$ while the explicit local equations for the closed formal subscheme $\mathfrak{h}_m(\tau_k)\hookrightarrow \mathfrak{h}_m$ computed by Kudla-Rapoport \cite[\S 3]{kudla00} (and, for $p=2$, by Kudla-Rapoport-Yang in the appendix to \cite[\S 11]{kudla04a}) show that $\co_{\mathfrak{X}^k(\tau),x}$ has dimension  at most $2$.
\end{proof}


\subsection{The construction of $\mathcal{C}_p$}


Continue with the notation of the previous subsection.  We have defined formal schemes $\mathfrak{M}$ and $\mathfrak{M}^k$ over $W$ which represent particular functors on $\mathbf{Nilp}_W$.  The definitions of these functors extend verbatim to functors on $\mathbf{Nilp}_{\Z_p}$, and presumably these functors are represented by formal schemes over $\Z_p$.  However, we will avoid this issue by keeping track of additional descent datum.  Let ${\mathrm{Frob}}:W\map{}W$ be the continuous $\Z_p$-algebra isomorphism whose reduction modulo $pW$ is the absolute Frobenius $x\mapsto x^p$.  Define a formal $W$-scheme $\mathfrak{M}^{\mathrm{Frob}}$ as the pullback of $\mathfrak{M}$ by ${\mathrm{Frob}}$, so that
$$
\xymatrix{
{\mathfrak{M}^{\mathrm{Frob}}}\ar[r]^{ \mathrm{Frob} } \ar[d]   &  {\mathfrak{M}}\ar[d]  \\
{\Spf(W)} \ar[r]^{\mathrm{Frob}} & {\Spf(W)}
}
$$
is cartesian.   If for any $r\in\Z^+$ we set $W_r= W/p^rW$ then the $W_r$-scheme $\mathfrak{M}_{/W_r}$ carries a universal quadruple $(\mathbf{A},\lambda,\nu, t_\alpha)$ which we may pull back by ${\mathrm{Frob}}$ to yield a quadruple $(\mathbf{A}^{\mathrm{Frob}},\lambda^{\mathrm{Frob}},\nu^{\mathrm{Frob}},t_\alpha^{\mathrm{Frob}})$ over the scheme $\mathfrak{M}^{\mathrm{Frob}}_{/W_r}$.  By the universal property of $\mathfrak{M}$ this quadruple determines a morphism of $W_r$-schemes $\mathfrak{M}_{/W_r}^{\mathrm{Frob}}\map{}\mathfrak{M}_{/W_r}$ which is easily seen to be an isomorphism.  Letting $r\to\infty$ we obtain  a canonical isomorphism of formal $W$-schemes $\mathfrak{M}^{\mathrm{Frob}} \iso\mathfrak{M}$.   If $\mathfrak{F}$ is a coherent $\co_{\mathfrak{M}}$-module we say that $\mathfrak{F}$ is ${\mathrm{Frob}}$-invariant if this isomorphism identifies the pullback  $\mathrm{Frob}^*\mathfrak{F}$ with $\mathfrak{F}$. Similarly we may speak of ${\mathrm{Frob}}$-invariant coherent sheaves on $\mathfrak{M}^k$  or  $Y_{/W}$, and Grothendieck's  theory of faithfully flat descent \cite[Chapter 6]{BLR} implies that there is a canonical bijection between ${\mathrm{Frob}}$-invariant coherent $\co_{Y_{/W}}$-modules  and coherent $\co_{Y}$-modules.

 Let $\mathfrak{F}^k$  be a ${\mathrm{Frob}}$-invariant coherent $\co_{\mathfrak{M}^k}$-module, which we view also as a ${\mathrm{Frob}}$-invariant coherent $\co_{\mathfrak{M} }$-module annihilated by the ideal sheaf of the closed formal subscheme $\mathfrak{M}^k$.  It follows from (\ref{formal cartesian}) that the ${\mathrm{Frob}}$-invariant coherent $\co_{\mathfrak{M}}$-module  $\mathrm{Tor}_\ell^{\co_\mathfrak{M}}(\mathfrak{F}^1,\mathfrak{F}^2)$ on $\mathfrak{M}$ is annihilated by ideal sheaf of the formal subscheme $\widehat{Y}_{/W}\map{} \mathfrak{M}$ for every $\ell\ge 0$.  Thus we may view $\mathrm{Tor}_\ell^{\co_\mathfrak{M}}(\mathfrak{F}^1,\mathfrak{F}^2)$  as a ${\mathrm{Frob}}$-invariant coherent $\co_{\widehat{Y}_{/W}}$-module and also, by formal GAGA \cite[Theorem 8.4.2]{FGA} and faithfully flat descent, as a coherent $\co_Y$-module.  Thus we may define
\begin{equation}
\label{Tor}
[\mathfrak{F}^1\otimes^L_{\co_{\mathfrak{M}}} \mathfrak{F}^2]\define \sum_{\ell \ge 0}  (-1)^\ell [\mathrm{Tor}_\ell^{\co_\mathfrak{M}}(\mathfrak{F}^1,\mathfrak{F}^2)] \in \mathbf{K}_0(Y).
\end{equation}
Our next goal is to prove that the class (\ref{Tor}) is supported in dimension one, in the following sense.

\begin{Def}
Let $S$ be a Noetherian scheme and $[\mathcal{F}]\in\mathbf{K}_0(S)$.  We will say that $[\mathcal{F}]$ is \emph{supported in dimension $m$} if $[\mathcal{F}]$ lies in the kernel of the localization map 
$$
\mathbf{K}_0(S) \map{} \mathbf{K}_0(\Spec(\co_{S,\eta}))
$$
defined by $[\mathcal{F}]\mapsto [\mathcal{F}_\eta]$ for every point $\eta\in S$ with $\dim \overline{\{\eta\}} > m.$
\end{Def}

\begin{Lem}\label{Lem:no tor}
Suppose $D$ is a local Noetherian domain  and $N_1$ and $N_2$ are finitely generated torsion $D$-modules with $N_2$ of finite length.  Then 
$$
\sum_{\ell\ge 0} (-1)^\ell \length_D \mathrm{Tor}^D_\ell(N_1 ,N_2) = 0.
$$
\end{Lem}

\begin{proof}
The proof is based on that of \cite[Lemma 4.1]{kudla00}.  Let 
$$
\cdots \map{} D^{n_1} \map{} D^{n_0} \map{} N_1\map{}0
$$ 
be a finite resolution of $N_1$ by free $D$-modules of finite rank.  As $N_1$ is torsion we must have $\sum(-1)^i n_i=0$.  As $N_2$ has finite length,  the alternating sum of the lengths of the homology modules of 
$$
\cdots \map{} D^{n_1} \otimes_DN_2\map{}D^{n_0}\otimes_D N_2 \map{}0
$$
 is equal to
$$
\sum_{i\ge 0} (-1)^i \cdot \length_D (D^{n_i} \otimes_D N_2 ) = \length_D(N_2 )\cdot\sum_{i\ge 0} (-1)^i n_i=0.
$$
\end{proof}

\begin{Prop}\label{Prop:derived support}
For  $k\in\{1,2\}$ let $\mathfrak{F}^k$ be a ${\mathrm{Frob}}$-invariant coherent $\co_{\mathfrak{M}^k}$-module.  The class (\ref{Tor}) is supported in dimension one.
\end{Prop}

\begin{proof}
 Let $\kappa \in Y$ be the generic point of any two-dimensional component, so that $\co_{Y,\kappa}$ is Artinian and $\kappa \in Y_{/\F_p}$.  Choose a point $\eta \in Y_{/W}$ above $\kappa$ and consider the commutative diagram
 $$
\xymatrix{ 
{ \mathbf{K}_0(Y) }  \ar[r] \ar[d] & { \mathbf{K}_0(\Spec(\co_{Y,\kappa}))} \ar[d] \ar[r]^{ } &  {\Z}\ar[d]^{\mathrm{id}} \\
{ \mathbf{K}_0(Y_{/W}) }  \ar[r]  & {\mathbf{K}_0(\Spec(\co_{Y_{/W},\eta}))  }\ar[r] & { \Z }
}
$$
 in which the  arrows are as follows: the left and middle vertical arrows are flat pullback of coherent sheaves,  the upper and lower horizontal arrows on the left are localization, the horizontal arrow in  the upper right is the isomorphism
 $[\mathcal{F}]\mapsto \length_{\co_{Y,\kappa}}(\mathcal{F})$, and the horizontal arrow in the lower right is defined similarly.  Thus to check that a class $[\mathcal{F}]\in\mathbf{K}_0(Y)$ is supported in dimension one it suffices to check that $[\mathcal{F}]$ satisfies 
 $$
 \length_{\co_{Y_{/W},\eta}}(\mathcal{F}_{/W,\eta})=0
 $$
 for every two-dimensional component $\overline{\{\eta \}} \in  Y_{/W}$.  As the residue field of $\eta$ has characteristic $p$ we may  identify $\eta$ with a point of the formal scheme $\widehat{Y}_{/W}$.       Lemma \ref{Lem:drinfeld codimension} implies that  $\co_{\mathfrak{M},\eta}$ is regular of dimension one, and so  is a discrete valuation ring; the same lemma implies that $\co_{\mathfrak{M}^k,\eta}$, and hence also $\mathfrak{F}^k_{\eta}$, is a finite length $\co_{\mathfrak{M},\eta}$-module.  Applying Lemma \ref{Lem:no tor}  we find
$$
\sum_{\ell \ge 0} (-1)^\ell\ \length_{\co_{\mathfrak{M},\eta}} \big( \mathrm{Tor}_\ell^{\co_{\mathfrak{M},\eta}}(\mathfrak{F}^1_{\eta},\mathfrak{F}^2_{\eta})\big) =0.
$$
This implies $\length_{\co_{Y_{/W},\eta}} ( \mathfrak{F}^1\otimes^L_{\co_\mathfrak{M}} \mathfrak{F}^2 )=0$, completing the proof.
\end{proof}

\begin{Lem}
\label{Lem:low support}
Suppose $[\mathcal{F}]\in \mathbf{K}_0(Y)$ is supported in dimension one.   Then the image of $[\mathcal{F}]$ under the homomorphism $$R\phi_*:\mathbf{K}_0(Y)\map{} \mathbf{K}_0^Y(M)$$
 of (\ref{derived push-forward})  lies in $F^2\mathbf{K}_0^Y(M)$.
\end{Lem}

\begin{proof}
Lemma \ref{Lem:sheaf decomp} implies that $[\mathcal{F}]$ lies in the image of $\mathbf{K}_0(Z)\map{}\mathbf{K}_0(Y)$ for some closed subscheme $Z\map{}Y$ of dimension one.  But then  $R\phi_*[\mathcal{F}]$ lies in the image of 
$$
 \mathbf{K}_0^Z(M) = F^2\mathbf{K}_0^Z(M)\map{} F^2 \mathbf{K}_0^Y(M).
$$
\end{proof}

We now have a machine for producing elements of $\chow^2_Y(M)$.  If $[\mathcal{F}]\in\mathbf{K}_0(Y)$ is supported in dimension one we define, using Lemma \ref{Lem:low support},
$$
\mathrm{cl}(\mathcal{F}) \in \chow^2_Y(M)
$$ 
to be the image of $R\phi_*[\mathcal{F}]$ under the map $F^2\mathbf{K}_0^Y(M)\map{} \chow^2_Y(M)$ induced by the isomorphism (\ref{GS iso}).   In particular, if $\mathfrak{F}^k$ is a  ${\mathrm{Frob}}$-invariant coherent $\co_{\mathfrak{M}^k}$-module for $k\in\{1,2\}$ then we may form, using Proposition \ref{Prop:derived support}, 
\begin{equation}\label{the machine}
\mathrm{cl}(\mathfrak{F}^1\otimes^L_{\co_{\mathfrak{M}}} \mathfrak{F}^2 ) \in  \chow^2_Y(M).
\end{equation}
The issue is which coherent sheaves to choose, and we make the obvious choice  $\mathfrak{F}^k=\co_{\mathfrak{M}^k}$.  Applying the construction (\ref{the machine}), define
$$
C_p=\mathrm{cl}(  \co_{\mathfrak{M}^1} \otimes^L_{\co_{\mathfrak{M} }}  \co_{\mathfrak{M}^2} ) \in \chow^2_Y(M).
$$
The group $H=U^\mathrm{max}/U$ acts on  $\mathfrak{M}$ and $\mathfrak{M}^k$ by permuting $U$-level structures, and the cycle class $C_p$ is $H$-invariant by construction.  By Lemma \ref{Lem:Galois descent} $C_p$ arises as the flat pullback of a cycle class
\begin{equation}
\label{the bad cycle}
\mathcal{C}_p  \in \chow^2_{\mathcal{Y}_{/\Z_p}}(\mathcal{M}_{/\Z_p}).
\end{equation}

\begin{Prop}\label{Lem:bad generic fiber}
The homomorphism (\ref{generic fiber}) takes $\mathcal{C}_p$ to the cycle $\mathcal{C}_\Q\times_\Q\Q_p$ constructed from (\ref{0-cycle}).
  \end{Prop}

\begin{proof}
First suppose that we have a closed point $x_0\in \mathfrak{X}_0$ and a $\tau_0\in \overline{V}_0\otimes_\Q\Q_p$.  In the language of \cite{kudla00} the point $x_0$ is either \emph{ordinary} or \emph{superspecial}.  Let $R_0$ denote the completion of the local ring $\co_{\mathfrak{X}_0,x_0}$ under the topology induced by its maximal ideal, and let $N_0$ denote the completion of the stalk  $\co_{\mathfrak{X}_0(\tau_0),x_0}$.   Let $P_0$ denote the quotient of $N_0$ by its maximal $W$-torsion submodule.  Kudla-Rapoport \cite[\S 3]{kudla00} and Kudla-Rapoport-Yang (in the case $p=2$; see the appendix to \cite[\S 11]{kudla04a})  determine the structure of $N_0$, and hence also $P_0$, quite explicitly.  If $x_0$ is ordinary then $R_0\iso W[[s]]$ and either $P_0=0$ or $P_0\iso W[[s]]/(s)$.  If $x_0$ is superspecial then $R_0\iso W[[s,t]]/(st-p)$ and either $P_0=0$ or 
$$
P_0\iso W[[s,t]]/(st-p,s+ap+ut) 
$$
for some $a,u\in W$ with $u$ a unit.  

Now fix a closed point $x\in \mathfrak{X}$ contained in both $\mathfrak{X}^1(\tau)$ and $\mathfrak{X}^2(\tau)$, let $R$ be the completed local ring of $\co_{\mathfrak{X}}$ at $x$,  let $N^k$ be the completed stalk of   $\co_{\mathfrak{X}^k(\tau)}$ at $x$, and let $P^k$ be the quotient of $N^k$ by its maximal $W$-torsion submodule.   The geometric point 
$$
\Spf(k(x))\map{}\mathfrak{X}\map{} \mathfrak{X}_0\times_W\mathfrak{X}_0
$$
determines an ordered pair of points $(x_1,x_2)$ of  $\mathfrak{X}_0$, and we consider separately the four possibilities for the pair of points appearing.  First suppose both points are ordinary.  Then
$$
R\iso W[[s_1,s_2]]\hspace{1cm} P^1\iso R/(s_1)\hspace{1cm} P^2\iso R/(s_2).
$$
Using the projective resolution
$$
0\map{} R\map{s_1} R\map{}P_1\map{}0
$$
of $P_1$ we find that the $R$-modules $\mathrm{Tor}^R_\ell(P^1,P^2)$ are equal to the homology modules of the complex
$$
0\map{} R/(s_2)\map{s_1}R/(s_2)\map{}0
$$
and therefore
\begin{equation}\label{rigid tor}
\mathrm{Tor}^R_\ell(P^1,P^2) =0 \hspace{.5cm} \forall  \ell >0.
\end{equation}
Next suppose that $x_1$ is ordinary and $x_2$ is superspecial.  Then 
$$
R\iso W[[s_1,s_2,t_2]]/(s_2t_2-p)  \hspace{1cm} P^1\iso R/(s_1)\hspace{1cm} P^2\iso R/( s_2+a_2p+u_2 t_2 )
$$
and using the same projective resolution of $P_1$ as in the ordinary/ordinary case we find that $\mathrm{Tor}^R_\ell(P^1,P^2)$ is given by the homology of 
$$
0\map{} R/(s_2+a_2+u_2t_2)\map{s_1}R/(s_2+a_2p+u_2t_2)\map{}0.
$$
Again we see that (\ref{rigid tor}) holds.  Obviously the case of $x_1$ superspecial and $x_2$ ordinary is similar.  Finally suppose that $x_1$ and $x_2$ are both superspecial.  Then
$$
R\iso W[[s_1,t_1, s_2,t_2]]/(s_1t_1-p, s_2t_2-p)  \hspace{1cm} P^k\iso R/(s_k+a_k p +u_k t_k )
$$
for some $a_k\in W^\times$.  Using the projective resolution
$$
0\map{}R\map{s_1+a_1p+u_1 t_1} R\map{}P^1\map{}0
$$
of $P^1$ one again verifies (\ref{rigid tor}).  We deduce that in all cases, whenever $\ell>0$ the $R$-module $\mathrm{Tor}^R_\ell(N^1,N^2)$  is $W$-torsion.  Using the $p$-adic uniformizations (\ref{formal covering}) and (\ref{formal covering II}) it follows  that
$$
\mathrm{Tor}_\ell^{\co_{\mathfrak{M},x}}( \co_{\mathfrak{M}^1,x} ,  \co_{\mathfrak{M}^2,x} )
$$
is $W$-torsion  for all $\ell>0$ and all closed points $x\in \mathfrak{M}$.

Fix $\ell>0$ and abbreviate 
$$
\mathcal{T} = \mathrm{Tor}_\ell^{\co_{\mathfrak{M}}}(\co_{\mathfrak{M}^1},\co_{\mathfrak{M}^2} )
$$
viewed either as a coherent $\co_Y$-module or a coherent $\co_{\widehat{Y}_{/W}}$-module.  A closed point $y\in Y$ is contained in the special fiber of $Y$, and so may be viewed as a point of $\widehat{Y}$.  For any point $x\in \widehat{Y}_{/W}$ above $y$, the previous paragraph shows  that $\mathcal{T}_x$ is $W$-torsion, and it follows that $\mathcal{T}_y$ is $\Z_p$-torsion.  As this holds for all closed points $y\in Y$, it holds for \emph{all} points $y\in Y$, including those contained in the generic fiber.  But if $y\in Y$ lies in the generic fiber then $\mathcal{T}_y$ is a $\Q_p$-vector space, and hence $\mathcal{T}_y=0$.   Thus the pullback $\mathcal{T}_{/\Q_p}$ of $\mathcal{T}$ by $Y_{/\Q_p}\map{}Y$ has trivial stalks, and so is trivial.  In particular $[\mathcal{T}]$ lies in the kernel of the flat pullback $\mathbf{K}_0(Y)\map{}\mathbf{K}_0(Y_{/\Q_p})$ and so  the class
$$
[\co_Y] =[\co_{\mathfrak{M}^1} \otimes_{\co_{\mathfrak{M}}}  \co_{\mathfrak{M}^2} ]= [\mathrm{Tor}_0^{\co_{\mathfrak{M}}}(\co_{\mathfrak{M}^1},\co_{\mathfrak{M}^2} )]
$$
has the same image as $[ \co_{\mathfrak{M}^1} \otimes^L_{\co_{\mathfrak{M}}}  \co_{\mathfrak{M}^2} ]$  under $\mathbf{K}_0(Y)\map{}\mathbf{K}_0(Y_{/\Q_p})$.  The claim  now follows exactly as in the proof of Lemma \ref{Lem:good generic fiber}.
\end{proof}


\subsection{Computing the pullback}
\label{S:formal pullback calculation}


We assume throughout \S \ref{S:formal pullback calculation} that  $Y_{0/\Q_p}=\emptyset$.  
Continuing with the notation of the previous subsections, our next task is to compute $\deg_p(i^*C_p) $ where $i^*$ is the pullback of (\ref{basic local pullback}) and $\deg_p$ is the composition (\ref{degree abuse}).    Let $\mu:Y_0\map{}\Spec(\Z_p)$ be the structure map.  The assumption that $Y_{0/\Q_p}=\emptyset$ implies that for any coherent $\co_{Y_0}$-module $\mathcal{F}$ the higher direct image $R^\ell\mu_*\mathcal{F}$ is a torsion $\Z_p$-module, and so there is a  push-forward
$$
R\mu_*:\mathbf{K}_0(Y_0)\map{}\mathbf{K}_0^{\{p\}}(\Spec(\Z_p) )
$$
defined by $$R\mu_*[\mathcal{F}] = \sum_{\ell\ge 0}(-1)^\ell [R^\ell \mu_*\mathcal{F}].$$  The composition of the push-forward with the isomorphism $\mathbf{K}_0^{\{p\}}(\Spec(\Z_p) )\iso \Z$ defined by $[\mathcal{G}]\mapsto \length_{\Z_p}(\mathcal{G})$ defines the \emph{Euler characteristic}
$
\chi:\mathbf{K}_0(Y_0)\map{}\Z
$
$$
\chi(\mathcal{F}) = \sum_{\ell\ge 0}(-1)^\ell \length_{\Z_p}(R^\ell \mu_*\mathcal{F}).
$$

Let $\widehat{M}_0$ and $\widehat{Y}_0$ denote the formal completions of $M_0$ and $Y_0$ along their special fibers, respectively.  Define  formal $W$-schemes ($k\in\{1,2\}$)
$$
\mathfrak{M}_0 = \widehat{M}_{0/W} \times_{\widehat{M}_{/W}} \mathfrak{M}
\hspace{1cm}
\mathfrak{M}_0^k = \widehat{M}_{0/W} \times_{\widehat{M}_{/W}} \mathfrak{M}^k
$$
so that   there is an isomorphism
$$
\widehat{Y}_{0/W}\iso  \mathfrak{M}_0\times_\mathfrak{M}\widehat{Y}_{/W}
$$
and a diagram
$$
 \xymatrix{
 &    {\mathfrak{M}^1_0  }  \ar[dr]   \\
    {\widehat{Y}_{0/W} }  \ar[dr]\ar[ur]   &  & {\mathfrak{M}_0 }  \ar[r] & {\widehat{M}_{0/W}}    \\
 &  {\mathfrak{M}^2_0 }  \ar[ur]
 }
$$
in which the square is cartesian.   For each $(s_1,s_2)\in \overline{V}_0$ we view $\tau=s_1\varpi_1+s_2\varpi_2$ as an element of $\overline{V}\iso \overline{V}_0\otimes_\Q F$ and let $(\tau_1,\tau_2)\in ( \overline{V}_0\otimes_\Q\Q_p)^2$ denote the image of $\tau$  under (\ref{CD decomp}).  If we define
$$
\Omega_0(s_1,s_2)  = \Omega(\tau)\cap  \overline{G}_0(\A_f^p) \\
$$
then, as in the discussion surrounding \cite[(8.17)]{kudla00}, there is a Cerednik-Drinfeld style uniformization
\begin{equation}\label{formal covering III}
\mathfrak{M}_0  \iso  \overline{G}_0(\Q)\backslash \bigsqcup_{  \substack{ (s_1, s_2) \in \overline{V}_0   \\ \overline{Q}(\tau)=\alpha } }  \left( \mathfrak{X}_0 \times  \Omega_0(s_1,s_2) \overline{U}^{\mathrm{max},p} / \overline{U}^p \right) .
\end{equation}
One obtains a similar uniformization of $\mathfrak{M}^k_0$ by replacing $\mathfrak{X}_0$ with  $\mathfrak{X}_0(\tau_k)$, and a uniformization of $\widehat{Y}_{0/W}$ by replacing $\mathfrak{X}_0$ with $\mathfrak{X}_0(\tau_1)\times_{\mathfrak{X}_0} \mathfrak{X}_0(\tau_2)$.

Define a homomorphism $\mathbf{K}_0(Y)\map{}\mathbf{K}_0(Y_0)$ as follows.
 If $\mathcal{F}$ is a coherent $\co_Y$-module then, by pulling back to $Y_{/W}$ and passing to the formal completion, we may view $\mathcal{F}$ as a $\sigma$-invariant coherent $\co_\mathfrak{M}$-module  annihilated by the ideal sheaf of the closed formal subscheme $\widehat{Y}_{/W} \map{}\mathfrak{M}$.  For each $\ell\ge 0$ the coherent $\co_{\mathfrak{M}}$-module $\mathrm{Tor}_\ell^{\co_{\mathfrak{M}}}(\mathcal{F},\co_{\mathfrak{M}_0})$ is $\mathrm{Frob}$-invariant and is annihilated by the ideal sheaf of the closed formal subscheme $\widehat{Y}_{0/W}\map{} \mathfrak{M}$, and so by formal GAGA and faithfully flat descent may be viewed as a coherent $\co_{Y_0}$-module.  Thus we may define
$$
 [\mathcal{F}  \otimes^L_{\co_\mathfrak{M}} \co_{\mathfrak{M}_0}] \define \sum_{\ell\ge 0} (-1)^\ell\ [\mathrm{Tor}_\ell^{\co_{\mathfrak{M}}}(\mathcal{F},\co_{\mathfrak{M}_0})] \in\mathbf{K}_0(Y_0),
$$
and one checks that $[\mathcal{F}]\mapsto  [\mathcal{F}  \otimes^L_{\co_\mathfrak{M}} \co_{\mathfrak{M}_0}]$ defines a map $\mathbf{K}_0(Y)\map{}\mathbf{K}_0(Y_0)$.

\begin{Prop}
\label{Prop:drinfeld degree}
  If  $[\mathcal{F}]\in\mathbf{K}_0(Y)$ is supported in dimension one then  the class $  [\mathcal{F}  \otimes^L_{\co_\mathfrak{M}} \co_{\mathfrak{M}_0}]  $ is supported in dimension zero and 
\begin{eqnarray}
\label{bad reduction degree}
\deg_p(i^* \mathrm{cl}(\mathcal{F})) = \chi( \mathcal{F}  \otimes^L_{\co_\mathfrak{M}} \co_{\mathfrak{M}_0}  )
\end{eqnarray}
where $i^*$ is the pullback of (\ref{basic local pullback}) and $\deg_p$ is the composition (\ref{degree abuse}).
\end{Prop}

\begin{proof}
Using Lemma \ref{Lem:sheaf decomp} we are reduced to the case where $\mathcal{F}$ is supported in dimension one in the usual sheaf-theoretic sense.  To show that $  [\mathcal{F}  \otimes^L_{\co_\mathfrak{M}} \co_{\mathfrak{M}_0}]  $ is supported in dimension zero we must show that it lies in the kernel of 
$$
\mathbf{K}_0(Y_0)\map{}\mathbf{K}_0(\Spec(\co_{Y_0,\kappa})) \iso \Z
$$
for every $\kappa \in Y_0$ satisfying $\dim\overline{\{\kappa \}}=1$, where the isomorphism is
$$[\mathcal{G}]\mapsto \length_{\co_{Y_0,\eta}}(\mathcal{G}).$$    As we assume that $Y_{0/\Q_p}=\emptyset$ such a $\kappa$  must lie in the special fiber of $Y_0$, and thus we may view $\kappa$ as a point of the formal completion $\widehat{Y}_0$.  Given a point $\eta \in\widehat{Y}_{0/W}$ lying above $\kappa$ the hypothesis that $\mathcal{F}$ is supported in dimension one implies that $\mathcal{F}_{/W,\eta}$ is a finite length  $\co_{\widehat{Y}_{/W},\eta}$-module, and Lemma \ref{Lem:no tor}  then shows that
 $$
 \sum_{\ell \ge 0} (-1)^\ell \length_{\co_{\mathfrak{M},\eta}} \big( \mathrm{Tor}_\ell^{\co_{\mathfrak{M},\eta}} (\mathcal{F}_{/W,\eta },  \co_{\mathfrak{M}_0,\eta} ) \big)=0.
 $$
 Exactly as in the proof of Proposition \ref{Prop:derived support} it follows that
$$
\length_{\co_{Y_0,\kappa}}( [\mathcal{F}\otimes^L_{\co_{\mathfrak{M} }} \co_{\mathfrak{M}_0} ] )=0.
$$
This completes the proof of the first claim.

It remains to prove (\ref{bad reduction degree}).  Let $\tilde{\mathbf{K}}_0(Y)\subset \mathbf{K}_0(Y)$ be the subgroup of classes supported in dimension one, and similarly let $\tilde{\mathbf{K}}_0(Y_0)\subset \mathbf{K}_0(Y_0)$ be the subgroup of classes supported in dimension zero.  There is a commutative diagram
$$
\xymatrix{
{\tilde{\mathbf{K}}_0(Y)  } \ar[r] \ar[d]^{R\phi_*} & {\tilde{\mathbf{K}}_0(Y_0)  } \ar[d]^{R\phi_{0*}}   \\
{F^2\mathbf{K}_0^Y(M) } \ar[r]^{i^*} \ar[d]  & {F^2\mathbf{K}_0^{Y_0}(M_0) } \ar[d]  \\
{\chow^2_Y(M)} \ar[r]^{i^*}  &  {\chow^2_{Y_0}(M_0) }
}
$$
in which the top horizontal arrow is $[\mathcal{F}]\mapsto [\mathcal{F}  \otimes^L_{\co_\mathfrak{M}} \co_{\mathfrak{M}_0}]$ and the middle horizontal arrow is the pullback of \S \ref{SS:K theory}, given by
$$
[\mathcal{G}]\mapsto [\mathcal{G}  \otimes^L_{\co_M} \co_{M_0}] = \sum_{\ell\ge 0} (-1)^\ell [\mathrm{Tor}_\ell^{\co_M}(\mathcal{G},\co_{M_0})].
$$
Thus the left hand side of (\ref{bad reduction degree}) is equal to the image of $[\mathcal{F}\otimes^L_{\co_{\mathfrak{M}}}\co_{\mathfrak{M}_0}]$ under the composition
$$
\tilde{\mathbf{K}}_0(Y_0) \map{R\phi_{0*}}  F^2\mathbf{K}_0^{Y_0}(M_0) \map{} \chow^2_{Y_0}(M_0) \map{\deg_p}\Q.
$$
Using Lemma \ref{Lem:sheaf decomp} to reduce to the case in which $\mathcal{G}$ is a skyscraper sheaf supported at a closed point of $Y_0$ one checks  that this composition is simply $[\mathcal{G}]\mapsto \chi(\mathcal{G})$, completing the proof.
\end{proof}

Given $\mathrm{Frob}$-invariant coherent $\co_{\mathfrak{M}_0^k}$-modules $\mathfrak{F}_0^k$ for $k\in\{1,2\}$, define 
$$
[ \mathfrak{F}_0^1\otimes^L_{\co_{\mathfrak{M}_0 }}\mathfrak{F}_0^2 ]
= \sum_{\ell\ge 0}  (-1)^\ell [\mathrm{Tor}_\ell^{\co_{\mathfrak{M}_0}} (\mathfrak{F}_0^1,\mathfrak{F}_0^2)  ]
\in \mathbf{K}_0(Y_0)
$$
as in the construction (\ref{Tor}).  This class is supported in dimension zero, by imitating the proof of Proposition \ref{Prop:derived support}.

\begin{Lem}
\label{Lem:associative tensor}
  We have the equality of Euler characteristics
$$
\chi\big( (\co_{\mathfrak{M}^1}  \otimes^L_{\co_{\mathfrak{M}}}\co_{\mathfrak{M}^2}) \otimes^L_{\co_{\mathfrak{M}}} \co_{\mathfrak{M}_0} \big) 
= 
\chi( \co_{\mathfrak{M}_0^1} \otimes^L_{\co_{\mathfrak{M}_0 }}\co_{\mathfrak{M}_0^2} ).
$$
\end{Lem}

\begin{proof}
Elementary homological algebra shows that
\begin{eqnarray*}
\lefteqn{
  (\co_{\mathfrak{M}^1}  \otimes^L_{\co_{\mathfrak{M}}}\co_{\mathfrak{M}^2}) \otimes^L_{\co_{\mathfrak{M}}} \co_{\mathfrak{M}_0}   }   \\
 & & =
  (\co_{\mathfrak{M}^1}  \otimes^L_{\co_{\mathfrak{M}}}\co_{\mathfrak{M}_0}) \otimes^L_{\co_{\mathfrak{M}}}    (\co_{\mathfrak{M}^2}  \otimes^L_{\co_{\mathfrak{M}}}\co_{\mathfrak{M}_0}) 
\end{eqnarray*}
in the derived category of  locally free quasi-coherent $\co_{\mathfrak{M}}$-modules.  From this it follows that
\begin{eqnarray*}
\lefteqn{
\chi\big(  (\co_{\mathfrak{M}^1}  \otimes^L_{\co_{\mathfrak{M}}}\co_{\mathfrak{M}^2}) \otimes^L_{\co_{\mathfrak{M}}} \co_{\mathfrak{M}_0}  \big)  }   \\
 & & =
 \sum_{i,j,\ell}(-1)^{i+j+\ell}  \chi  \big(\mathrm{Tor}^{\co_{\mathfrak{M}_0 } }_\ell 
 \big(  \mathrm{Tor}^{\co_{\mathfrak{M} } }_i ( \co_{\mathfrak{M}^1 }  , \co_{\mathfrak{M}_0})  ,
  \mathrm{Tor}^{\co_{\mathfrak{M} } }_j ( \co_{\mathfrak{M}^2 }  , \co_{\mathfrak{M}_0})      \big)    \big).
\end{eqnarray*}

We claim that 
$$
\mathrm{Tor}^{\co_{\mathfrak{M} } }_i ( \co_{\mathfrak{M}^1 }  , \co_{\mathfrak{M}_0})  
\iso \left\{ \begin{array}{ll}
\co_{\mathfrak{M}_0^1} & \mathrm{if\ }i=0 \\
0 & \mathrm{otherwise} \end{array}
\right.
$$
for all $i \ge 0$ (and similarly for $\co_{\mathfrak{M}^2}$).   When $i=0$ this is clear from $\mathfrak{M}^1\times_\mathfrak{M} \mathfrak{M}_0\iso \mathfrak{M}_0^1$.   When $i>0$ it suffices to check the vanishing at stalks of the sheaf on the left.  Using the uniformization (\ref{formal covering}) we are then reduced to proving the vanishing of 
\begin{equation}\label{fancy tors}
\mathrm{Tor}^{\co_{\mathfrak{h}_m\times_W \mathfrak{h}_m} }_i ( \co_{\mathfrak{h}_m(\tau_1)\times_W\mathfrak{h}_m}   , \co_{\mathfrak{h}_m}) =0
\end{equation}
for every $\tau_1\in\overline{V}_0\otimes_\Q\Q_p$.  Here we view $\mathfrak{h}_m$ as a closed formal subscheme of $\mathfrak{h}_m \times_W\mathfrak{h}_m$ using the diagonal embedding, and hence view $\co_{\mathfrak{h}_m}$ as a coherent $\co_{\mathfrak{h}_m \times_W \mathfrak{h}_m}$-module.  The key observation is that the flatness of $\mathfrak{h}_m\map{}\Spf(W)$ implies the flatness of the first projection $\pi_1:\mathfrak{h}_m\times_W\mathfrak{h}_m\map{}\mathfrak{h}_m$.
 Choosing a  resolution 
 $$
 \mathcal{P}^\bullet\map{} \co_{\mathfrak{h}_m(\tau_1)} 
 $$ 
 of $\co_{\mathfrak{h}_m(\tau_1)}$ by locally free quasi-coherent $\co_{\mathfrak{h}_m}$-modules the pullback complex 
 $$
 \pi_1^*\mathcal{P}^\bullet\map{} \co_{\mathfrak{h}_m(\tau_1)\times_W\mathfrak{h}_m} 
 $$ 
 is a locally free resolution of $\co_{\mathfrak{h}_m(\tau_1)\times_W\mathfrak{h}_m}$.   Using this resolution to compute the left hand side of (\ref{fancy tors}) and using the isomorphism  
 $$
 (\pi_1^*\mathcal{P})\otimes_{\co_{\mathfrak{h}_m \times_W \mathfrak{h}_m } } \co_{\mathfrak{h}_m} \iso \mathcal{P}
 $$ 
 for any quasi-coherent $\co_{\mathfrak{h}_m}$-module $\mathcal{P}$,  (\ref{fancy tors}) follows.

Combining the previous two paragraphs we are now left with
$$
\chi\big(  (\co_{\mathfrak{M}^1}  \otimes^L_{\co_{\mathfrak{M}}} \co_{\mathfrak{M}^2})  \otimes^L_{\co_{\mathfrak{M}}} \co_{\mathfrak{M}_0}  \big)
  =
\sum_{\ell}(-1)^{\ell}  \chi  \big(\mathrm{Tor}^{\co_{\mathfrak{M}_0 } }_\ell (  \co_{\mathfrak{M}^1_0} ,  \co_{\mathfrak{M}^2_0 }  ) \big)  
$$
as desired.
\end{proof}

\begin{Cor}\label{Cor:too much tensor} 
We have
$$
\deg_p(i^* C_p) = \chi (\co_{\mathfrak{M}_0^1} \otimes^L_{\co_{\mathfrak{M}_0 }}  \co_{\mathfrak{M}_0^2} ).
$$
\end{Cor}

\begin{proof}
Combine Proposition \ref{Prop:drinfeld degree} and Lemma \ref{Lem:associative tensor}.
\end{proof}

We now turn to the calculation of the right hand side of the equality of Corollary \ref{Cor:too much tensor}, which we reduce to calculations of  Kudla-Rapoport-Yang.   Let $\mathrm{Tr}$ be the reduced trace on $\overline{B}_0$.  Define a $\Q$-valued bilinear form on $\overline{V}_0$ 
$$
[s_1,s_2]=-\mathrm{Tr} (s_1s_2)
$$
and for each  $T\in\Sym_2(\Z)^\vee$ let  $\overline{V}_0(T)\subset \overline{V}_0\times \overline{V}_0$ be the set of pairs $(s_1,s_2)$ such that (\ref{fundamental form}) holds.   It follows from \cite[Theorem III.3.1]{lam} that the group $\overline{G}_0(\Q)$ acts transitively on $\overline{V}_0(T)$ by conjugation.  Define a $\Z[1/p]$-lattice $\overline{L}_0\subset \overline{V}_0$
$$
\overline{L}_0= \{ v\in \overline{V}_0 \mid \widehat{\Lambda}_0^p \cdot \iota_0(v)\subset \widehat{\Lambda}_0^p  \}
$$
 and let $\overline{L}_0(T)= \overline{V}_0(T)\cap ( \overline{L}_0 \times \overline{L}_0 )$.  Define also a discrete subgroup  
$$
\overline{\Gamma}_0= \{ \gamma \in \overline{G}_0(\Q) \mid \widehat{\Lambda}^p_0 \cdot  \iota_0(\gamma) = \widehat{\Lambda}^p_0  \} 
$$
and note that $\overline{\Gamma}_0$ acts on $\overline{L}_0$ and on $\overline{L}_0(T)$ by conjugation.

Let $ \overline{Q}_0(\tau_0)= -\tau_0^2$  be the  quadratic form on $\overline{V}_0$ associated to $[\ ,\ ]$.  Suppose that  $T\in\Sym_2(\Z)^\vee$ is nonsingular, and that $s_1, s_2\in \overline{V}_0 \otimes_\Q\Q_p$  satisfy both $\overline{Q}_0(s_k) \in\Z_p-\{0\}$  and the relation (\ref{fundamental form}).   In particular this implies that $s_1$ and $s_2$ are linearly independent.  Letting $\mu:\mathfrak{h}_m\map{}\Spf(W)$ denote the structure map, a theorem of Kudla-Rapoport  \cite[Theorem 6.1]{kudla00} (and using the appendix to \cite[Chapter 6]{KRY} for the case $p=2$) shows that for any $m\in\Z$ the integer
$$
e_p(T)\define \sum_{k,\ell\ge 0} (-1)^{k+\ell} \mathrm{length}_W  R^\ell\mu_* \mathrm{Tor}_k^{\co_{\mathfrak{h}_m}} (\co_{\mathfrak{h}_m(s_1)} ,\co_{\mathfrak{h}_m(s_2)} )
$$
depends only on the isomorphism class of the rank two quadratic space 
$$
(\Z_p\oplus\Z_p,T)\iso (\Z_ps_1+\Z_ps_2, \overline{Q}_0)
$$ and not on $s_1$, $s_2$ themselves (and all $W$-modules appearing in the sum have finite length).  Furthermore $e_p(T)$ is independent of $m$ by the argument following \cite[(8.40)]{kudla00}.  In the notation of \cite[Chapter 6.2]{KRY} we have  $\nu_p(T)=2\cdot e_p(T)$.

For each $T\in\Sigma(\alpha)$ define 
$
Z(T) = \mathcal{Z}(T)_{/\Z_p}  \times_{\mathcal{M}_{/\Z_p}}  M.
$
The decomposition (\ref{moduli decomposition}) induces an analogous decomposition of $Y_0$, and hence an isomorphism
\begin{equation}
\label{alt moduli decomposition}
\mathbf{K}_0(Y_0) \iso \bigoplus_{T\in \Sigma(\alpha)} \mathbf{K}_0(Z(T)).
\end{equation}
For any $[\mathcal{G}]\in \mathbf{K}_0(Y_0)$ let $\chi_T(\mathcal{G})$ denote the Euler characteristic of the projection of $[\mathcal{G}]$ to the summand $\mathbf{K}_0(Z(T))$.

\begin{Prop}
\label{Prop:drinfeld decomposition}
Assume that $F(\sqrt{-\alpha} )/\Q$ is not biquadratic, so that $Y_{0/\Q_p}=\emptyset$ and  $\Sigma(\alpha)$ contains no singular matrices by Lemma \ref{Lem:biquadratic}.  If $T\in\Sigma(\alpha)$ satisfies (\ref{fundamental form}) for some $s_1, s_2\in \overline{V}_0 $ then  
$$
\frac{1}{|H|} \cdot 
\chi_T ( \co_{\mathfrak{M}_0^1}  \otimes^L_{\co_{\mathfrak{M}_0}} \co_{\mathfrak{M}_0^2}  )
=  e_p(T) \cdot | \overline{\Gamma}_0 \backslash \overline{L}_0(T)| .
$$
\end{Prop}

\begin{proof}
The proof is based on  that of \cite[Theorem 8.5]{kudla00}.
Abbreviate $[\mathfrak{F}_T]$ for the projection of  $[\co_{\mathfrak{M}_0^1}  \otimes^L_{\co_{\mathfrak{M}_0}} \co_{\mathfrak{M}_0^2} ]$ to the summand of (\ref{alt moduli decomposition}) indexed by $T$ and let $$\mu:Y_0\map{}\Spec(\Z_p)$$ be the structure map so that
$$
\chi_T( \co_{\mathfrak{M}_0^1}  \otimes^L_{\co_{\mathfrak{M}_0}} \co_{\mathfrak{M}_0^2}  ) = \length_{\Z_p}( R\mu_* \mathfrak{F}_T ).
$$

 Let $\widehat{Z}(T)$ denote the formal completion of $Z(T)$ along its special fiber.  After formal completion and base change to $W$ the morphism $Z(T)\map{} M_0$ admits a factorization
$$
\widehat{Z}(T)_{/W} \map{} \mathfrak{M}_0(T) \map{} \widehat{M}_{0/W}
$$ 
in which the open and closed formal subscheme $\mathfrak{M}_0(T)$ of $ \mathfrak{M}_0$ is defined,  following (\ref{formal covering III})  and \cite[(8.17)]{kudla00},  as 
$$
\mathfrak{M}_0(T)  \iso  \overline{G}_0(\Q)\backslash \bigsqcup_{  (s_1, s_2) \in \overline{V}_0(T)   }  \left( \mathfrak{X}_0 \times  \Omega_0(s_1,s_2) \overline{U}^{\mathrm{max},p} / \overline{U}^p \right). 
$$
The closed  formal subscheme $\widehat{Z}(T)_{/W}\map{}\mathfrak{M}_0(T)$  is obtained by replacing $\mathfrak{X}_0$ with $$\mathfrak{X}_0(\tau_1) \times_{\mathfrak{X}_0} \mathfrak{X}_0(\tau_2)$$ as in the comments following (\ref{formal covering III}), where $(\tau_1,\tau_2)\in ( \overline{V}_0\otimes_\Q\Q_p)^2$ is the image of $\tau=s_1\varpi_1+s_2\varpi_2\in  \overline{V}$  under (\ref{CD decomp}) .  Similarly we define
$$
\mathfrak{M}^k_0(T)  \iso  \overline{G}_0(\Q)\backslash \bigsqcup_{     (s_1, s_2) \in \overline{V}_0(T) }       \left( \mathfrak{X}_0 (\tau_k) \times \Omega_0(s_1,s_2) \overline{U}^{\mathrm{max},p} /  \overline{U}^p \right)
$$
so that there is a diagram
$$
\xymatrix{
  &  {  \mathfrak{M}_0^1(T)  } \ar[dr]   \\
  { \widehat{Z}(T)_{/W}  }    \ar[ur]\ar[dr]   &     &    {  \mathfrak{M}_0(T)  }   \ar[r]   &  {\mathfrak{M}_0}   \\
   &   {  \mathfrak{M}_0^2(T)  }  \ar[ur]
}
$$
in which the square is cartesian.    We have the equality in $\mathbf{K}_0( Z(T))$
$$
[\mathfrak{F}_T] = \sum_{\ell \ge 0}(-1)^\ell [\mathrm{Tor}_\ell^{\co_{\mathfrak{M}_0(T)}} (   \co_{\mathfrak{M}_0^1(T)}  ,  \co_{\mathfrak{M}_0^2(T)}   ) ].
$$

 Fix one pair $(s_1,s_2)\in \overline{V}_0(T)$.  As all such pairs are conjugate by $\overline{G}_0(\Q)$ with stabilizer $\Q^\times$ (using the linear independence of $\{s_1,s_2\}$) there are isomorphisms
\begin{eqnarray*}
\mathfrak{M}_0(T)
& \iso & \Q^\times \backslash  \left( \mathfrak{X}_0 \times  \Omega_0(s_1,s_2) \overline{U}^{\mathrm{max},p}  /  \overline{U}^p \right)  \\
& \iso &  (\mathfrak{h}_0 \sqcup \mathfrak{h}_1)\times  \left( \Q^\flat \backslash \Omega_0(s_1,s_2) \overline{U}^{\mathrm{max},p} / \overline{U}^p \right)
\end{eqnarray*}
where
$\Q^\flat =\{x\in\Q^\times\mid \ord_p(x)=0\}$.  For the second isomorphism we have used the fact that the action of $z\in \Q^\times$ on $\mathfrak{X}_0$ takes $\mathfrak{h}_m$ isomorphically to $\mathfrak{h}_{m-2\ord_p(z)}$.  Under these identifications we have
$$
\mathfrak{M}_0^k(T)  \iso  (\mathfrak{h}_0(\tau_k) \sqcup \mathfrak{h}_1(\tau_k))\times  \left( \Q^\flat \backslash \Omega_0(s_1,s_2) \overline{U}^{\mathrm{max},p} / \overline{U}^p \right)
$$
and it follows that $\length_{\Z_p}(R\mu_*\mathfrak{F}_T)$ is equal to
\begin{eqnarray*}
\length_{\Z_p}(R\mu_*\mathfrak{F}_T)  
& = &  | \Q^\flat \backslash \Omega_0(s_1,s_2) \overline{U}^{\mathrm{max},p} / \overline{U}^p  |  \\
& &  \times \Big[\sum_{ k,\ell \ge 0}(-1)^{k+\ell} \length_W R^k\mu_* \mathrm{Tor}_\ell^{\co_{\mathfrak{h}_0}} (  \co_{\mathfrak{h}_0}(\tau_1) , \co_{\mathfrak{h}_0} (\tau_2)  )  \\
& &+  \sum_{ k,\ell \ge 0}(-1)^{k+\ell} \length_W R^k\mu_* \mathrm{Tor}_\ell^{\co_{\mathfrak{h}_1}} (  \co_{\mathfrak{h}_1}(\tau_1) , \co_{\mathfrak{h}_1} (\tau_2)  )  \Big]  \\
& =  &  2\cdot    | \Q^\flat \backslash \Omega_0(s_1,s_2) \overline{U}^{\mathrm{max},p} / \overline{U}^p  |   \cdot e_p(T) 
\end{eqnarray*}
where we have used  the equality  $\Z_p s_1+\Z_p s_2=\Z_p\tau_1 +\Z_p\tau_2$  of lattices in $\overline{V}_0\otimes_\Q\Q_p$ for the final equality.  It now suffices to prove
\begin{equation}
\label{silly cosets}
2\cdot |\Q^\flat\backslash \Omega_0(s_1,s_2)\overline{U}^{\mathrm{max},p}/\overline{U}^p |  =     |H| \cdot  | \overline{\Gamma}_0 \backslash \overline{L}_0(T) |.
\end{equation}
By strong approximation any element of $\Omega_0(s_1,s_2)\overline{U}^{\mathrm{max},p}$ has the form $\gamma u$ with $$\gamma\in \overline{G}_0(\Q)\cap \Omega_0(s_1,s_2)$$ and $u\in \overline{U}^{\mathrm{max},p}$.  The element $\gamma$ then satisfies
 $$
  (\gamma^{-1}s_1,\gamma^{-1}  s_2) \in \overline{L}_0(T)
 $$ 
and we have a well-defined function
$$
\Q^\flat\backslash \Omega_0(s_1,s_2) \overline{U}^{\mathrm{max},p}/ \overline{U}^p \map{}  \overline{\Gamma}_0 \backslash \overline{L}_0(T)
$$
which takes the double coset of $\gamma u$ to  $\overline{\Gamma}_0\cdot (\gamma^{-1}s_1,\gamma^{-1}  s_2).$
The surjectivity of this function follows easily from the transitivity of $\overline{G}_0(\Q)$ on $\overline{V}_0(T)$.  The fiber  over $\overline{\Gamma}_0\cdot (\gamma^{-1} s_1,\gamma^{-1} s_2)$ is in bijection with $ \{\pm 1\} \backslash  \gamma  \overline{U}^{\mathrm{max},p}/ \overline{U}^p$, and (\ref{silly cosets})  follows.
\end{proof}


\section{Pullbacks of arithmetic cycles}



\subsection{Construction of arithmetic cycles}
\label{SS:construction}


Throughout \S \ref{SS:construction} we fix  $\alpha \in\co_F$ and $v\in F\otimes_\Q\R$, both totally positive, and  let $v=(v_1,v_2)$ be the corresponding element under the isomorphism (\ref{F split}).    Abbreviate $\mathcal{Y}=\mathcal{Y}(\alpha)$ and $\mathcal{Y}_0=\mathcal{Y}_0(\alpha)$ as in earlier sections.  For each prime $p$ let 
$$
\mathcal{C}_p \in\chow^2_{\mathcal{Y}_{/\Z_p}}(\mathcal{M}_{/\Z_p})
$$
be the cycle class (\ref{the good cycle}) if $p$ does not divide $\mathrm{disc}(B_0)$, or the cycle class (\ref{the bad cycle})  if $p$ does divide $\mathrm{disc}(B_0)$.  Let $\mathcal{C}_\Q\in Z^2_{\mathcal{Y}_{/\Q}}(\mathcal{M}_{/\Q})$ be the cycle constructed in \S \ref{SS:Green} with Green current $\Xi_v=\Xi_v(\alpha)$.  By Lemmas \ref{Lem:good generic fiber} and \ref{Lem:bad generic fiber} each triple $(\mathcal{C}_\Q,\Xi_v ,\mathcal{C}_p)$ is a local cycle datum of codimension two with support on $\mathcal{Y}$ in the sense of \S \ref{SS:chow2}, and  it is easy to see that  $\mathcal{C}_p^\vertical =0$ for any prime $p$ with the property that $\mathcal{Y}$ has no components supported in characteristic $p$. Hence $(\mathcal{C}_\Q,\Xi_v,\mathcal{C}_\bullet)$ is a global cycle datum.  
Applying the construction of (\ref{datum to cycle}) to $(\mathcal{C}_\Q,\Xi_v,\mathcal{C}_\bullet)$ yields an arithmetic cycle class
$$
\widehat{\mathcal{Y}}(\alpha,v) 
=
\widehat{\mathcal{Y}}(\alpha,v)^\horizontal + \sum_{p} \widehat{\mathcal{Y}}(\alpha)^\vertical_p
\in \widehat{\chow}^2(\mathcal{M})
$$
in which $\widehat{\mathcal{Y}}(\alpha,v)^\horizontal$ is the Zariski closure of $\mathcal{C}_\Q$ endowed with the Green current $\Xi_v$ and $\widehat{\mathcal{Y}}(\alpha)^\vertical_p$ is the image of $\mathcal{C}_p^\vertical$ under the map (\ref{vertical embedding}). 

For the remainder of \S \ref{SS:construction} we assume that $F(\sqrt{-\alpha})/\Q$ is not biquadratic.
By the final claim of Lemma \ref{Lem:biquadratic} the stack $\mathcal{Y}_{0/\Q}$ is empty, and so  (\ref{pullback degree})  allows us to compute the arithmetic degree along $\mathcal{M}_0$ of the arithmetic cycle class $\widehat{\mathcal{Y}}(\alpha,v)$ in terms of the global cycle datum $(i^*\mathcal{C}_\Q, i^*\Xi_v, i^*\mathcal{C}_\bullet)$ of codimension two supported on $\mathcal{Y}_0$.  Indeed, (\ref{pullback degree}) tells us that 
\begin{equation}
\label{proper pullback decomposition}
\widehat{\deg}_{\mathcal{M}_0} \widehat{\mathcal{Y}}(\alpha,v)= \deg_\infty(i^* \Xi_v)  +  \sum_{p \mathrm{\ prime} }\deg_p(i^*\mathcal{C}_p)\cdot \log(p) 
\end{equation}
where, slightly abusing notation, $\deg_p$ is the composition
$$
\chow^2_{\mathcal{Y}_{0/\Z_p}} (\mathcal{M}_{0/\Z_p} ) \map{}\chow^2_\vertical(\mathcal{M}_{0/\Z_p} ) \map{\deg_p}\Q.
$$

For every $T\in\Sym_2(\Z)^\vee$ and symmetric positive definite $\mathbf{v}\in M_2(\R)$ Kudla, Rapoport, and Yang  have defined (see \S 3.6 of \cite{KRY}  and Chapter 6 of  \emph{loc.~cit.}) an arithmetic zero cycle 
$$
\widehat{\mathcal{Z}}(T,\mathbf{v})\in\widehat{\chow}^2_\R(\mathcal{M}_0)
$$
in the $\R$-arithmetic Chow group as defined in \cite[\S 2.4]{KRY}.  Futhermore, when  $\det(T)\not=0$ they define a finite set of places $\mathrm{Diff}(T,B_0)$ of $\Q$   with the property 
\begin{equation}
\label{zero-cycle support}
\mathcal{Z}(T)(\F_p^\alg)\not=\emptyset \implies \mathrm{Diff}(T,B_0)= \{p\}
\end{equation}
for each finite place $p$.  When $T\in\Sym_2(\Z)^\vee$ is nonsingular and $|\mathrm{Diff}(T,B_0)| > 1$ then $\widehat{\mathcal{Z}}(T,\mathbf{v})=0$ by definition.   Recall that $\det(T)\not=0$ for every $T\in\Sigma(\alpha)$ by Lemma \ref{Lem:biquadratic} and our hypothesis that $F(\sqrt{-\alpha})/\Q$ is not biquadratic.

Recalling our fixed $\Z$-basis $\{\varpi_1,\varpi_2\}$ of $\co_F$ set
\begin{equation}\label{twisting matrix}
R=\left( \begin{matrix}    \varpi_1 & \varpi_1^\sigma  \\ \varpi_2 & \varpi_2^\sigma \end{matrix} \right)
\end{equation}
and abbreviate
\begin{equation}\label{bold v}
\mathbf{v} = R \left(\begin{matrix}  v_1 & \\  & v_2 \end{matrix}\right) {}^t R.
\end{equation}
According to \cite[\S 2.4]{KRY} the $\R$-arithmetic Chow group admits  an \emph{arithmetic degree} isomorphism
$$
\widehat{\deg}: \widehat{\chow}^2_\R(\mathcal{M}) \iso \R
$$
defined in exactly the same way as the arithmetic degree constructed in  \S \ref{SS:chow2} for arithmetic Chow groups with rational coefficients.

\begin{Lem}\label{Lem:proper archimedean}
We have 
$$
\deg_\infty(i^*\Xi_v)  = \sum_{ \substack{T\in\Sigma(\alpha) \\ \mathrm{Diff}(T,B_0)=\{ \infty \} }}\widehat{\deg}\ \widehat{\mathcal{Z}}(T,\mathbf{v}).
$$
\end{Lem}

\begin{proof}
Return to the notation of \S \ref{SS:Green}.   For any nonsingular $T\in\Sym_2(\Z)^\vee$ we have $\mathrm{Diff}(T,B_0)=\{\infty\}$ if and only if the matrix $T$ is represented by the quadratic space $V_0$.  Assume that this is the case.  For each pair $(s_1,s_2)\in V_0(T)$ set 
$$
\tau=s_1\varpi_1+s_2\varpi_2 \in V\iso V_0\otimes_\Q F
$$ 
and write  $(\tau_1 ,\tau_2)$ for the image of $\tau$ under $V\otimes_\Q \R \iso (V_0\otimes_\Q\R)^2$.  
 The matrix 
 $$
\mathbf{u} = R \left( \begin{matrix}   v_1^{1/2} &  \\   &   v_2^{1/2} \end{matrix} \right)
$$
satisfies   $\mathbf{v}=\mathbf{u}\cdot {^t}\mathbf{u}$  and $(s_1,s_2)\cdot\mathbf{u}=(v_1^{1/2}\tau_1, v_2^{1/2}\tau_2) $.  Therefore the real number $\nu_\infty(T,\mathbf{v})$ of \cite[Corollary 6.3.4]{KRY} is given by
$$
\nu_\infty(T,\mathbf{v})= \frac{1}{2} \int_{X_0} \xi_0(v_1^{1/2}\tau_1) * \xi_0(v_2^{1/2}\tau_2)
$$
and does not depend on the choice of $(s_1,s_2) \in V_0(T)$.  But then by definition (see \cite[(6.3.2)]{KRY}) we  have
\begin{equation}\label{boring archimedean}
  \widehat{\deg}\ \widehat{\mathcal{Z}}(T,\mathbf{v})  
 =
   \sum _{ (s_1,s_2)\in \Gamma_0^\mathrm{max} \backslash L_0(T)  }   \frac{1}{ 2 \cdot  |\mathrm{stab}(s_1,s_2) | } \int_{X_0} \xi_0(v_1^{1/2}\tau_1) * \xi_0(v_2^{1/2}\tau_2)  \\
\end{equation}
where $\mathrm{stab}(s_1,s_2)$ is the stabilizer of $(s_1,s_2)$ in $\Gamma_0^\mathrm{max}$. When $\mathrm{Diff}(T,B_0) \not= \{\infty\}$ the sum on the right hand side of (\ref{boring archimedean}) is empty,  and so the claim follows from  Proposition \ref{Prop:archimedean decomp}.
\end{proof}

\begin{Lem}\label{Lem:proper I}
If $p\nmid\mathrm{disc}(B_0)$ then
$$
\deg_p(i^*\mathcal{C}_p) \log(p) = \sum_{ \substack{T\in\Sigma(\alpha) \\ \mathrm{Diff}(T,B_0)=\{p\} }} \widehat{\deg}\ \widehat{\mathcal{Z}}(T,\mathbf{v}).
$$
\end{Lem}

\begin{proof}
As $p\nmid \mathrm{disc}(B_0)$, Proposition 6.11 and Corollary 6.1.2 of \cite{KRY} imply that for every $T\in\Sigma(\alpha)$ with $\mathrm{Diff}(T,B_0)=\{p\}$
$$
\widehat{\deg}\ \widehat{\mathcal{Z}}(T,\mathbf{v}) = 
\sum_{z \in \mathcal{Z}(T)(\F_p^\alg) } \frac{\log(p)}{|\Aut_{\mathcal{Z}(T)}(z)|}  \length_{\mathcal{O}^\mathrm{sh}_{\mathcal{Z}(T),z }}  (\mathcal{O}^\mathrm{sh}_{\mathcal{Z}(T),z})  
$$
where the sum is over the isomorphism classes of objects in $\mathcal{Z}(T)(\F_p^\alg)$ and  $\mathcal{O}^\mathrm{sh}_{\mathcal{Z}(T),z }$ is the strictly Henselian local ring of $\mathcal{Z}(T)$ at $z$. On the other hand, as we assume that $F(\sqrt{-\alpha})/\Q$ is not biquadratic Lemma \ref{Lem:biquadratic}, the decomposition (\ref{moduli decomposition}), and \cite[Theorem 3.6.1]{KRY} imply that $\mathcal{Y}_{0/\Z_p}$ is zero dimensional.  Thus we may apply Proposition \ref{Prop:the good degree} to obtain
$$
\deg_p(i^*\mathcal{C}_p) 
=  
 \sum_{y \in \mathcal{Y}_0(\F_p^\alg) } \frac{ 1}{|\Aut_{\mathcal{Y}_0}(y)|}  \length_{\mathcal{O}^\mathrm{sh}_{\mathcal{Y}_0,y }}  (\mathcal{O}^\mathrm{sh}_{\mathcal{Y}_0,y})
$$
(by dividing both sides of the equality of Proposition \ref{Prop:the good degree}  by $|H|$). Combining this with the decomposition (\ref{moduli decomposition}) and using (\ref{zero-cycle support}) we arrive at the desired equality.
\end{proof}

\begin{Lem}\label{Lem:proper II}
Suppose $p\mid\mathrm{disc}(B_0)$.  Then
$$
\deg_p(i^*\mathcal{C}_p) \log(p) = \sum_{ \substack{T\in\Sigma(\alpha) \\ \mathrm{Diff}(T,B_0)=\{p\} }} \widehat{\deg}\ \widehat{\mathcal{Z}}(T,\mathbf{v}).
$$
\end{Lem}

\begin{proof}  
Return to the notation of \S \ref{S:bad components}.  By Corollary \ref{Cor:too much tensor}  we have
$$
\deg_p(i^* C_p) =   \chi( \co_{\mathfrak{M}_0^1} \otimes^L_{ \co_{\mathfrak{M}_0} }  \co_{\mathfrak{M}_0^2 }) .
$$
For each $T\in\Sigma(\alpha)$ either $\mathrm{Diff}(T,B_0)\not=\{p\}$ or $\mathrm{Diff}(T,B_0)=\{p\}$.  In the former  case $\mathcal{Z}(T)(\F_p^\alg)=\emptyset$.  In the latter case  the matrix $T$ is represented by the quadratic space $\overline{V}_0$ of   \S \ref{SS:CD}, and so  the hypotheses of Proposition \ref{Prop:drinfeld decomposition} are satisfied.  Therefore
$$
\frac{1}{|H|}\cdot  \deg_p(i^* C_p) =  \sum_{ \substack{T\in\Sigma(\alpha) \\ \mathrm{Diff}(T,B_0)=\{p\} }} e_p(T)\cdot | \overline{\Gamma}_0 \backslash \overline{L}_0(T)|.
$$
By \cite[Theorem 6.2.1]{KRY} we have, for any $T$ appearing in the  sum,
$$
\widehat{\deg}\ \widehat{\mathcal{Z}}(T,\mathbf{v}) =  e_p(T)\cdot | \overline{\Gamma}_0 \backslash \overline{L}_0(T)|\cdot \log(p)
$$
which completes the proof.
\end{proof}

\begin{Thm}\label{Thm:arithmetic decomposition}
Suppose we are given  $\alpha\in\co_F$ and $v\in F\otimes_\Q\R$, both totally positive. If the extension $F(\sqrt{-\alpha})/\Q$ is not biquadratic then
$$
\widehat{\deg}_{\mathcal{M}_0} \widehat{\mathcal{Y}}(\alpha,v)    =  \sum_{ T\in\Sigma(\alpha) }  \widehat{\deg}\ \widehat{\mathcal{Z}}(T,\mathbf{v})
$$
where $\mathbf{v}$ is defined by (\ref{bold v}).
\end{Thm}

\begin{proof}
This is immediate from the decomposition (\ref{proper pullback decomposition}), Lemmas \ref{Lem:proper archimedean}, \ref{Lem:proper I}, and \ref{Lem:proper II}, and,  by definition of $\widehat{\mathcal{Z}}(T,\mathbf{v})$, the implication 
$$
|\mathrm{Diff}(T,B_0)|>1\implies \widehat{\mathcal{Z}}(T,\mathbf{v}) =0
$$
for each nonsingular $T\in\Sym_2(\Z)^\vee$.
\end{proof}


\subsection{A genus two Eisenstein series}
\label{SS:automorphic}


Let $J_n$ be the $2n\times 2n$ matrix
$$
J_n=\left(\begin{matrix}  & I_n \\ -I_n &  \end{matrix}\right).
$$
Recall our fixed basis $\{\varpi_1,\varpi_2\}$ for the $\Z$-module $\co_F$ and let $\{\varpi_1^\vee, \varpi_2^\vee\}$ be the dual basis relative to the trace form $(x,y)\mapsto \mathrm{Tr}_{F/\Q}(xy)$.
The $F$-vector space $V_F=F^2$ comes equipped with the standard $F$-valued symplectic form 
$\langle v,w\rangle_F= {}^t wJ_1  v$,  while the $\Q$-vector space $V_\Q$ underlying $V_F$ is  equipped with the symplectic form 
$$
\langle v,w\rangle_\Q=\mathrm{Tr}_{F/\Q}\langle v,w\rangle_F.
$$  
With respect to the $\Q$-basis 
$$
\left\{ \left[\begin{matrix} \varpi_1^\vee \\ 0\end{matrix}\right] ,
\left[\begin{matrix} \varpi_2^\vee \\ 0 \end{matrix}\right] ,
\left[\begin{matrix}  0 \\ \varpi_1 \end{matrix}\right] ,
\left[\begin{matrix} 0 \\ \varpi_2 \end{matrix}\right]  \right\}
$$ 
of $V_\Q$ the  symplectic form $\langle v,w\rangle_\Q$ is given by 
$\langle v,w\rangle_\Q= {}^t w J_2 v.$
In this way we obtain a homomorphism
$$
\mathrm{Sp}_1(F)\iso \mathrm{Sp}(V_F) \hookrightarrow \mathrm{Sp}(V_\Q)\iso \Sp_2(\Q)
$$
which is given explicitly by
$$
\left(\begin{matrix}
a & b\\ c& d
\end{matrix}\right)
\mapsto
\left(\begin{matrix}
R & \\ & ^t R^{-1}
\end{matrix}\right)
\left(\begin{matrix}
\underline{a} & \underline{b}\\ \underline{c}& \underline{d}
\end{matrix}\right)
\left(\begin{matrix}
R^{-1} & \\ & ^t R  
\end{matrix}\right)
$$
where $R$ is defined by (\ref{twisting matrix}) and 
$$
  \underline{x}= \left(\begin{matrix}x& \\ & x^\sigma  \end{matrix}\right)
$$
for $x\in F$.   If $\Gamma_0(M)\subset \mathrm{Sp}_2(\Z)$ is the usual  subgroup of matrices congruent modulo $M$ to $\left(\begin{matrix}  * & * \\ 0 & * \end{matrix}\right)$ then
$$
\mathrm{Sp}_1(F)\cap\Gamma_0(M)=\left\{
\left( \begin{matrix} a & b \\ c & d \end{matrix} \right)  \in\mathrm{Sp}_1(F) \Big| \ a,d\in\co_F,  b\in \mathfrak{D}^{-1}, c\in M\mathfrak{D}
\right\} .
$$

The group $\mathrm{Sp}_1(F)$ acts in the usual way on the product of two upper half planes $\mathfrak{h}_1\times\mathfrak{h}_1$ via the embedding 
$$
\mathrm{Sp}_1(F)\hookrightarrow \SL_2(\R)\times\SL_2(\R)
$$ 
defined by $A\mapsto (A, A^\sigma)$.  Similarly  the rank two symplectic group $\Sp_2(\Q)$ acts in the habitual  manner on the Siegel half-space $\mathfrak{h}_2$ of genus two.    If we let $\mathrm{Sp}_1(F)$ act on $\mathfrak{h}_2$ through the above homomorphism  $\mathrm{Sp}_1(F)\map{}\mathrm{Sp}_2(\Q)$ then the embedding of complex manifolds 
\begin{equation}\label{siegel embedding}
\mathfrak{h}_1\times\mathfrak{h}_1 \hookrightarrow \mathfrak{h}_2
\hspace{1cm}
(\tau_1,\tau_2)  \mapsto R\ \left(\begin{matrix} \tau_1 & \\ & \tau_2 \end{matrix}\right)\ ^t R
\end{equation}
 respects the actions  of $\mathrm{Sp}_1(F)$.

To the quaternionic order  $\co_{B_0}$, Kudla \cite{kudla97} attaches a  Siegel Eisenstein series of weight $3/2$
$$
\mathcal{E}_2(\tau,s,B_0)=\sum_{T\in\Sym_2(\Q)} \mathcal{E}_{2,T}(\tau,s,B_0)
$$  
where $\tau=\mathbf{u}+i \mathbf{v} \in\mathfrak{h}_2$ and $s\in\C$.  See \cite[(5.1.44)]{KRY} for the precise definition.  The Eisenstein series satisfies a function equation which forces $\mathcal{E}_{2,T}(\tau,0,B_0)=0$ for every $T$, and one of the main results of Kudla-Rapoport-Yang \cite[p.12 Theorem B]{KRY} is the equality 
$$
\frac{d}{ds}\mathcal{E}_{2,T}(\tau,s,B_0) \big|_{s=0}= \widehat{\deg} \ \widehat{\mathcal{Z}}(T,\mathbf{v})\cdot q^T
$$
for every $T\in \Sym_2(\Q)$, in which $q^T=e^{2\pi i\cdot\mathrm{Tr}(T\tau)}$ and 
$$
\widehat{\mathcal{Z}}(T,\mathbf{v})\in\widehat{\chow}^2(\mathcal{M}_0)
$$
is the arithmetic cycle class appearing in \S \ref{SS:construction} and defined in \cite[Chapter 6]{KRY}.  When $T\not\in\Sym_2(\Z)^\vee$ both sides of the stated equality are $0$.  In particular if we denote by  $\widehat{\phi}_2(\tau)$  Kudla's $q$-expansion
$$
\widehat{\phi}_2(\tau)=\sum_{T\in \Sym_2(\Z)^\vee}
\widehat{\deg}\ \widehat{\mathcal{Z}}(T,\mathbf{v})\cdot q^T
$$
then $\mathcal{E}_2'(\tau,0,B_0)=\widehat{\phi}_2(\tau)$ is a (non-holomorphic) Siegel modular form of weight $3/2$ for the congruence subgroup $\Gamma_0(4m_o)\subset \mathrm{Sp}_2(\Z)$ where
$$
m_o=  \mathrm{disc}(B_0) \cdot \left\{\begin{array}{ll}
1  & \mathrm{if\ } \mathrm{disc}(B_0) \mathrm{\ is\ odd} \\
1/2 & \mathrm{if\ } \mathrm{disc}(B_0) \mathrm{\ is \ even} .
 \end{array}\right.
$$
  It follows that the pullback of $\widehat{\phi}_2$ to $\mathfrak{h}_1\times\mathfrak{h}_1$ is a Hilbert  modular form of parallel weight $3/2$ for the congruence subgroup
\begin{equation}\label{congruence subgroup}
\left\{ \left(\begin{matrix}  a & b \\ c & d   \end{matrix}\right) \in \SL_2(F) \mid a,d\in\co_F, b\in\mathfrak{D}^{-1}, c\in 4 m_o \mathfrak{D} \right\}.
\end{equation}

\begin{Lem}\label{Lem:Siegel pullback}
 The pullback of $\widehat{\phi}_2$ to $\mathfrak{h}_1\times\mathfrak{h}_1$
has  Fourier expansion
$$
\widehat{\phi}_2(\tau_1,\tau_2)=
\sum_{\alpha\in \co_F}  \Big(
\sum_{T\in\Sigma(\alpha)}  \widehat{\mathrm{deg}}\  \widehat{\mathcal{Z}}\left( T, \mathbf{v} \right) \Big)
\cdot q^\alpha
$$
where we have set  $q^\alpha= e^{2\pi i \tau_1 \alpha} \cdot e^{2\pi i \tau_2 \alpha^\sigma}$.  The positive definite matrix  $\mathbf{v}$ defined by (\ref{bold v})  with $v_1$ and $v_2$ the imaginary parts of $\tau_1$ and $\tau_2$ respectively, is the imaginary part of the image of $(\tau_1,\tau_2)$ in $\mathfrak{h}_2$.  
 \end{Lem}

\begin{proof}
If we  set $\tau=R\left(\begin{matrix} \tau_1 & \\ & \tau_2 \end{matrix}\right) \hspace{-1pt} ^tR$ then for every $\alpha\in\co_F$ and  $T\in\Sigma(\alpha)$ we have
$$
\mathrm{Tr}(T\tau)= \alpha \tau_1+ \alpha^\sigma \tau_2.
$$
It follows now from $\Sym_2(\Z)^\vee=\bigcup_{\alpha\in\co_F}\Sigma(\alpha)$ that
$$
\widehat{\phi}_2(\tau_1,\tau_2) = \sum_{\alpha\in \co_F} \sum_{T\in\Sigma(\alpha)}   \widehat{\mathrm{deg}}\  \widehat{\mathcal{Z}}\left( T, \mathbf{v} \right)   \cdot e^{2\pi i \alpha \tau_1  } e^{2\pi i \alpha^\sigma \tau_2  } ,
$$
proving the claim.
\end{proof}

\begin{Cor}
Suppose $\alpha\in\co_F$ is totally positive and that $F(\sqrt{-\alpha})/\Q$ is not biquadratic.  Then the $\alpha^\mathrm{th}$ Fourier coefficient of $\widehat{\phi}_2(\tau_1,\tau_2)$ is 
$$
\widehat{\deg}_{\mathcal{M}_0} \widehat{\mathcal{Y}}(\alpha, v) = \sum_{T\in\Sigma(\alpha)}  \widehat{\mathrm{deg}}\  \widehat{\mathcal{Z}}\left( T, \mathbf{v} \right) 
$$
where $v=(v_1,v_2)\in \R\times \R$ is the imaginary part of $(\tau_1,\tau_2)$ and $\mathbf{v}$ is defined by (\ref{bold v}).
\end{Cor}

\begin{proof}
This is immediate from Lemma \ref{Lem:Siegel pullback} and Theorem \ref{Thm:arithmetic decomposition}.
\end{proof}

\bibliographystyle{plain}
\def\cprime{$'$}

\end{document}